\documentclass[a4paper,11pt]{article}
\usepackage{amsmath,amssymb,amsthm} 
\usepackage{booktabs,caption}
\usepackage{amsthm}
\usepackage{latexsym}
\usepackage{graphicx}
\usepackage{a4wide}
\usepackage{pdfpages}
\usepackage[labelfont=bf]{caption}
\usepackage[utf8x]{inputenc}
\usepackage[english]{babel}
\usepackage{color}
\usepackage{booktabs,threeparttable}
\usepackage{bm}
\usepackage{dblfloatfix}
\usepackage{multicol}

\usepackage[toc,page]{appendix}

\usepackage[bottom]{footmisc}
\usepackage{hyperref}
\usepackage[figurename=Fig.]{caption}

\usepackage[authoryear]{natbib}
\bibpunct{(}{)}{;}{a}{,}{,}
\usepackage{url}
\newtheorem{remark}{Remark}[section]

\theoremstyle{plain}
\newtheorem{theorem}{Theorem}[section]
\newtheorem{corollary}{Corollary}[section]

\newtheorem{lemma}{Lemma}[section]

\renewcommand{\baselinestretch}{1.5}
\makeatletter
\def\@tempa#1{\@xp\@tempb\meaning#1\@nil#1}
\def\@tempb#1>#2#3 #4\@nil#5{%
  \@xp\ifx\csname#3\endcsname\mathaccent
    \@tempc#4?"7777\@nil#5%
  \else
    \PackageWarningNoLine{amsmath}{%
      Unable to redefine math accent \string#5}%
  \fi
}
\def\@tempc#1"#2#3#4#5#6\@nil#7{%
  \chardef\@tempd="#3\relax\set@mathaccent\@tempd{#7}{#2}{#4#5}}

\@tempa\widehat
\makeatother

\allowdisplaybreaks

\makeatletter
\newcommand{\fancysum}{%
  \mathop{%
    \vphantom{\sum}%
    \mathpalette\fancy@sum\relax
  }\slimits@
}

\newcommand{\fancy@sum}[2]{%
  \sbox\z@{$#1\sum$}%
  \dimen@=\ht\z@
  \advance\dimen@\dp\z@
  \vcenter{\hbox{%
    \resizebox{!}{\ifx#1\displaystyle.9\fi\dimen@}{%
      $\m@th\mathcal{S}$%
    }%
  }}%
}

\usepackage{enumitem}
\setlist[enumerate]{itemsep=-0.5mm}
\begin{document}

\setcounter {page}{1} 
\setcounter {subsection}{1} 

\title{\bf Estimating
Generalized Additive Conditional Quantiles for Absolutely Regular
Processes}

\author{Yebin Cheng$^1$ and Jan G. De Gooijer$^{2,}$\footnote{Corresponding author: j.degooijer@contact.uva.nl  }}
\date{}
\maketitle
\vspace{-0.5cm}

\renewcommand{\baselinestretch}{1}
\small{\centerline{$^1$ Glorious Sun School of Business and Management,
Donghua University, Shanghai, China}

\vspace{0.3cm}

\centerline{$^2$ Amsterdam School of Economics, University of Amsterdam, the Netherlands} 
\vspace{0.8cm}

\begin{abstract}
\noindent 
We propose a nonparametric method for  estimating  the conditional quantile
function that admits  a  generalized additive specification with an unknown link function. This model nests single-index, additive, and multiplicative quantile regression models. Based on  a full local
linear polynomial  expansion, we first obtain the asymptotic
representation for the proposed quantile estimator for each additive
component. Then, the link function is  estimated by noting that it corresponds to the conditional quantile function of a response variable  given the sum of  all additive components. The  observations are supposed to be a sample from a strictly stationary and absolutely regular process. We provide results on (uniform) consistency rates, second order asymptotic  expansions and point wise asymptotic normality of each  proposed estimator. 
\end{abstract}

\bigskip


\renewcommand{\theequation}{1.\arabic{equation}}

\section{Introduction}
Suppose that $Y$ is a response variable of interest which depends on a vector of random covariates $X=(X_{1},\ldots,X_{d})^{\mbox{\tiny{T}}}\in \mathbb{R}^{d}$, $d\geq 2$. We are interested in estimating the $\alpha$th $(0<\alpha<1)$ conditional quantile $q(x)$ of $Y$ given $X$.
For the $i$th subject, we assume that  the sample $\{(X_{i},Y_{i})\}^{n}_{i=1}$
 is
defined on a probability space $(\Omega ,\mathfrak{F},\mathbb{P})$, be a strictly stationary and absolutely regular stochastic process
from the population $(X,Y)$. It is well known that the $\alpha$th
conditional quantile of $Y$ given
$X=x=(x_{1},\ldots,x_{d})^{\mbox{\tiny{T}}}$ is defined as the value
$q(x)$ such that
$q(x) =\inf \{t\colon F(t|x)\geq \alpha \}$,
where  $F(t|x)$ is the conditional distribution of
$Y$ given $X=x$. Equivalently, $q(x) =\arg \inf_{a}\mathbb{E}
\{\rho _{\alpha }(Y-a)|X=x\}$, where $\rho_{\alpha }(y)
=|y|+(2\alpha -1)y$ for any real $y$.

There is an extensive literature
dealing with the estimation of  $q(x)$ when the
functional relationship between  $Y$ and  $X$ is unknown. In particular,  there have been many proposals using  the additive quantile regression model $q(x)=C+%
\sum_{u=1}^{d}q_{u}(x_{u})$, where $C$ is a constant, $q_{u}(x_{u})$ $(u=1,2,\ldots ,d)$
are each additive components, and $x_{u}$ is the $u$th component of $x$. For instance, \citet{cheng11},  \citet{degooijer03}, and \citet{yu04}
use this model to obtain estimates of additive
conditional quantiles in a time series setting by nonparametric
methods, and \citet{horowitz05} and \citet{noh14} for independent and identically distributed (i.i.d.)\!\! data by splines.
In this paper,  we  consider estimating conditional quantiles in a more
generalized setting. That is, we assume that the generalized additive model is of the
form
\begin{equation}
q(x) =G\big(\sum_{u=1}^{d}q_{u}( x_{u})
\big) , \label{model}
\end{equation}
where $G(\cdot)$ is an \textit{unknown} link function. It
encompasses single-index, additive models and  generalized additive models
with \textit{known} links as special cases. It also contains
multiplicative models of the
form $q(x)=\widetilde{G}\big(\prod_{u=1}^{d}\widetilde{q}_{u}(x_{u})\big)$, where $\widetilde{G}%
(\cdot)$ and $\widetilde{q}_{u}(\cdot)$ are unknown functions.

Building on the insight of \citet{horowitz01} for the generalized additive conditional mean regression model, 
 the main idea of this paper  to estimate the components $q_{u}(x_{u})$, $u=1,2,\ldots,d$, is to write them as  functionals of the distribution of the data, independent of $G(\cdot)$. Then, we estimate the unknown link function $G(\cdot)$ by noting that it corresponds to the conditional quantile function of $Y$ given $q_{0}(X)$, where $q_{0}(x)=\sum^{d}_{u=1}q_{u}(x_{u})$. Also, we present theorems giving conditions under which the estimators of $q_{u}(\cdot)$ and $G(\cdot)$ are consistent and asymptotically normally distributed. \citet[Ch.\ 5]{cheng07} uses this latter result to formulate a test statistic for additivity of conditional quantile functions, under the assumption that the sample $\{(X_{i},Y_{i})\}^{n}_{i=1}$ is still an absolutely regular process.

The rest of the paper is organized as follows. In Section
\ref{sec:meth}, we provide a description of the estimator of the additive
component and the estimator of the unknown link function, 
 Section \ref{sec:convergence} gives the asymptotic representation of the
nonparametric estimator of $q_{u}(\cdot)$ for $1\leq u\leq d$, in
which the second order asymptotic representations\ are included,
either. From this, it can be seen that the convergence rate of each
estimator for each additive component is at the rate 
$(nh)^{-\frac{1}{2}}$, where $h$ is the bandwidth. Compared to
the rate of convergence $(nh^{d})^{-\frac{1}{2}}$ in
the usual nonparametric setting, this kind of rate tends to $0$ more
quickly.
In order to get the asymptotic representation for the
estimator of the unknown link function $G(\cdot)$, we address in
Section \ref{sec:add} the uniform convergence for the additive components. Then, in Section \ref{sec:link}, we  discuss the asymptotic
representation of the estimation of the unknown link function
$G(\cdot)$ in \eqref{model} and subsequently we discuss
the corresponding asymptotic normality. Concluding comments are presented in Section \ref{sec:concl}. The proofs of  theorems and lemmas are provided in four supplementary appendices. This document also  contains some useful
lemmas and the proof of the Bahadur type linear representation for
the local linear estimator of $q_{u}(\cdot )$.

Unless otherwise stated the symbol $\overset{d}{\rightarrow}$ signifies convergence in distribution. The superscript $\rm\ T$ denotes matrix or vector transposition. For any $a<b$, we use the notation $\mathfrak{M}_{a}^{b}$ to denote the sigma algebra generated by
$(Z_{a},\ldots,Z_{b})$, where $Z_{i}=$
$(X_{i},Y_{i})$. Given this notation, a
process is called absolutely regular $(\beta$-mixing), if as $\tau \rightarrow \infty $, $%
\beta _{\tau }=\sup_{s\in N}\mathbb{E}\{ \sup_{A\in
\mathfrak{M}_{s+\tau }^{\infty }}\{\mathbb{P}(
A|\mathfrak{M}_{-\infty }^{s}) -\mathbb{P}(A) \}
\} \rightarrow 0$; see, e.g.,
\citet{yoshihara78} and \citet{arcones98}.

\setcounter{equation}{0}
\renewcommand{\theequation}{2.\arabic{equation}}

\section{Methodology}\label{sec:meth}
\subsection{Preamble}\label{sec:preamble}
Following \citet{horowitz01}, we  express $q_{k}(\cdot)$, 
$1\leq k\leq d$, as functionals of the population distribution of
$(X,Y)$. We first consider the case $d\geq 3$. The case $d=2$ is briefly discussed at the end of this section. To
simplify notation, for $x\in \mathbb{R}^{d}$ and any fixed $2\leq u\leq d$, $x_{\bar{u}}$ is a $(d-2)$-dimensional vector consisting
of all  components of $x$ except $x_{1}$ and $x_{u}$, and similar
for the notations $X_{\bar{u}}$ and $X_{j,\bar{u}}$, etc. For $1\leq j\leq n$, let $\widetilde{X}_{j}=\widetilde{X}_{j}^{t_{1,u}}$ be a $d$-dimensional vector
with the first component $t_{1}$, the $u$th component $t_{u}$ and
the other components $X_{j,\bar{u}}$. Also, let
$\mathcal{X}=\Pi _{k=1}^{d}[a_{k},b_{k}]\subseteq \mathbb{R}^{d}$ be
a subset for the support set of $X$ and $\mathcal{X}$ does not contain its
boundary. For identification of each
component function and given some fixed points $a_{k}<x_{k,0}<b_{k}$ for each  $k=1,2,\ldots,d$, we assume that $q_{k}(x_{k,0})=0$ and
\begin{align}
\int_{a_{1}}^{b_{1}}\frac{w_{1}(x_{1}) }{q_{1}^{\prime
}(x_{1})}{\rm d}x_{1}=1,  \label{a1}
\end{align}
where the weight function $w_{1}(\cdot)$ defined on 
$[a_{1},b_{1}]$ is non-negative and integrates to one.

In order to make
 \eqref{a1} hold, it is required that $q_{1}^{\prime}(x_{1})\neq 0$ for
$a_{1}\leq x_{1}\leq b_{1}$. Let
$\mathcal{X}_{1,u}=[a_{1},b_{1}]\times \lbrack a_{u},b_{u}]$,
$\mathcal{X}_{\bar{u}}=\prod _{2\leq k\neq u\leq d}[a_{k},b_{k}]$ and
$\partial_{k}q(x) $ be the first order partial
derivative of $q(x)$ with respect to the $k$th
component $x_{k}$ of $x$. Define
\begin{align}
D_{u}(x_{1},x_{u}) =\int_{\mathcal{X}_{\bar{u}}}
[\partial
_{u}q(x)] p_{\bar{u}}(x_{\bar{u}}) {\rm d}x_{\bar{u}%
}=\mathbb{E}[ \partial _{u}q(x_{1},x_{u},X_{\bar{u}}) \mathbb{I}
(X_{\bar{u}}\in \mathcal{X}_{\bar{u}})]
\label{a2}
\end{align}
and
\begin{equation}
D_{1,u}(x_{1},x_{u}) =\int_{\mathcal{X}_{\bar{u}}}[
\partial _{1}q(x)] p_{\bar{u}}(x_{\bar{u}})
{\rm d}x_{\bar{u}}=\mathbb{E}[ \partial_{1}q(x_{1},x_{u},X_{\bar{u}%
})\mathbb{I}(X_{\bar{u}}\in \mathcal{\mathcal{X}}_{\bar{u}})
] , \label{a3}
\end{equation}%
 where $p_{\bar{u}}(x_{\bar{u}})$ is the
marginal density function of $X_{\bar{u}}$, and $\mathbb{I}(\cdot )$
denotes the indicator function.

From \eqref{model} we have
$\partial _{j}q(x) =G^{\prime }\big(
\sum_{k=1}^{d}q_{k}(x_{k})\big) q_{j}^{\prime }(
x_{j})$ for $j=1,\ldots,d$. Then, from \eqref{a2} and
\eqref{a3}, we see that
\begin{align*}
D_{u}( x_{1},x_{u}) =q_{u}^{\prime }(x_{u}) \int_{%
\mathcal{X}_{\bar{u}}}G^{\prime }\big(\sum_{k=1}^{d}q_{k}(
x_{k})\big)p_{\bar{u}}(x_{\bar{u}})
{\rm d}x_{\bar{u}}
\end{align*}
and
\begin{align*}
D_{1,u}(x_{1},x_{u})=q_{1}^{\prime }(x_{1}) \int_{%
\mathcal{X}_{\bar{u}}}G^{\prime }\big(\sum_{k=1}^{d}q_{k}(
x_{k})\big) p_{\bar{u}}(x_{\bar{u}})
{\rm d}x_{\bar{u}}.
\end{align*}
Assume that $D_{1,u}(x_{1},x_{u}) \neq 0$ for any 
$(x_{1},x_{u})\in \mathcal{X}_{1,u}$. This implies that $q_{1}^{\prime
}(x_{1})\neq 0$ for $x_{1}\in \lbrack a_{1},b_{1}]$. Then it follows from \eqref{model} that
\begin{align*}
\frac{q_{u}^{\prime }( x_{u}) }{q_{1}^{\prime}(
x_{1}) }=\frac{D_{u}( x_{1},x_{u}) }{D_{1,u}(
x_{1},x_{u}) }.
\end{align*}%
Furthermore, by integrating both sides of the above expression, for any $%
x_{u}\in \lbrack a_{u},b_{u}]$, we get
\begin{align*}
\int_{x_{u,0}}^{x_{u}}\int \frac{q_{u}^{\prime }(t_{u}) }{%
q_{1}^{\prime }( t_{1}) }w_{1}(t_{1})
{\rm d}t_{1}{\rm d}t_{u}=\int_{x_{u,0}}^{x_{u}}\int \frac{
D_{u}(t_{1},t_{u})}{D_{1,u}(t_{1},t_{u}) }%
w_{1}(
t_{1}){\rm d}t_{u}{\rm d}t_{1},
\end{align*}
i.e.,
\begin{align}
q_{u}(x_{u}) =\int_{x_{u,0}}^{x_{u}}\int
\frac{D_{u}(t_{1},t_{u})
}{D_{1,u}(t_{1},t_{u})}w_{1}(
t_{1}){\rm d}t_{u}{\rm d}t_{1}.  \label{qu}
\end{align}%
Observe that when $x_{u}<x_{u,0}$, \eqref{qu} still holds
because exchanging the location between $x_{u}$ and $x_{u,0}$
requires to add a minus notation before the integral simultaneously.

Analogously, for another non-negative weight function $w_{2}(
t_{2})$ defined on $[a_{2},b_{2}]$, which integrates to one,
it follows that
\begin{align*}
c^{-1}q_{1}(x_{1}) =\int_{x_{1,0}}^{x_{1}}\int \frac{%
D_{1,2}(t_{1},t_{2})}{D_{2}(t_{1},t_{2}) }%
w_{2}(t_{2}) {\rm d}t_{2}{\rm d}t_{1},
\end{align*}
with $c^{-1}=\int \big(w_{2}(t_{2})/q_{2}^{\prime}
(t_{2})\big){\rm d}t_{2}$. From this, we obtain that
\begin{align}
\int w_{1}(t_{1})\Big[ \int \frac{D_{1,2}(
t_{1},t_{2}) }{D_{2}(t_{1},t_{2})}w_{2}(
t_{2}){\rm d}t_{2}\Big] ^{-1}{\rm d}t_{1}=\int \frac{w_{1}(t_{1}) }{%
q_{1}^{\prime }(t_{1})}{\rm d}t_{1}\Big[ \int
\frac{w_{2}(t_{2})}{q_{2}^{\prime }(t_{2})
}{\rm d}t_{2}\Big]^{-1}=c. \label{q2}
\end{align}
Thus, we have
\begin{equation}
q_{1}(x_{1}) =c\int_{x_{1,0}}^{x_{1}}\int
\frac{D_{1,2}(t_{1},t_{2})}{D_{2}(
t_{1},t_{2})}w_{2}(t_{2}){\rm d}t_{2}{\rm d}t_{1}.
\label{q1}
\end{equation}

Assume $q^{\prime}_{1}(\cdot)\neq 0$ and $q_{2}^{\prime}(\cdot)\neq 0$.
Then for the case $d=2$, going along the same lines as  for the case
 $d\geq 3$, we can still establish 
\eqref{qu}, \eqref{q2} and \eqref{q1} if we set $D_{2}(x_{1},x_{2}) =
\partial _{2}q(x_{1},x_{2})$ and $D_{1,2}(x_{1},x_{2}) = \partial_{1}q(x_1,x_2)$.

\subsection{Estimating the additive components $q_{u}(\cdot)$}
Based on relationships \eqref{qu}--\eqref{q1}, we construct
the desired estimators by local polynomial fitting. To this end, we assume that
$q(x) $ is partially differentiable up to the order
$p+1$, which implies there are $d^{p-1}$ parameters to estimate.  By Taylor's expansion, for any $w$ close to $x$, $q(w)$ can be expressed as
\begin{equation*}
q(w) =\sum_{\lambda \in \Lambda }\beta ( \lambda
,x) h^{-\vert \lambda \vert }(w-x)
^{\lambda }+R_{x}(w) =\beta_{x}^{\mbox{\tiny{T}}} A\big(
\frac{w-x}{h}\big) +R_{x}(w),
\end{equation*}%
where $\Lambda =\{(\lambda _{1},\ldots ,\lambda _{d})\}$, $\lambda _{i}$ are non-negative integers and $\sum^{d}_{i=1}\lambda _{i}\leq
p-1\}$, $|\lambda |=\sum_{i=1}^{d}\lambda _{i}$, $x^{\lambda }=\prod^{d}_{i=1}
x_{i}^{\lambda _{i}}$. Here, $h$ is a bandwidth specified  below, and the two
column vectors $A(\frac{w-x}{h})$ and $\beta_{x}$ are constructed from the
elements $h^{-\vert \lambda \vert }(w-x)
^{\lambda }$ and $\beta(\lambda,x) $ respectively, which are arranged in natural order with respect to $\lambda \in
\Lambda $. Note that $\beta (\lambda,x)$ is related to $h$ and is
of order $h^{\vert \lambda \vert }$.

  To estimate
$\beta_{x}$, using the local polynomial method, the corresponding
estimator $\widehat{\beta}_{x}$ can be obtained through minimizing the
following objective function
\begin{equation*}
\widehat{\beta}_{x}=\arg \min_{\beta}\sum_{i=1}^{n}w_{n,i}\rho_{\alpha
}\Big(Y_{i}- \beta^{\mbox{\tiny{T}}} A \Big(\frac{X_{i}-x}{h}\Big)\Big),
\end{equation*}
where the weight function $w_{n,i}$ is equal to $K\big(\frac{x-X_{i}}{h}%
\big) /\sum_{j=1}^{n}K\big(\frac{x-X_{j}}{h}\big)$ with the
usual kernel function $K(\cdot)$. Then, the estimator of
 $q(x)$ and its partial derivatives can
be derived explicitly from $\widehat{\beta}_{x}$.

Based on this method, we can plug all the relevant estimators into \eqref{qu} and obtain the estimator $\widehat{q}_{k}( x_{k}) $ for $%
q_{k}(x_{k}) $, $1\leq k\leq d$. When $d\geq 3$,
$\widehat{q}_{u}(x_{u}) $, $2\leq u \leq d $, we have the
representation
\begin{equation}
\widehat{q}_{u}(x_{u}) =\int_{x_{u,0}}^{x_{u}}\int \frac{\widehat{D}%
_{u}(t_{1},t_{u})}{\widehat{D}_{1,u}(t_{1},t_{u}) }%
w_{1}(t_{1}){\rm d}t_{u}{\rm dt}_{1},  \label{equ}
\end{equation}%
where $\widehat{D}_{u}(t_{1},t_{u}) =\frac{1}{nh}%
\sum_{i=1}^{n}e_{u}^{\mbox{\tiny{T}}}\widehat{\beta}_{\widetilde{X}_{i}}\mathbb{I}( X_{i,\bar{u}%
}\in \mathcal{X}_{\bar{u}})$ is the estimator of $D_{u}(
t_{1},t_{u}) $,
$\frac{e_{u}^{\mbox{\tiny{T}}}\widehat{\beta}_{\widetilde{X}_{i}}}{h}$ is the estimator of
$\partial_{u}q(\widetilde{X}_{i})$ and $e_{u}$, which has
the same dimension as $\widehat{\beta}_{x}$, denotes a unit vector such
that its $u$th component is equal to $1$ and all other
components are equal to $0$. Here, it should be noted that we adopt
the leave-one-out
rule to estimate $\widehat{\beta}_{\widetilde{X}_{i}}$. Similarly, $\widehat{D}%
_{1,u}(t_{1},t_{u}) =\frac{1}{nh}\sum_{i=1}^{n}e_{1}^{\mbox{\tiny{T}}}\widehat{%
\beta}_{\widetilde{X}_{i}}\mathbb{I}(X_{i,\bar{u}}\in \mathcal{X}_{\bar{u}%
})$. Analogously,  for $d\geq 3$, the estimator of
${q}_{1}(x_{1})$ is given by 
\begin{equation}
\widehat{q}_{1}(x_{1}) =\widehat{c}\int_{x_{1,0}}^{x_{1}}\int \frac{\widehat{%
D}_{1,2}(t_{1},t_{2}) }{\widehat{D}_{2}(t_{1},t_{2})}%
w_{2}(t_{2}) {\rm d} t_{2}{\rm d} t_{1},  \label{eq1}
\end{equation}%
where
\begin{equation}
\widehat{c}=\int w_{1}(t_{1}) \Big[ \int \frac{\widehat{D}%
_{1,2}(t_{1},t_{2}) }{\widehat{D}_{2}( t_{1},t_{2}) }%
w_{2}(t_{2}) {\rm d} t_{2}\Big] ^{-1}{\rm d} t_{1}.  \label{eqc}
\end{equation}

When $d=2$ in \eqref{equ})--\eqref{eqc},
we take the two estimators
$\widehat{D}_{2}(t_{1},t_{2})$ and $\widehat{D}
_{1,2}(t_{1},t_{2})$ as
$e_{2}^{\mbox{\tiny{T}}}\widehat{\beta}_{(t_1,t_2)}$ and
$e_{1}^{\mbox{\tiny{T}}}\widehat{\beta}_{(t_1,t_2)}$. Then,
$\widehat{q}_{2}(x_{2})$
and $\widehat{q}_{1}(x_{1})$ are of the same form as 
\eqref{equ} and \eqref{eq1}, respectively.

\subsection{Estimating the link function $G(\cdot)$}
To estimate $G(\cdot)$, we define  $q_{0}(x)\! =\!\sum_{k=1}^{d}q_{k}(x_{k})$,
$\mathcal{V}\!=\!\{q_{0}(x)\vert x\in \mathcal{X}%
\} $ and $\widehat{q}_{0}(X_{i}) =\sum_{k=1}^{d}\widehat{q}%
_{k}( X_{i,k}) $, where $X_{i,k}$ is the $k$th component of $X_i$%
. Then, $\mathcal{V}$ is also a compact set since the continuity of $%
q_{0}(x)$. From the definition of the conditional
quantile, the property on the conditional expectation and the Borel
measurability on the function $q_0(\cdot)$, we note that
\begin{align*}
1-\alpha&=\mathbb{E}\big(\mathbb{I}\{Y\leq
G(q_0(X))\}|X\big)=\mathbb{E}\big [\mathbb{E} (\mathbb{I}
\{Y\leq
G(q_0(X))\}|X)|q_0(X)\big] \\
&= \mathbb{E}\big(\mathbb{I}\{Y\leq
G(q_0(X))\}|q_0(X)\big).
\end{align*}
Thus the unknown function $G(v)$ is also the conditional quantile
function of $Y$ given that $q_0(X)=v$. For any
fixed $v\in \mathcal{V}$, the estimator of 
$G(v)$ is given by the following empirical function
\begin{equation*}
\widehat{G}_{n}(v) =\inf \big \{y|\widehat{F}_{n}(y|\
q_{0}(x) =v)\geq 1-\alpha \big\},
\end{equation*}
where 
\begin{equation}
\widehat{F}_{n}(y|\ q_{0}(x)
=v)=\sum_{i=1}^{n}\frac{K_{G}\left( \frac{v-\widehat{q}_{0}\left(
X_{i}\right) }{h_{G}}\right) \mathbb{I}(
Y_{i}\leq y,X_{i}\in \mathcal{X}) }{\sum_{j=1}^{n}K_{G}\left( \frac{v-%
\widehat{q}_{0}\left( X_{j}\right) }{h_{G}}\right) \mathbb{I}(
X_{j}\in \mathcal{X}) },  \label{G}
\end{equation}%
with $K_{G}(\cdot )$ a kernel function of a scalar argument (in the sense of nonparametric density estimation), and
$h_{G}$ a bandwidth tending to $0$ as $n\rightarrow \infty$.

\setcounter{equation}{0}
\renewcommand{\theequation}{3.\arabic{equation}}

\section{Convergence of the additive components}\label{sec:convergence}
For convenience of presentation, we introduce some notation. For
$2\leq u\leq d$ and $t_{1,u}=(t_{1},t_{u}) \in
\mathcal{X}_{1,u}$, let $\varepsilon _{i}=Y_{i}-q(X_{i})$, $Q=\int
A(z)A^{\mbox{\tiny{T}}}(z) K(z){\rm d}z$, $Q^{\ast }=\int
K(x)A(x)A^{\mbox{\tiny{T}}}(x) x^{\mbox{\tiny{T}}}{\rm d}x$, $\beta_{j,t_{1,u}}=\beta_{\widetilde{X}_{j}}$, $%
K_{ij,t_{1,u}}=K\left( \frac{X_{i}-\widetilde{X}_{j}}{h}\right) $, $%
A_{ij,t_{1,u}}=A\left( \frac{X_{i}-\widetilde{X}_{j}}{h}\right) $, $%
r_{ij,t_{1,u}}=q(X_{i})-\beta_{j,t_{1,u}}^{\mbox{\tiny{T}}}A_{ij,t_{1,u}}$, $%
P_{ij,t_{1,u}}=(\beta -\beta_{j,t_{1,u}})^{\mbox{\tiny{T}}}A_{ij,t_{1,u}}$, $%
\Delta _{1,u}=\widehat{D}_{1,u}(t_{1},t_{u}) -D_{1,u}\left(
t_{1},t_{u}\right) $ and $\Delta _{u}=\widehat{D}_{u}(
t_{1},t_{u})
-D_{u}(t_{1},t_{u}) $. Also, for any real $y$ and $1\leq k\leq d$, we assume that
\begin{align*}
f_{k}(y)=
\int_{-1}^{y}\int A(t)K(t){\rm d}t_{\underline{k}}{\rm d}t_{k}\,\mathbb{I}(|y|\leq 1)
\end{align*}
 with  $t_{\underline{k}}$ a
$(d-1)$-dimensional vector constructed from $t$ by deleting the
$k$th argument $t_{k}$ of $t$. Let the set $A_{(u)}=\left[
x_{u,0}-h,x_{u}+h\right] \times \Pi _{1\leq l\neq u\leq d}\left[
a_{l}-h,b_{l}+h\right] $ and ${\varepsilon >0}$ be an any
sufficiently small constant. Also, let 
$S_{t_{1}}$ be the support set of the distribution of $X_{j}^{t_{1},u}$ and $S_{t_{2}}$ be the support set of the distribution of $X_{j}^{t_{2},u}$, $t_{1}\neq t_{2}$. In addition, for ease of presentation, we introduce the notation $\mathbb{E}_{j}$ defined as $\mathbb{E}_{j0}g(\xi_{i},\xi_{0})=g_{1}(\xi_{j0})$.

The asymptotic properties of $\widehat{q}_{n}(\cdot)$ and $\widehat{G}_{n}(\cdot)$ are established under the  following regularity conditions: 
\begin{enumerate}
\item[(B1)] The density function $p(\cdot )$ of $X$ is bounded and continuous
on its support set. 

\item[(B2)] $K(\cdot)$ has bounded and continuous partial
derivatives of order $1$ and has the support set as unit sphere.
 Moreover, $\int tK(t) {\rm d}t=0$.

\item[(B3)] The bandwidth $h=O( n^{-\kappa })$ satisfies that $%
\frac{1-\frac{1}{r}}{4p+d-\frac{d}{r^{2}}}<\kappa <\frac{1-\frac{1}{r}}{d+p+%
\frac{d+3p}{r}}$. And the mixing coefficient $\beta _{k}=O(
k^{-r})$ with
\begin{equation*}
r\geq \max \left\{ d+\frac{1}{2}-\frac{d}{p}+\frac{17p+d-3}{2p(d-1-dp^{-1})}%
,d-7+\frac{2d^{2}-4-4d}{p}-\frac{22p+6}{dp},d,11\right\} .
\end{equation*}

\item[(B4)] Let $G(x,y)$ be the conditional distribution of $\varepsilon_{i}$
given that $X_{i}=x$. Its conditional density function $g(x,y)$ has
first
continuous derivative for $y$ in the neighbourhood of $0$ and $x\in \mathcal{X%
}$.

\item[(B5)] $g_{1}(x)=g(x,0)p(x)$ has  bounded second
derivatives and is bounded away from zero on $\mathcal{X}$.

\item[(B6)] The bandwidth satisfies $\frac{1}{2p+1}\leq \kappa \leq \frac{1-\frac{%
1}{r}}{3d+2+\frac{2d}{r}-\frac{d}{r^{2}}}$.

\item[(B7)] For any $1\leq l\leq n-1$, the density function $f(\cdot,\cdot )$
of $\left( X_{1},X_{1+l}\right) $ exists and is continuous and
bounded on its domain.

\item[(B8)] $D_{1,u}(t_{1},t_{u})$ has first continuous derivatives with
respect to $t_{u}$ for any $t_{1}$.

\item[(B9)] There are $m>0$ and compact intervals $\overline{S}_{1}\subset \mbox{int}(S_{t_{1}})$ and $\overline{S}_{2}\subset \mbox{int}(S_{t_{2}})$ such that $|g^{\prime}_{1}(t_{1})|\geq m$ for all $t_{1}\in\overline{S}_{1}$ and $g^{\prime}_{2}(t_{2})|\geq m$ for all $t_{2}\in\overline{S}_{2}$.

\item[(B10)] $w_{1}(t_{1})=O( t_{1}-a_{1})$ as $t_{1}\rightarrow
a_{1}$, $w_{1}(t_{1})=O(t_{1}-b_{1}) $ as
$t_{1}\rightarrow b_{1} $, $w_{2}(t_{2})=O\left( t_{2}-a_{2}\right)
$ as $t_{2}\rightarrow a_{2}$ and $w_{1}(t_{2})=O(t_{2}-b_{2}) $ as $t_{2}\rightarrow b_{2}$.
 \end{enumerate}

\begin{lemma} 
\label{WW} Under conditions (B1)--(B9), it holds uniformly for $\left(
t_{1},t_{u}\right) \in \mathcal{X}_{1,u}$ with probability one that
\begin{equation}
\frac{1}{n^{2}h^{d+1}}\sum_{j=1}^{n}
Q_{jn,t_{1,u}}^{-1}W_{j,n}(t_{1,u})  =B_{1,u}h^{p}  \left( 1+o\left(
1\right) \right),\label{Toeplitz}
\end{equation}
where $Q_{jn,t_{1,u}}=\frac{1}{h^{d}}\mathbb{E}_{j}\big (
K_{ij,t_{1,u}}A_{ij,t_{1,u}}A_{ij,t_{1,u}}^{\mbox{\tiny{\rm T}}}g(X_{i},0)\big )$,
$B_{1,u}=\mathbb{E}\iint  B_{2,u}(\widetilde{X}_{j}){\rm d}t_1 {\rm d}t_u$,
\begin{equation*}
B_{2,u}(t)= e_{u}^{\mbox{\tiny{\rm T}}}Q^{-1}Q^*\frac{1}{p! g_1 (t)}\frac{\partial
g_1(t)}{\partial t}\int A(s)s^pK(s){\rm d}s\,q^{(p)}(t),
\end{equation*}
and
\begin{equation}
W_{j,n}(t_{1,u})
=\sum_{i=1}^{n}K_{ij,t_{1,u}}A_{ij,t_{1,u}} \left[\mathbb{I}(
\varepsilon _{i}\leq 0) -\mathbb{I}(\varepsilon _{i}\leq
r_{ij,t_{1,u}}) \right]\mathbb{I}(X_{j,\bar{u}}\in
\mathcal{X}_{\bar{u}}).  \label{w1}
\end{equation}
\end{lemma}

\begin{lemma}
\label{delta} Under conditions (B1)--(B9), it holds almost surely that
\begin{equation*}
\sup_{\left( t_{1},t_{u}\right) \in \mathcal{X}_{1,u}}\Delta
_{1,u}=O\left( \left( nh^{4+\frac{1+\varepsilon }{r}}\right)
^{-\frac{1}{2}}\right)\,\,
 and \,\,
\sup_{\left( t_{1},t_{u}\right) \in \mathcal{X}_{1,u}}\Delta
_{u}=O\left( \left( nh^{4+\frac{1+\varepsilon }{r}}\right)
^{-\frac{1}{2}}\right)
\end{equation*}
for $2\leq u\leq d$.
\end{lemma}

\begin{theorem}
\label{L1}Let conditions (B1)--(B10) hold.
\begin{enumerate}
\item[i)] For $2\leq u\leq d$ and $a_{u}\leq x_{u}\leq b_{u}$, we have the
following asymptotic representation%
\begin{align}
 \sqrt{nh}\big(\widehat{q}_{u}(
x_{u}) -q_{u}(x_{u})\big)
& =\sum_{i=1}^{n}\frac{\big( \left( 1-\alpha \right) -\mathbb{I}(
\varepsilon _{i}\leq 0\big))
w_{1}(X_{i,1}) p_{\bar{u}}(X_{i,\bar{u}})\,\mathbb{I}
(X_{i}\in A_{(u)}) }{\sqrt{nh}D_{1,u}(X_{i,1},X_{i,u}) g_{1}(X_{i})} \notag \\
&\left(e_{u}^{\mbox{\tiny{\rm T}}}-\frac{e_{1}^{\mbox{\tiny{\rm T}}}D_{u}(X_{i,1},X_{i,u}) }{%
D_{1,u}(X_{i,1},X_{i,u}) }\right)Q^{-1}\left[ f_{u}\Big(\frac{%
X_{i,u}-x_{u}}{h}\Big) -f_{u}\Big (
\frac{X_{i,u}-x_{u,0}}{h}\Big )
\right.  \notag \\
&\left. +\sum_{2\leq k\leq d,k\neq u,}\left( f_{k}\Big (\frac{X_{i,k}-b_{k}%
}{h}\Big ) -f_{k}\Big (\frac{X_{i,k}-a_{k}}{h}\Big ) \right)
\right] + o_{\mathbb{P}}\left( 1\right).  \label{main}
\end{align}
\item[ii)] For $a_1\leq x_1\leq b_1$, it holds that
\begin{eqnarray}
\!\!\!&\!\!\!&\!\!\!\sqrt{nh}\big ( \widehat{q}_{1}( x_{1})
-q_{1}(x_{1}) \big ) =\sum_{i=1}^{n}\frac{\big ( \left(
1-\alpha \right) -\mathbb{I}(\varepsilon _{i}\leq 0) \big )
w_{2}(X_{i,2}) p_{\bar{2}}(X_{i,\bar{2}})\,
\mathbb{I}(X_{i}\in A_{(u)}) }{\sqrt{nh}%
D_{2}(X_{i,1},X_{i,2})g_{1}(X_{i})}  \notag \\
\!\!\!&\!\!\!&\!\!\!\cdot \left(
e_{1}^{\mbox{\tiny{\rm T}}}-\frac{e_{2}^{\mbox{\tiny{\rm T}}}D_{1,2}(
X_{i,1},X_{i,2}) }{D_{2}(X_{i,1},X_{i,2}) }\right) Q^{-1}%
\left[ c\left( f_{1}\left( \frac{X_{i,1}-x_{1}}{h}\right) -f_{1}\left( \frac{%
X_{i,1}-x_{1,0}}{h}\right) \right) \right.  \notag \\
\!\!\!&\!\!\!&\!\!\!+\big ( c-c_1(x_{1}) w_{1}(
X_{i,1}) \big ) \left. \!\!\sum_{3\leq k\leq d}\!\!\left(
f_{k}\left( \frac{X_{i,k}-b_{k}}{h}\right) -f_{k}\left(
\frac{X_{i,l}-a_{k}}{h}\right) \right) \right] + o_{\mathbb{P}}(
1),  \label{main2}
\end{eqnarray}%
where $c_1(x_{1}) =\int_{x_{1,0}}^{x_{1}}\int \frac{%
D_{1,2}(t_{1},t_{2})}{D_{2}(t_{1},t_{2}) }%
w_{2}(t_{2}) {\rm d}t_{2}{\rm d}t_{1}$.
\end{enumerate}
\end{theorem}

\begin{remark} \normalfont 
Moreover, if the more restrictive condition
$\frac{1}{2p-2-\frac{d }{r}}<\kappa <\frac{1-\frac{1}{r} }{3d+4+%
\frac{4d }{r}-\frac{d}{r^{2}}}$ is given compared to condition (B6), the second
order representation in Theorem \ref{L1} can be specified
explicitly as follows.\vspace{-0.2cm} 
\begin{enumerate}
\item[i)] For $2\leq u\leq d$, the remainder term in Theorem \ref{L1} is
equal to
\begin{equation*}
\xi _{n1}h^{\frac{1}{2}}+O_{\mathbb{P}}\left( h^{1-\frac{1+\varepsilon }{2r}%
}+\left( nh^{2p-1}\right) ^{\frac{1}{2}}+\frac{n^{\frac{1}{2}}}{h^{\frac{1}{2%
}}}\left( n^{1-\frac{1}{3r}-\frac{2\varepsilon }{3}}h^{d\left( 1+\frac{2}{3r}%
-\frac{1}{3r^{2}}\right) }\right) ^{-\frac{3}{4}}\right) ,
\end{equation*}%
where
\begin{eqnarray}
\xi _{n1} &\!=\!&\sum_{i=1}^{n}\left\{ \frac{1}{\sqrt{n}}\left(
a_u(x_u,i)-b_u(x_u,i) \right) \right. +\left[ \left( 1-\alpha
\right) -\mathbb{I}(\varepsilon _{i}\leq 0)
\right]\mathbb{I}(X_{i}\in
A_{(u)})  \notag \\
\!\! &\!\!&\!\! \cdot\left[ \frac{e_{u}^{\mbox{\tiny{\rm T}}}Q^{-1}}{\sqrt{n}}\int
\left.
\frac{\partial }{\partial x^{\mbox{\tiny{\rm T}}}}\left( \frac{w_{1}(x_{1}) p_{\bar{u%
}}\left(x_{\bar{u}}\right) }{g_{1}(x)D_{1,u}(x_{1},x_{u}) }%
\right) \right\vert_{x=\widetilde{X}_{i}}t A(t)K(t){\rm d} t\right.  \notag \\
&&\left. -\frac{e_{1}^{\mbox{\tiny{\rm T}}}Q^{-1}}{\sqrt{n}}\int \left. \frac{\partial }{%
\partial x^{\mbox{\tiny{\rm T}}}}\left( \frac{w_{1}(x_{1}) p_{\bar{u}}(x_{\bar{%
u}}) D_{u}(x_{1},x_{u})}{g_{1}(x)D_{1,u}^{2}(
x_{1},x_{u}) }\right) \right\vert _{x=\widetilde{X}_{i}}t
A(t)K(t){\rm d} t\right]
\notag \\
&&\left.+\left(e_{u}^{\mbox{\tiny{\rm T}}}-e_{1}^{\mbox{\tiny{\rm T}}}\frac{D_{u}( X_{i,1},X_{i,u}) }{%
D_{1,u}(X_{i,1},X_{i,u}) }\right)\frac{w_{1}(
X_{i,1}) p_{\bar{u}}(X_{i,\bar{u}}) Q^{-1}}{\sqrt{nh^{2}}%
g_{1}(X_{i})D_{1,u}(X_{i,1},X_{i,u}) } M_{2}^{(u)}(X_{i}) \right\} h^{\frac{1}{2}}.  \notag
\end{eqnarray}

\item[ii)] For the first additive component, the remainder term in Theorem \ref{L1} is equal to%
\begin{equation*}
\xi _{n2}h^{\frac{1}{2}}+O_{\mathbb{P}}\left( h^{1-\frac{1+\varepsilon }{2r}}+%
\frac{n^{\frac{1}{2}}}{h^{\frac{1}{2}}}\left( n^{1-\frac{1}{3r}-\frac{%
2\varepsilon }{3}}h^{d\left( 1+\frac{2}{3r}-\frac{1}{3r^{2}}\right)
}\right) ^{-\frac{3}{4}}+\left( nh^{2p-1}\right)
^{\frac{1}{2}}\right) ,
\end{equation*}%
where
\begin{eqnarray}
\xi_{n2} = \!\!\!\!&\!\!\!\!&\!\!\sum_{i=1}^{n}\left\{\frac{
b_c(i)-a_c(i)+a_1(x_1,i)-b_1(x_1,i)}{\sqrt{n}}+\big ( \left(
1-\alpha \right) -\mathbb{I}(\varepsilon_{i}\leq 0)\big)\,
\mathbb{I}(X_{i}\in A_{(1)}) \right.  \notag \\
\!\!\!\!&\!\!\!\!&\!\!\cdot\left[ - c_1(x_{1} )e_{1}^{\mbox{\tiny{T}}}Q^{-1}\int
\left. \frac{\partial }{\partial x}\left( \frac{w_{1}(
x_{1})w_{2}(x_{2}) p_{\bar{2}}\left(
x_{\bar{2}}\right)}{g_{1}(x)D_{2}(x_{1},x_{2}) }\right)
\right\vert _{x=\widetilde{X}_{i}^{1,2}}t
A(t)K(t){\rm d} t\right.  \notag \\
\!\!\!\!&\!\!\!\!&\!\!+ c_{1}(x_{1})e_{2}^{\mbox{\tiny{\rm T}}}Q^{-1}\int \left.
\frac{\partial }{\partial x}\left( \frac{w_{1}(x_{1})w_{2}(x_{2})
p_{\bar{2}}(x_{\bar{2}}) D_{1,2}(x_{1},x_{2}) }{%
g_{1}(x)D_{2}^{2}(x_{1},x_{2}) }\right) \right\vert _{x=\widetilde{X}%
_{i}^{1,2}}t A(t)K(t){\rm d} t  \notag \\
\!\!\!\!&\!\!\!\!&\!\! +\mbox{}  c\,e_{1}^{\mbox{\tiny{\rm T}}}Q^{-1}\int \left. \frac{\partial }{%
\partial x}\left( \frac{w_{2}(x_{2}) p_{\bar{2}}(x_{\bar{2}%
}) }{g_{1}(x)D_{2}(x_{1},x_{2}) }\right) \right\vert _{x=%
\widetilde{X}_{i}^{1,2}}t A(t)K(t){\rm d}t  \notag \\
\!\!\!\!&\!\!\!\!&\!\! - c\,e_{2}^{\mbox{\tiny{\rm T}}}Q^{-1}\int \left. \frac{\partial }{%
\partial x}\left( \frac{ w_{2}(x_{2}) p_{\bar{2}}(x_{\bar{2%
}}) D_{1,2}(x_{1},x_{2}) }{g_{1}(x)D_{2}^{2}(
x_{1},x_{2}) }\right) \right\vert_{x=\widetilde{X}_{i}^{1,2}}t
A(t)K(t){\rm d} t
\notag \\
\!\!\!\!&\!\!\!\!&\!\!\left.\left.+\frac{w_{2}(X_{i,2}) p_{\bar{2%
}}(X_{i,\bar{2}})\mathbb{I}(X_{i}\in A) }{\sqrt{nh^2}%
D_{2}(X_{i,1},X_{i,2})g_{1}(X_{i})} \!\left(\! e_{1}^{\mbox{\tiny{\rm T}}}-\frac{%
e_{2}^{\mbox{\tiny{\rm T}}}D_{1,2}(X_{i,1},X_{i,2}) }{D_{2}(
X_{i,1},X_{i,2}) }\!\right)\! Q^{-1} M_{1}^{(u)}(X_{i})\!%
\right] \right\}.  \label{sec2}
\end{eqnarray}
\end{enumerate}
\end{remark}

\begin{remark} \normalfont 
Conditions (B3) and (B6) are about the restriction on the bandwidth. In
order to
get a chosen \ bandwidth, it should hold that $p>\frac{d+1+\frac{d-1}{r}}{1-%
\frac{5}{r}}$. Condition (B10) can be relaxed from Theorem \ref{L1}.
Otherwise, there are two extra similar terms which will be included
in \eqref{main}. The remaining conditions in Theorem \ref{L1} are
standard; see, e.g., \citet{chaudhuri91} and \citet{honda04}. Condition (B9) is used to identify the $q^{\prime}_{u}(\cdot)$'s.
\end{remark}

\vspace{0.2cm}

For convenience, let the set $A_{(u)}^*$ be the limit of $A_{(u)}$ for $%
1\leq u\leq d$. From Theorem \ref{L1}, the following Corollary
\ref{L2} can be inferred from the standard Doob's large-block and
small-block technique; see, e.g.,  \citet[Theorem 2]{cai03}. 

\begin{corollary} 
\label{L2} Under the conditions of Theorem \ref{L1}, for $1\leq
u\leq d$, it
holds that%
\begin{equation*}
\sqrt{nh}\big( \widehat{q}_{u}(x_{u}) -q_{u}(
x_{u}) -B_{1,u} h^{p}\big )  \overset{d}{\rightarrow }%
 \mathcal{N}(0,\sigma _{u}^{2}),
\end{equation*}%
where, for $u=1$, $\frac{\sigma _{u}^{2}}{\alpha \left( 1-\alpha \right) }$ is defined as%
\begin{eqnarray*}
\!\!\!\!\!&\!&\!\int_{\left( t\in A_{(1)}^*\right) }\frac{c^{2}w_{2}(t_{2}) p_{%
\bar{2}}(t_{\bar{2}})
}{D_{2}(x_{1},t_{2})g_{1}(t)}\left( \left(
e_{1}^{\mbox{\tiny{\rm T}}}-\frac{e_{2}^{\mbox{\tiny{\rm T}}}D_{1,2}(x_{1},t_{2})
}{D_{2}(x_{1},t_{2}) }\right) Q^{-1}f_{1}(t_{1}) \right)
^{2}p( x_{1},t_{2},t_{\bar{2}}) {\rm d} t \\
\!\!\!\!\!&\!&\!+\int_{\left( t\in A_{(1)}^*\right)
}\frac{c^{2}w_{2}(t_{2}) p_{\bar{2}}(
t_{\bar{2}}) }{D_{2}(x_{1,0},t_{2})g_{1}(t)}\left( \
\left(e_{1}^{\mbox{\tiny{\rm T}}}-\frac{e_{2}^{\mbox{\tiny{\rm T}}}D_{1,2}(x_{1,0},t_{2}) }{%
D_{2}(x_{1,0},t_{2}) }\right) Q^{-1}f_{1}(
t_{1})
\right)^{2}p(x_{1,0},t_{2},t_{\bar{2}}) {\rm d} t \\
\!\!\!\!\!&\!&\!+\sum_{3\leq k\leq d}\int_{\left( t\in
A_{(1)}^*\right) }\frac{\big(c-c(x_{1},x_{1,0})
w_{1}(t_{1}) \big )
^{2}w_{2}(t_{2}) p_{\bar{2}}(t_{\bar{2}}) }{%
D_{2}(t_{1},t_{2})g_{1}(t)} \\
\!\!\!\!\!&\!&\!\cdot \left( \left(e_{1}^{\mbox{\tiny{\rm T}}}-\frac{e_{2}^{\mbox{\tiny{\rm T}}}D_{1,2}%
\left( t_{1},t_{2}\right) }{D_{2}(t_{1},t_{2}) }\right)
Q^{-1}f_{k}(t_{k}) \right)^{2}\big (p(
t_{1},b_{k},t_{\bar{k}}) +p(t_{1},a_{k},t_{\bar{k}}) \big) {\rm d} t
\end{eqnarray*}%
and, for $2\leq u\leq d$, $\frac{\sigma _{u}^{2}}{\alpha \left(
1-\alpha
\right) }$ is defined as%
\begin{eqnarray*}
&&\int_{\left(t\in A_{(u)}^*\right) }\frac{w_{1}^{2}\left( t_{1}\right) p_{%
\bar{u}}^{2}(t_{\bar{u}}) p(
t_{1},x_{u,0},t_{\bar{u}}) }{D_{1,u}^{2}(
t_{1},x_{u,0}) g_{1}^{2}(t_{1},x_{u,0},t_{\bar{u}})}\left(
\left(e_{u}^{\mbox{\tiny{\rm T}}}-\frac{e_{1}^{\mbox{\tiny{\rm T}}}D_{u}(t_{1},x_{u,0})
}{D_{1,u}(t_{1},x_{u,0}) }\right) Q^{-1}f_{u}(t_{u}) \right)
^{2}{\rm d} t \\
&&+\int_{\left(t\in A_{(u)}^*\right) }\frac{w_{1}^{2}(t_{1}) p_{%
\bar{u}}^{2}(t_{\bar{u}})p(t_{1},x_{u},t_{
\bar{u}})  }{D_{1,u}^{2}(t_{1},x_{u})
g_{1}^{2}(t_{1},x_{u},t_{\bar{u}})}\left( \left(e_{u}^{\mbox{\tiny{\rm T}}}-\frac{%
e_{1}^{\mbox{\tiny{\rm T}}}D_{u}(t_{1},x_{u}) }{D_{1,u}(t_{1},x_{u}) }%
\right)Q^{-1}f_{u}(t_{u}) \right) ^{2}{\rm d} t \\
&&+\sum_{2\leq k\leq d,k\neq u}\int_{\left(t\in A_{(u)}^*\right) }\frac{%
w_{1}^{2}(t_{1}) p_{\bar{u}}^{2}(t_{\bar{u}}) }{%
D_{1,u}^{2}( t_{1},t_{u}) g_{1}^{2}(t)}\left(\left(e_{u}^{\mbox{\tiny{\rm T}}}-%
\frac{e_{1}^{\mbox{\tiny{\rm T}}}D_{u}(t_{1},t_{u}) }{D_{1,u}(
t_{1},t_{u}) }\right) Q^{-1}f_{k}(t_{k}) \right)^{2} \\
&&\cdot (p\big(t_{1},b_{k},t_{\bar{k}}) +p(t_{1},a_{k},t_{\bar{k}}) \big ) {\rm d} t,
\end{eqnarray*}
and where $B_{1,u}$ is defined in Lemma \ref{WW}. 
\end{corollary}

\begin{remark}\normalfont
The optimal bandwidth is equal to $h_{opt}=C n^{-\frac{1}{2p+1}}$.
\end{remark}

\section{Uniform convergence of additive components}\label{sec:add}
\begin{theorem} 
\label{uniform} Under conditions (B1)--(B8), it holds with probability
one that
\begin{equation}
\sup_{x_{u}\in \left[ a_{u},b_{u}\right] }|
\widehat{q}_{u}(x_{u}) -q_{u}(x_{u}) | =O\left( \left( nh^{1+%
\frac{1+\varepsilon }{r}}\right) ^{-\frac{1}{2}}\right)  \label{U3}
\end{equation}%
for $1\leq u\leq d$ and any sufficiently small constant $\varepsilon
>0$.
\end{theorem}

\setcounter{equation}{0}
\renewcommand{\theequation}{5.\arabic{equation}}

\section{Convergence of the unknown link function}\label{sec:link}

In this section, we address the asymptotic representation for the
estimated  link function $\widehat{G}_{n}(v)$. In particular, we show that the resulting representation
holds uniformly for $v\in \mathcal{V}$.
Then, the corresponding
asymptotic normality   with the bias will be illustrated. Furthermore,
we discuss the choice of the optimal bandwidth $h_{G,opt}$. In
the sequel, let $F_{n}(t|v)$ be an empirical conditional\
distribution, which equals the right-hand side (RHS) of \eqref{G} with
$\widehat{q}_{0}(X_{i})$ replaced by $q_{0}(X_{i})$.

We impose the following conditions:
\vspace{-0.2cm}
\begin{enumerate}
\item[(C1)] $K_{G}(x)$ is symmetric on the support set $[-1,1]$. $
K_{G}^{\prime }(1) =K_{G}^{\prime }(-1) =0$,
$ K_{G}(1) =K_{G}(-1) =0$ and $K_G^{\prime
\prime }\leq 0$. There exists a constant $C>0$ such that $\left\vert
K_{G}^{\prime \prime }(x+t)-K_{G}^{\prime \prime }(x)\right\vert
\leq C|t|$ for any $x$ and $t$.

\item[(C2)] The density function $f_{q_{0}}(v)$ of $q_0(X)$
has the second order continuous derivative for $v\in \mathcal{V}$,
and $f_{q_{0}}(v) >0 $.

\item[(C3)] $\liminf_{n\rightarrow \infty }nh^{1+\frac{1+\varepsilon }{r}%
}h_{G}^{3}>0$, \,\,$\sqrt{\frac{\log n}{nh_{G}^{5(1+\frac{1+\varepsilon }{r})}}}%
\geq 1$,\,\, $h_{G}/h^{1+\frac{1+\varepsilon }{r}}\rightarrow 0$ hold.

\item[(C4)] Let $F(y|v) $ be the conditional distribution
function of $Y_{i}$ given $q_{0}(X_{i})=v$. $f(y|v)$ is the
conditional density
function of $F(y|v)$ and has the first order continuous derivative at $%
y=G(v) $. For any $y$ in the neighbourhood of $G(v)$, $F(y|v)$ has
the first order derivative with respect to $v\in \mathcal{V}$.

\item[(C5)] Let $f_{q_{0}}(y|z)$ be the conditional density function of $%
\varepsilon _{i}$ given that $q_{0}(X_{i})=z$.
Furthermore, $\frac{\partial^{2}f_{q_{0}}(y|z) }{\partial z\partial
y}$ exists in the neighbourhood of $( 0,v)$ for any $v\in \mathcal{V}$.
\end{enumerate}

During the process of proving Theorem \ref{link3} in Appendix \ref{app:D},
the following Lemma \ref{link6} is in fact proved. We list it here
as an independent result.

\begin{lemma}\label{link6} 
Under conditions of Theorem \ref{link3}, it holds with
probability one that
\begin{equation}
\frac{1}{nh_G^2}\sum_{m=1}^{n}K_G^{'}\left(\frac{v-q_0(X_m)}{h_G}\right)\big (\widehat{q}_0(X_m)-q_0(X_m)\big )
=o\left(\frac{1}{\sqrt{nh_G}}\right).  \label{glem1}
\end{equation}
\end{lemma}

\begin{lemma}\label{link4}i) 
Under conditions ii) of Theorem \ref{link3}, it
holds with probability one that
\begin{equation}
\big | \widehat{G}_{n}(v) -G(v) \big | =O_{{\!}%
}\left( \left( \frac{1}{nh_{G}^{1+\frac{1+\varepsilon }{r}}}\right) ^{\frac{1%
}{2}}\right) .  \label{rate2}
\end{equation}%
ii) Under conditions ii) of Theorem \ref{link3}, with probability one, 
\eqref{rate2} holds uniformly with respect to $v\in \mathcal{V}$.
\end{lemma}

\begin{theorem}\label{link3} 
Assume that the conditions of Theorem \ref{uniform} and
(C1)--(C4) hold. Then \newline i) For any fixed $v$ $\in \mathcal{V}$,
with probability one, we have the following asymptotic
representation
\begin{equation}
\widehat{G}_{n}(v) -G(v) =\frac{1}{f\big(G(v) |v\big)%
}\big ( (1-\alpha )-F_{n}(G(v) |v)\big ) +O\left( \frac{1}{%
\sqrt{nh^{1+\frac{1+\varepsilon }{r}}}}\right) .  \label{link2}
\end{equation}%
ii) Furthermore, if conditions (C1)--(C4) hold for any $v\in \mathcal{V}$,
\eqref{link2} holds uniformly for $v\in \mathcal{V}$ with
probability one.
\end{theorem}

\begin{corollary}\label{link5}
Under the conditions of Theorem \ref{link3} and condition (C5), we
have that
\begin{equation}
\sqrt{\frac{nh_{G}f_{q_{0}}(v) }{\alpha \left( 1-\alpha \right) }%
}\left( \widehat{G}_{n}(v)-G(v)-a(v)
\big (1+o\left( 1\right) \big ) h_{G}^{2}\right) \overset{d}{\rightarrow }%
\mathcal{N}(0,1),  \label{a9}
\end{equation}%
where
\begin{equation*}
a(v) =\frac{\int s^{2}K_{G}(s)
{\rm d}s}{f_{q_{0}}(v) }\left[ \frac{\partial \left(
f_{q_{0}}(0|v)
f_{q_{0}}(v) G^{\prime }(v) \right) }{\partial v}%
+f_{q_{0}}^{\prime }(v) \left. \frac{\partial
^{2}f_{q_{0}}(y|v) }{\partial v\partial y}\right\vert _{y=0}%
\right] .
\end{equation*}
\end{corollary}

\begin{remark}\normalfont 
From \eqref{a9}, the asymptotic mean squared error (AMSE) for $\widehat{G}%
_{n}\left( v\right) -G\left( v\right) $ is equal to
\begin{equation*}
\frac{\alpha \left( 1-\alpha \right) }{nh_{G}f_{q_{0}}(v)}%
+a^{2}(v)h_{G}^{4}\big(1+o(1) \big) .
\end{equation*}%
Hence, the optimal bandwidth of $h_{G}$ in the sense of the AMSE is chosen as%
\begin{equation*}
h_{G,opt}=\left( \frac{\alpha \left( 1-\alpha \right) }{a^{2}(v) f_{q_{0}}(v) }\right)
^{\frac{1}{5}}n^{-\frac{1}{5}}.
\end{equation*}
\end{remark}
\endproof%

\section{Concluding Remarks}\label{sec:concl}
This paper has been concerned with  estimating the conditional quantile of a scalar random variable $Y$ conditional on a vector of covariates $X$ for a generalized additive  model  specification  with an unknown  link function. We have established various theoretical properties of the proposed estimators including consistency and asymptotic normality. This extension of  estimating  the  generalized additive conditional mean regression model, is certainly non-trivial and demanding from a technical point of view for the large-sample properties of the proposed estimators. Furthermore, by allowing for a general form of serial dependence in the data, we enlarged the range of possible applications in practical situations.

\newpage

 \setcounter{page}{1} 
\noindent
\section*{Supplementary Material for ``Estimating Generalized Additive Conditional Quantiles by Absolutely Regular Processes'' } 
\centerline{By Yebin Cheng and Jan G.\ De Gooijer}

\medskip

\begin{appendices}
\renewcommand{\theequation}{A.\arabic{equation}}
\setcounter{equation}{0}
\renewcommand{\thesubsection}{\Alph{subsection}}
 \setcounter{lemma}{0}
\renewcommand{\thelemma}{\Alph{subsection}\arabic{lemma}}

\subsection{Preliminary Lemmas}
In this Appendix we introduce two preliminary Lemmas \ref{lemmaA1} and \ref{lemmaA2} which
are used in the proofs of Theorems 3.3, 4.1, and 5.3. Lemma \ref{lemmaA1} is a kind of Bernstein's inequality on absolutely regular processes. Lemma \ref{lemmaA2}  is a moment inequality on degenerated   U-statistics. This lemma extends Lemma
3 of Arcones (1998) to the case of higher moments.
In the following two lemmas, we assume that $\{\xi_i\},\ i=1,2,\ldots,n$, is a sequence of
strictly stationary random variables.
Recall that 
$\mathbb{E}_{j_{0}}g(\xi_{i},\xi_{j_{0}})=g_{1}(\xi_{j_{0}})$.

\begin{lemma}\label{lemmaA1} 
 Suppose that $g(\cdot,\cdot)$ is a Borel measurable function with the bound $M>0$. Let $g_{2}(\cdot )=\mathbb{E}g(\xi _{i},\cdot )$,
$\sigma (\cdot )=\mbox{Var}(\sum_{i=1}^{q}g(\xi _{i},\cdot ))$ and
$U_{ij_{0}}=g(\xi
_{i},\xi _{j_{0}})-g_{2}(\xi _{j_{0}})$\textit{, where }$1\leq j_{0}\leq n$
 is fixed. Then, for any $x>0$,
$r_{1}>1$ and positive integer $q\leq
\frac{n}{4}$, it holds that
\begin{align}
\mathbb{P}\Big\{\Big| \sum_{1\leq i\leq n,i\neq
j_{0}}U_{ij_{0}}\Big | \geq x\Big \} \!\leq
\!2\mathbb{E}\exp \left\{ \frac{-\left( \frac{x}{4}\right)
^{2}}{\frac{n}{2q}\sigma (\xi
_{j_{0}})+\frac{2}{3}qM\frac{x}{4}}\right\} \!+\!\frac{n}{q}\beta
(q) \!+\!\frac{2^{r_{1}}q^{r_{1}-1}}{x^{r_{1}}}\sum_{|i-j_{0}|<2q}
\mathbb{E}|U_{ij_{0}}|^{r_{1}}.  \label{Bernstein5}
\end{align}%
\end{lemma}

\vspace{-0.4cm}
\proof
Assume that $j_{0}-1=m_{1}q+r_{3}$ and $n-j_{0}=m_{2}q+r_{4}$, where
$m_{1}$ and $m_{2}$ are two non-negative integers and $0\leq r_{3}$,
$r_{4}<q$. The summation in the probability on the left-hand side (LHS) of
\eqref{Bernstein5} can be rewritten into a sum of two different
summations according to whether the subscript $i$ satisfies
$-(q+r_{3})\leq i-j_{0}\leq q+r_{4}$ or not. By
using Markov's inequality, the last term on the RHS of the  inequality \eqref%
{Bernstein5} comes from the summation in which $i$ satisfies the
condition. Through exploring Berbee's lemma and Bernstein's
inequality, the first two terms on the RHS of \eqref{Bernstein5} can be derived from the summation in
which $i$ does not satisfy the condition.
\endproof 

\noindent 
\textbf{Remark 1.}
The bound of $\sigma (\cdot)$ in \eqref{Bernstein5} can
be obtained from Davydov's inequality, H\"{o}lder's inequality and
$C_{r}$-inequality as follows
\begin{align}
\sigma (\cdot )\leq q \mbox{Var}\big(g(\xi _{1},\cdot
)\big)+2q\sum_{l=1}^{q}\big (\beta \left( l\big ) \right)
^{1-\frac{2}{r_{2}}}\big (\mathbb{E}\left\vert
g(\xi _{i},\cdot )-g_{1}(\cdot )\right\vert ^{r_{2}}\big)^{\frac{2}{r_{2}}%
}\leq Cq\big (\mathbb{E}\left\vert g(\xi _{i},\cdot )\right\vert
^{r_{2}}\big ) ^{\frac{2}{r_{2}}},  \label{Bernstein3}
\end{align}%
where $r_{2}$ satisfies that $\sum_{l=1}^{\infty }\left(\beta
\left(
l\right) \right)^{1-\frac{2}{r_{2}}}<\infty$ and $1-\frac{2}{r_{2}}=\frac{%
1+\varepsilon }{r}$ for a sufficiently small constant $\varepsilon
>0$. Without loss of generality (w.l.o.g.), $q$ is assumed to 
be an integer throughout the next three appendices
although it may not be sometimes according to its formula.
In some cases, if $\xi_{j_{0}}$ is replaced by a constant, then
\eqref{Bernstein5} still holds but its last term will be excluded.

\vspace{0.1cm}

\begin{lemma}\label{lemmaA2} 
 Let $U_{n}=\sum_{1\leq
i<j\leq n}h_{n}(\xi_{i},\xi_{j})$ be a degenerated
U-statistic with the symmetric kernel $h_{n}(\cdot ,\cdot
)$, i.e., for any $t\in \mathbb{R}$, $\mathbb{E}h
(\xi_{i},t)=0$. Then for $k\in \mathbb{N}$, there
exists a universal constant $C>0$ such that
\begin{equation}
\mathbb{E}U_{n}^{k}\leq Cn^{k}\left( 1+\sum_{i=1}^{n-1}i^{k}\beta _{i}^{1-\frac{1%
}{s}}\right) M_{sk}^{k},  \label{U}
\end{equation}%
where $s>1$ and
\begin{equation*}
M_{sk}=\sup_{(i_{1},i_{2}),\mathbb{P}}\left( \int \left\vert h_{n}(\xi
_{i_{1}},\xi_{i_{2}})\right\vert ^{sk}{\rm d}\mathbb{P}\right)
^{\frac{1}{sk}}
\end{equation*}%
with $\mathbb{P}$ being either the probability measure 
$\mathbb{P}_{(\xi_{i_{1}},\xi_{i_{2}})}$ or $\mathbb{P}_{\xi
_{i_{1}}}\otimes \mathbb{P}_{\xi_{i_{2}}}$.
\end{lemma}

\noindent
\begin{proof}
Let $J(i_{1},\ldots ,{i_{{2k}}})=\mathbb{E}[h_{n}(\xi _{i_{1}},\xi
_{i_{2}})\cdots h_{n}(\xi_{i_{2k-1}},\xi_{i_{2k}})]$. First, assume that $%
i_{1},\ldots ,i_{2k}$ are different values. Rearrange the initial sequence $%
i_{1},\ldots ,i_{2k}$ in natural order as $j_{1},\ldots
,j_{2k}$. Define $d_{1}=j_{2}-j_{1}$, $d_{2k}=j_{2k}-j_{2k-1}$ and
$d_{i}=\min \{j_{i}-j_{i-1},\ j_{i+1}-j_{i}\}$ for $i=2,\ldots
,2k-1$. If, $d_{i_{0}}$ is the largest number among
$\{d_{i},i=1,\ldots ,{2k}\}$, then we compare the initial sequence
$(i_{1},\ldots ,i_{2k})$ with the one having the
independent blocks $(i_{1},\ldots ,i_{i_{0}-1})$, $i_{i_{0}}$ and $%
(i_{i_{0}+1},\ldots ,i_{2k})$ and the identical block distributions.
Also,
it can be inferred that there exist at least $k$ numbers among $%
(d_{1},\ldots ,d_{2k})$ which are in no excess of $d_{i_{0}}$. Thus,
by part \textit{ii)}  of \citet[Lemma 2]{arcones98}, monotonicity on $\beta _{i}$
and H\"{o}lder's inequality, we obtain that
\begin{equation*}
\sum_{i_{1}\neq \ldots \neq i_{2k}}J(i_{1},\ldots ,i_{2k})\leq
Cn^{k}\sum_{i=1}^{n-1}i^{k}\beta _{i}^{1-\frac{1}{s}}M_{sk}^{k}.
\end{equation*}%
The other cases of $i_{1},\ldots ,i_{2k}$ can be dealt with
similarly.
\end{proof}

\noindent \textbf{Remark 2.} 
 Usually, $s$ is taken to be $1-\frac{%
k+1+\varepsilon }{r}$ with $r>k+1$.

\subsection{Bahadur Representation}\label{sec:bahadur}
\renewcommand{\theequation}{B.\arabic{equation}}
\newtheorem{thm}{Theorem}[subsection]
\setcounter{equation}{0}
\setcounter{lemma}{0}
\setcounter{section}{1}
\setcounter {subsection}{2}
\renewcommand{\thetheorem}{B,\arabic{section}}
\renewcommand{\thetheorem}{\thesubsection\arabic{section}}

In this Appendix, and following an identical line as that of \citet{honda04} and \citet{chaudhuri91},
we  address the uniformly strong
Bahadur representation for the conditional quantile and its
derivatives at a random point. Let
\begin{equation*}
V_{ij}(t_{1,u},\beta) =K_{ij,t_{1,u}}\left\{\rho
_{\alpha }\left( Y_{i}-\beta_{j,t_{1,u}}^{\mbox{\tiny{T}}}
A_{ij,t_{1,u}}\right)
-\rho_{\alpha }\left( Y_{i}-\beta^{\mbox{\tiny{T}}}
A_{ij,t_{1,u}}\right)
\right\}.
\end{equation*}

\begin{lemma}\label{lemmaB3}  
If conditions (B1)--(B4) and $r\geq d+%
\frac{1}{2}-\frac{d}{p}+\frac{17p+d-3}{2p(d-1-dp^{-1})}$
hold, then for any $1\leq j\leq n$, it holds almost surely
and uniformly both on $\mathcal{X}_{1,u}$\ and on
$|\beta -\beta_{j,t_{1,u}}| =\frac{B}{\sqrt{nh^{(1+\frac{1}{r})d+\varepsilon }}}$%
 that 
\begin{equation}
\left\vert \sum_{i=1}^{n}(V_{ij}(t_{1,u},\beta) -\mathbb{E}%
_{j}V_{ij}(t_{1,u},\beta)  \right\vert \leq
Bh^{-\left( \frac{d}{r}+{\varepsilon }\right) },  \label{eq}
\end{equation}%
 where $B>0$ is a constant and may be chosen
large.
\end{lemma} 
\begin{proof}
By using Lemma \ref{lemmaA1} and taking $q=\sqrt{nh^{\left(
1-\frac{1}{r}\right)d}}$, we have that
\begin{equation}
\mathbb{P}\left( \left\vert \sum\limits_{i=1}^{n}\big(V_{ij}(
t_{1,u},\beta) -\mathbb{E}V_{ij}(t_{1,u},\beta )
\big)
\right\vert >CBh^{-(\frac{d}{r}+\varepsilon )}\right) \leq 2n^{-CB}+\frac{n}{%
q}\beta (q) +C\left( 2h^{\frac{\varepsilon }{2}}\right)
^{r_{1}}, \label{in1}
\end{equation}%
where the three terms on the RHS of \eqref{in1} are
derived from their counterparts in \eqref{Bernstein5}, respectively.
More specifically,
in the first term the relationship $\sigma (Z_{j})\leq Cqh^{\frac{2d}{r_{2}}%
}\delta ^{2}$, which follows from \eqref{Bernstein3} and condition (B1), is used, where $%
\delta =\frac{B}{\sqrt{nh^{(1+\frac{1}{r})d+\varepsilon }}}$ and $1-\frac{2}{%
r_{2}}=\frac{1}{r}+\frac{\varepsilon }{2d}$. And in the third term,
the relationship $\mathbb{E}|V_{ij}|^{r_{1}}\leq C\delta^{r_{1}}$
is used.

In order to prove uniformity related to $t_{1,u}$ in \eqref{eq}, we divide $%
\mathcal{X}_{1,u}$ into smaller squares with the side length $\ell
_{n}=\delta ^{2}$. For any point $s=(s_{1},s_{u})\in
\mathcal{X}_{1,u}$, let $t_{1,u}\in \mathcal{X}_{1,u}$ be the
nearest grid points close to $s$. Then, $\left\vert
t_{1,u}-s\right\vert \leq \ell _{n}/\sqrt{2}$. To prove uniformity
related to $\beta $ in \eqref{eq}, it is necessary to divide the two
spheres $|\beta - \beta_{j,t_{1,u}}|=B\delta $ 
and $\left\vert
\gamma
-\beta_{j,s}\right\vert =B\delta $ into smaller cells with the radius $%
d_{1}=\ell_{n}$ such that the two divisions are location
equivariant. It can be seen that the total number of such kind of
cells related to each
sphere is equal to $O(\ell _{n}^{-\frac{d-1}{2}}) $. For any $%
\alpha $ in the sphere $\left\vert \gamma -\beta_{j,s}\right\vert =B\delta $%
, let $\beta$ in the sphere $|\beta -\beta
_{j,t_{1,u}}| =B\delta $ be the nearest grid point to the
point which is in the sphere $|\beta -\beta_{j,t_{1,u}}|
=B\delta $ and is equivariant to $\gamma$. Then, it can be inferred that $|\beta _{j,t_{1,u}}-\beta _{j,s}| \leq C\ell _{n}$ and $%
|\gamma -\beta | \leq C(\ell
_{n}+d_{1}) \leq C\ell_{n}$. 

Next we  prove that for
$s$ and $t_{1,u}$ mentioned above, it holds that
\begin{equation}
\left\vert V_{ij}(s,\gamma) -V_{ij}(t_{1,u},\beta
)
\right\vert \leq CB\ell_{n}\mathbb{I}_{\left( \left\vert X_{i}-\widetilde{X}%
_{j}\right\vert \leq h\right) }+B\delta\mathbb{I}_{\left( h-\ell
_{n}\leq \left\vert X_{i}-\widetilde{X}_{j}\right\vert \leq h+\ell
_{n}\right) }. \label{eq2}
\end{equation}%
In fact, if the two events $|X_{i}-\widetilde{X}_{j,t_{1,u}}|
\leq h$ and $|X_{i}-\widetilde{X}_{j,s}| \leq
h$ occur simultaneously, then from  (B2), it can be inferred that
\begin{equation}
\big|(K_{ij,t_{1,u}}-K_{ij,s})(\rho
_{\alpha}(\varepsilon_{i}+r_{i,j,t_{1,u}}) -\rho
_{\alpha }(\varepsilon_{i}+r_{i,j,t_{1,u}}+P_{i,j,t_{1,u}}) \big) 
|\leq C\ell _{n}\mathbb{I}_{(|X_{i}-\widetilde{X}_{j,t_{1,u}|}| \leq h)}  \label{K12}
\end{equation}
in view of $\frac{\delta }{h}\rightarrow 0$. Noting that
\begin{equation}
\left\vert r_{ij,t_{1,u}}-r_{ij,s}\right\vert \leq \left\vert \beta
_{j,t_{1,u}}-\beta _{j,s}\right\vert \left\vert
A_{ij,t_{1,u}}\right\vert +\left\vert \beta _{j,s}\right\vert
\left\vert A_{ij,t_{1,u}}-A_{ij,s}\right\vert \leq C\ell_{n},
\label{res1}
\end{equation}%
we have
\begin{align}
& \big |K_{ij,s}\big\{\big( \rho_{\alpha}
(\varepsilon_{i}+r_{i,j,t_{1,u}}) 
-\rho_{\alpha}
(\varepsilon_{i}+r_{i,j,t_{1,u}}+P_{i,j,t_{1,u}})\big) \nonumber \\ 
&- 
\big((\rho_{\alpha} 
\varepsilon_{i}+r_{i,j,s})
 -\rho_{\alpha}(\varepsilon_{i}+r_{i,j,s}+P_{i,j,s}^{\gamma }) \big) \big\}  \big | \leq C\ell_{n}\mathbb{I}_{X_{i}-\widetilde{X}_{j,t_{1,u}}|\leq h) }.  \label{K2}
 \end{align}
From  \eqref{K12} and \eqref{K2}, it follows that
\begin{equation*}
\left\vert V_{ij}(s,\gamma) -V_{ij}(t_{1,u},\beta)
\right\vert \leq C\ell_{n}\mathbb{I}_{(|X_{i}-\widetilde{X}%
_{j,t_{1,u}}| \leq h) }.
\end{equation*}%
If $|X_{i}-\widetilde{X}_{j,t_{1,u}}| \leq h$ and $%
|X_{i}-\widetilde{X}_{j,s}| >h$ occur simultaneously or $%
|X_{i}-\widetilde{X}_{j,t_{1,u}}| >h$ and $|X_{i}-%
\widetilde{X}_{j,s}| \leq h$ occur simultaneously, then it can
be inferred similarly that
\begin{equation*}
|V_{ij}(s,\gamma) -V_{ij}(t_{1,u},\beta)
| \leq C\delta\mathbb{I}_{ h-\ell _{n}\leq \left\vert X_{i}-%
\widetilde{X}_{j,t_{1,u}}\right\vert \leq h+\ell _{n}t) }.
\end{equation*}%
From the two cases above, we know that  \eqref{eq2} holds. Thus, we
consider the following two probabilities
\begin{equation*}
\mathbb{P}\left(\sum_{i=1}^{n}\mathbb{I}_{\left( \left\vert X_{i}-\widetilde{X}%
_{j}\right\vert \leq h\right) }\geq \frac{CB}{\ell _{n}}h^{-\left( \frac{d}{r%
}+\varepsilon \right) }\right)\,\, \mbox{and}\,\, \mathbb{P}\left(
\sum_{i=1}^{n}\mathbb{I}_{\left( h-\ell_{n}\leq \left\vert X_{i}-\widetilde{X}%
_{j}\right\vert \leq h+\ell _{n}\right) }\geq CB\delta ^{-1}h^{-\left( \frac{%
d}{r}+\varepsilon \right) }\right).
\end{equation*}%
Similar as the proof of \eqref{in1}, the two probabilities above can
also be bounded by the RHS of \eqref{in1}. For the first
probability, to compare the two terms of the denominator in the
first term on the RHS of \eqref{Bernstein5}, the fact
that
\begin{equation*}
\frac{n}{2q}\cdot q\left(\mathbb{E}_{j}\mathbb{I}_{\left( \left\vert X_{i}-%
\widetilde{X}_{j}\right\vert \leq h\right) }\right)^{\frac{2}{r_{2}}}\leq Cq%
\frac{h}{\ell _{n}}h^{-\left( \frac{d}{r}+\varepsilon \right) },
\end{equation*}%
where $r_{2}$ satisfies that $1-\frac{2}{r_{2}}=\frac{1}{r}+\frac{%
\varepsilon }{2d}$, is used in view of $nh^{\left(
1+\frac{1}{r}\right) d+\varepsilon }\rightarrow \infty $. For the
second probability, we use the fact that
\begin{equation*}
\frac{n}{2q}\cdot q\left(\mathbb{E}_{j}\mathbb{I}_{\left( h-\ell
_{n}\leq \left\vert X_{i}-\widetilde{X}_{j,t_{1,u}}\right\vert \leq
h+\ell _{n}\right) }\right) ^{\frac{2}{r_{2}}}\leq \frac{n}{2q}\cdot
q\left( C\ell
_{n}h^{d-1}\right) ^{\frac{2}{r_{2}}}\leq Cq\delta ^{-1}h^{-\left( \frac{d}{r%
}+\varepsilon \right) },
\end{equation*}%
where $r_{2}$ satisfies that $1-\frac{2}{r_{2}}=\frac{1}{r}+\varepsilon _{0}$%
, and the last inequality follows from $\kappa
<\frac{1-\frac{d}{r}}{d\left(
1-\frac{1}{r}-\frac{d}{r^{2}}+\frac{1}{d}\right) }$, and
$\varepsilon_{0}$ and $\varepsilon$ are sufficiently small
positive reals. Since the total
number of small cells related to the divisions both on the domain $\mathcal{X%
}_{1,u}$ and the sphere $|\beta -\beta_{j}|=\delta $
is $O(\ell_{n}^{-2}\cdot \ell_{n}^{-\frac{d-1}{2}}) =O(\ell_{n}^{-\frac{%
d+3}{2}})$. Multiplying by $O(\ell_{n}^{-\frac{d+3}{2}})$ on
both sides of the inequality
\eqref{in1}, then we can see that its RHS is controlled
by $\frac{C}{n\left( \log n\right) ^{2}}$, provided we choose $B$
and $r_{1}$ sufficiently large. Also, the inequality
\begin{equation}
\ell_{n}^{-\frac{d+3}{2}}\cdot\frac{n}{q}\cdot \beta (q) \leq \frac{C}{%
n\left( \log n\right) ^{2}}  \label{in3}
\end{equation}%
follows from  (B3) and the condition on $r$. Then,
by the Borel-Cantelli Lemma, 
\eqref{eq} holds.
\end{proof}

\begin{lemma} \citep[Lemma 3.2]{honda04}
Under conditions (B1)--(B4), it holds uniformly on 
$\mathcal{X}_{1,u}$\textit{\ that }\[\left\vert
\widehat{\beta}_{j,t_{1,u}}-\beta
_{j,t_{1,u}}\right\vert =O\left( \frac{1}{\sqrt{nh^{(1+\frac{1}{r}%
)d+\varepsilon }}}\right) .\]
\end{lemma}

Let
\begin{equation}
\Delta_{ij}(t_{1,u},\beta)
=K_{ij,t_{1,u}}A_{ij,t_{1,u}}\left[\mathbb{I}\left( \varepsilon
_{i}\leq -r_{ij,t_{1,u}}-P_{ij,t_{1,u}}\right) -\mathbb{I}\left(
\varepsilon _{i}\leq -r_{ij,t_{1,u}}\right) \right]. \label{W}
\end{equation}

\begin{lemma}\label{lemmaA5} 
 Under conditions (B1)--(B4) and $r\geq
\max \left\{ d-7+\frac{2d^{2}-4-4d}{p}-\frac{22p+6}{dp},d\right\}
$, with
probability one, it holds uniformly on $t_{1,u}\in \mathcal{X}_{1,u}$
and the sphere $\left\vert \beta -\beta _{j,t_{1,u}}\right\vert =%
\frac{B}{\sqrt{nh^{(1+\frac{1}{r})d+\varepsilon }}}$ that 
\begin{equation}
\left\vert \sum_{i=1}^{n}\big(\Delta _{ij}(t_{1,u},\beta) -
\mathbb{E}_{j}\Delta_{ij}(t_{1,u},\beta) \big)
\right\vert \leq B\left( n^{1+\frac{1}{r}+\varepsilon }h^{d\left(
1-\frac{1}{r}\right)^{2}}\right)^{\frac{1}{4}}.  \label{Bahadur4}
\end{equation}%
\end{lemma}

\begin{theorem}\label{theoremA6} 
Under conditions (B1)--(B4),
the following strong Bahadur representation  
\begin{align}
\widehat{\beta}_{j,t_{1,u}}-\beta _{j,t_{1,u}}&=\frac{Q_{jn,t_{1,u}}^{-1}}{nh^{d}}%
\sum_{i=1}^{n}K_{ij,t_{1,u}}A_{ij,t_{1,u}}\big ((1-
 \alpha)
-\mathbb{I}\left(
\varepsilon_{i}\leq -r_{i,j}\right) \big ) \notag\\
& +O\left( \left( n^{1-\frac{1}{3r%
}-\frac{2\varepsilon }{3}}h^{d\left(
1+\frac{2}{3r}-\frac{1}{3r^{2}}\right) }\right)
^{-\frac{3}{4}}\right)  \label{Bahadur}
\end{align}%
holds almost surely and uniformly for $1\leq j\leq n$, 
$t_{1}$\textit{\ and }$t_{u}$.
\end{theorem}

\begin{proof}
We first note that
\begin{align}
\sum_{i=1}^{n}K_{ij,t_{1,u}}A_{ij,t_{1,u}}\left[ (1-
 \alpha)
-\mathbb{I}\left( \varepsilon_{i}\leq -r_{i,j}\right) \right]~~~~~~~~~~~~~~~~~~~~~~~~~~~~~~~~~~~~~~~~~~~~~~~~~~\nonumber \\ 
=n\mathbb{E}_{j}\Delta_{ij}(t_{1,u},\beta)
+\sum_{i=1}^{n}\big( \Delta_{ij}(t_{1,u},\beta)
-\mathbb{E}_{j}\Delta_{ij}(t_{1,u},\beta ) \big )\
 +R_{n}(\beta) ,  \label{a4}
\end{align}%
where
\begin{equation*}
R_{n}(\beta)
=\sum_{i=1}^{n}K_{ij,t_{1,u}}A_{ij,t_{1,u}}\big((1-
 \alpha) -\mathbb{I}(\varepsilon _{i}\leq
-r_{ij,t_{1,u}}-P_{ij,t_{1,u}})\big) .
\end{equation*}%
Then,  using condition (B4) and Taylor's expansion for $g(
X_{i},\cdot )$, it can be inferred that
\begin{align}
n\mathbb{E}_{j}\Delta_{ij}(t_{1,u},\beta )
&=n\mathbb{E}_{j}K_{ij,t_{1,u}}A_{ij,t_{1,u}}\left[ G(
X_{i},-r_{ij,t_{1,u}}-P_{ij,t_{1,u}}) -G(
X_{i},-r_{ij,t_{1,u}}) \right]  \notag \\
&=nh^{d}Q_{jn,t_{1,u}}\left( \left( \beta -\beta
_{j,t_{1,u}}\right) +O(\left\vert \delta \right\vert
^{2}+\left\vert \delta \right\vert h^{p}) \right) .
\label{a5}
\end{align}%
Under condition (B3), we have that $nh^{d}\left\vert \delta \right\vert
^{2}=O\left( \delta _{n}\right) $ and $nh^{d}\left\vert \delta
\right\vert h^{p}=O\left(
\delta_{n}\right) $. Also, there exists a constant $\phi >0$\ such that $%
|R_{n}(\widehat{\beta}_{j,t_{1,u}})|
\leq \phi $ holds almost surely. Thus, \eqref{Bahadur} holds.
\end{proof}

\subsection{Proofs of Theorems}\label{sec:proofs}
\renewcommand{\theequation}{C.\arabic{equation}}
\setcounter{equation}{0}
\setcounter{lemma}{0}
\setcounter{section}{1}
\setcounter {subsection}{2}
\renewcommand{\thetheorem}{C,\arabic{section}}
\renewcommand{\thetheorem}{\thesubsection\arabic{section}}

\subsection*{Proof of Theorem 3.1}
\proof
\textit{ i)} In considering (2.7), it can be seen that
$\widehat{q}_{u}(x_{u}) -q_{u}(x_{u}) =I_{1}-I_{2}-I_{3}$, where $%
I_{1}=\int_{x_{u,0}}^{x_{u}}\int \frac{w_{1}\left( t_{1}\right) \Delta_{u}}{%
D_{1,u}(t_{1},t_{u}) }{\rm d} t_{u}{\rm d} t_{1}$, $I_{2}=%
\int_{x_{u,0}}^{x_{u}}\int \frac{\Delta _{1,u}D_{u}(
t_{1},t_{u}) }{D_{1,u}^{2}(t_{1},t_{u})
}w_{1}(t_{1}){\rm d} t_{u}{\rm d} t_{1}$ and
\begin{equation*}
I_{3}=\int_{x_{u,0}}^{x_{u}}\int \frac{\big ( \Delta
_{u}D_{1,u}(t_{1},t_{u}) -\Delta _{1,u}D_{u}(
t_{1},t_{u}) \big ) \Delta _{1,u}}{\big ( \Delta
_{1,u}+D_{1,u}(t_{1},t_{u}) \big ) D_{1,u}^{2}(
t_{1},t_{u}) }w_{1}(t_{1}){\rm d} t_{u}{\rm d} t_{1}.
\end{equation*}%
From Lemma 3.2, we know that $\sqrt{nh}I_{3}=O\left( \left( nh^{7+%
\frac{2+2\varepsilon }{r}}\right) ^{-\frac{1}{2}}\right) $. 

Next we
 deal with $I_{1}$. Note that $\Delta _{u}$ can be decomposed
into two terms as
\begin{equation*}
\frac{1}{n}\sum_{j=1}^{n}\left( \partial _{u}q(
\widetilde{X}_{j}) \mathbb{I}(X_{j,\bar{u}}\in
\mathcal{X}_{\bar{u}}) -D_{u}(
t_{1},t_{u}) \right)+\frac{1}{n}\sum_{j=1}^{n}\left( \partial_{u}\widehat{%
q}(\widetilde{X}_{j})\! -\partial_{u}q(
\widetilde{X}_{j})\! \right)\mathbb{I}\left( X_{j,\bar{u}}\in
\mathcal{X}_{\bar{u}}\right) .
\end{equation*}%
Then, substituting this expression into $I_{1}$, we obtain two terms, say $%
I_{11}$ and $I_{12}$, respectively. For $I_{11}$, it is included in
$\xi_{n1}$ given in Remark 3.1. This term, however, is not essential for understanding the results of this paper. 
In considering Theorem \ref{theoremA6}, we conclude that $I_{12}$
is equal to
\begin{align*}
\!\!&\!\!&\!\! \frac{1
 }{n^2
h^{d+1}}\sum_{j
 =1}^{n}
 \sum_{1\leq i\neq j\leq n} \int_{x_u}^{x_{u,0}}\!\!\!\int
e_u^{\mbox{\tiny{T}}}
Q_{jn,t_{1,u}}^{-1}%
K_{ij,t_{1,u}} A_{ij,t_{1,u}}((1-
 \alpha)
-\mathbb{I}(\varepsilon _{i}\leq -r_{i,j})
){\rm d} t_{1}{\rm d} t_{u}\cdot\mathbb{I}(X_{j,\bar{u}}\in
\mathcal{X}_{\bar{u}})
 \notag\\
\!\!&\!\!&\!\! \qquad\qquad\qquad+\,O\left(\frac{1}{h} \left( n^{1-\frac{1}{3r%
}-\frac{2\varepsilon }{3}}h^{d\left(
1+\frac{2}{3r}-\frac{1}{3r^{2}}\right) }\right)
^{-\frac{3}{4}}\right).~~~~~~~~~~~~~~~~~~~~~~~~~~~~~~~~~~~~~~~~~~~~~~~~~~~~~~~~~
\end{align*}
Denote by $I_{13}$ the first term of the expression above. Then,  using Lemma 3.1, it holds with probability one that
\begin{equation}
I_{13}=\frac{1}{n^{2}h^{d+1}}\sum_{1\leq i\neq j\leq n}\eta (
Z_{i},Z_{j}) +B_{1,u}h^{p}\big(1+o(1)\big)
  \label{rr1}
\end{equation}
 with
\begin{equation*}
\eta (Z_{i},Z_{j}) = e_{u}^{\mbox{\tiny{T}}}\int_{x_{u,0}}^{x_{u}}\int
Q_{jn,t_{1,u}}^{-1}K_{ij,t_{1,u}}A_{ij,t_{1,u}}\frac{\big (
(1-\alpha )-\mathbb{I}(\varepsilon_{i}\leq 0) \big )
w_{1}(t_{1}) }{D_{1,u}(t_{1},t_{u})
}{\rm d} t_{1}{\rm d} t_{u}\cdot\mathbb{I}(X_{j,\bar{u}}\in
\mathcal{X}_{\bar{u}}).
\end{equation*}%

Let $\psi(Z_{i},Z_{j}) =\eta (Z_{i},Z_{j})
+\eta (Z_{j},Z_{i}) $, $\psi_{i}=\mathbb{E}\psi
(z,Z_{j}) |_{z=Z_{i}}$, $\varphi
_{ij}=\psi(Z_i,Z_j) -\psi_{i}-\psi_{j}$,
$U_{n}=\frac{1}{n^{2}h^{d+1}}\sum_{1\leq i<j\leq n}\varphi
_{ij}$ and $I_{4}$ be the first term on the RHS of relationship \eqref{rr1}. Taking into account the Hoeffding decomposition of an U-statistic (see, e.g., \citet{lee90}), we rewrite $I_4$ as $
I_{4}=U_{n}+\frac{n-1}{n^{2}h^{d+1}}\sum_{i=1}^{n}\psi_{i}$. Then,
it can be inferred from Lemma A2 that
\begin{equation*}
\mathbb{E}U_{n}^{2}\leq \frac{Cn^{2}M^{2}}{\left( n^{2}h^{d+1}\right) ^{2}}%
\left( 1+\sum_{j=1}^{n-1}j^2
 \beta _{j}^{\left( r_{3}-2\right)
/r_{3}}\right) ,
\end{equation*}%
where $M=\sup_{i,j}\left(\mathbb{E}\left\vert \varphi
_{ij}\right\vert
^{r_{3}}\right) ^{\frac{1}{r_{3}}}$ for some $r_{3}>2$ satisfying that $1-%
\frac{2}{r_{3}}=\frac{2+\varepsilon }{r}$. It can be inferred from condition (B7) that $%
\sup_{i,j}\mathbb{E}\left\vert \eta (Z_{i},Z_{j})
\right\vert ^{r_{3}}\leq Ch^{d}$, and thus
\begin{equation}
\mathbb{E}U_{n}^{2}\leq \frac{C}{n^{2}h^{d\left( 1+\frac{2+\varepsilon }{r}%
\right) +2}},  \label{U2}
\end{equation}%
i.e., $\sqrt{nh}U_{n}=O_{\mathbb{P}}\left( \left( nh^{d+1+\frac{%
d(2+\varepsilon )}{r}}\right) ^{-\frac{1}{2}}\right) $. Also, from
the property of the conditional expectation, it follows that
\begin{equation}
\psi_{i}=\big ((1-\alpha )-\mathbb{I}(\varepsilon_{i}\leq
0) \big) \zeta _{i},  \label{b6}
\end{equation}%
where
\begin{equation*}
\zeta _{i}=\zeta (X_{i})
=\mathbb{E}_{i}\int_{x_{u,0}}^{x_{u}}\int \frac{w_{1}(t_{1})e_{u}^{\mbox{\tiny{T}}}
Q_{jn,t_{1,u}}^{-1}K_{ij,t_{1,u}}A_{ij,t_{1,u}}}{%
D_{1,u}(t_{1},t_{u}) }{\rm d} t_{1}{\rm d} t_{u}\cdot\mathbb{I}(
X_{j,\bar{u}}\in \mathcal{X}_{\bar{u}}) .
\end{equation*}%

Next, we  consider the asymptotic expression of
$\sum_{i=1}^{n}\psi_{i} $. Note that the domain of the covariates
of $\zeta (X_{i})$ is $A_{(u)}$, which is defined at
the beginning of Section 3. We then divide $A_{(u)}$ into a
sequence of subsets $\{A_{l}\}$ and try to get the asymptotic
representations $\{M_{l}\}$ of $\zeta(X_{i})$ on
these subsets, respectively. Let $M(X_{i})$ be the summation of all these $%
\{M_{l}\}$. Without loss of generality, we only consider
some special cases of $\{A_{l}\}$, all other left cases of
$\{A_{l}\}$ can be settled similarly. By the inequality
\eqref{Bernstein3}, we see that
\begin{equation}
\mbox{Var}\left( \frac{\left( n-1\right)
}{n^{2}h^{d+1}}\sum_{i=1}^{n}\big ( 1-\alpha -\mathbb{I}(
\varepsilon _{i}\leq 0) \big ) \big ( \zeta _{i}-M(X_{i}) \big ) \right) \leq \frac{C}{nh^{2d+2}}\left(
\mathbb{E}\left\vert \zeta (X_{i}) -M(X_{i})
\right\vert^{r_{2}}\right)^{\frac{2}{r_{2}}} \label{variance}
\end{equation}%
for $1-\frac{2}{r_{2}}=\frac{1+\varepsilon }{r}$. Let
\begin{align*}
A_{1} &=\left[ x_{u,0}-h,x_{u,0}+h\right] \times \Pi_{1\leq l\neq u\leq d}%
\left[ a_{l}-h,b_{l}+h\right], & I_{u1}=\left(
\int_{A_{1}}\left\vert \zeta
\left(x\right) -M_{1}(x) \right\vert^{r_{2}}p(x){\rm d} x\right)^{\frac{2}{r_{2}}}, \\
A_{2}& =\left[ a_{1}-h,a_{1}+h\right] \times \Pi_{2\leq l\leq
d}\left[ a_{l}-h,b_{l}+h\right], & I_{u2}=\left(
\int_{A_{2}}\left\vert \zeta \left(
x\right) -M_{2}(x) \right\vert ^{r_{2}}p(x){\rm d} x\right) ^{\frac{2}{%
r_{2}}}, \\
A_{c}& =\left[ x_{u,0}+h,x_{u}-h\right] \times \Pi_{1\leq l\neq u\leq
d}\left[ a_{l}+h,b_{l}-h\right],& I_{c}=\left(
\int_{A_{c}}\left\vert \zeta (x) -M_{c}(x) \right\vert ^{r_{2}}p(x){\rm d} x\right)^{\frac{2}{%
r_{2}}}.
\end{align*}

We now deal with each term mentioned above separately. First, we consider $%
I_{u1}$. By variable substitution, it can be seen from condition (B5) that
\begin{align*}
Q_{jn,t_{1,u}}=\int K(x)A(x)A^{\mbox{\tiny{T}}}(x) g_{1}(\widetilde{X}%
_{j}+hx){\rm d}x=g_{1}(\widetilde{X}_{j})Q+hQ^{\ast }\left. \frac{\partial g_{1}(t)}{%
\partial t}\right\vert_{t=\widetilde{X}_{j}}+O(h^{2}) .
\end{align*}%
Then, from Newman's expansion \citep[see, e.g.,][]{stewart90} it can be established that
\begin{align}
Q_{jn,t_{1,u}}^{-1}=\frac{Q^{-1}}{g_{1}(\widetilde{X}_{j})}-\frac{hQ^{-1}Q^{\ast
}}{g_{1}^{2}(\widetilde{X}_{j})}\left. \cdot\frac{\partial g_{1}(t)}{\partial t}%
\right\vert_{t=\widetilde{X}_{j}}Q^{-1}+O( h^{2}) .
\label{QE}
\end{align}%
Let $B_{k}=\Pi_{1\leq l\neq k\leq d}\left[ a_{l}+h,b_{l}-h\right]
\subset \mathbb{R}^{d-1}$,
\begin{equation*}
B_{k,i}=\left[ a_{i}-h,a_{i}+h\right] \times \Pi _{l=1,l\neq k\neq i}^{d}%
\left[ a_{l}-h,b_{l}+h\right] \subset \mathbb{R}^{d-1},
\end{equation*}%
\begin{equation}
M_{1}(z) =\int_{x_{u,0}}^{x_{u}}\int \frac{e_{u}^{\mbox{\tiny{T}}}Q^{-1}w_{1}%
(t_{1}) A\left( \frac{z-t}{h}\right) K\left( \frac{z-t}{h}%
\right) p_{t_{\bar{u}}}(t_{\bar{u}})
}{g_{1}(t)D_{1,u}(
t_{1},t_{u}) }\,\mathbb{I}(t_{\bar{u}}\in \mathcal{X}_{\bar{u}%
}){\rm d}t\cdot \mathbb{I}(z\in A_{1})  \label{M1}
\end{equation}%
and
\begin{equation}
M_{u}(z) =\frac{w_{1}(z_{1}) p_{\bar{u}}(z_{%
\bar{u}})e_{u}^{\mbox{\tiny{T}}}Q^{-1}}{g_{1}(z)D_{1,u}( z_{1},z_{u}) }%
\int_{-1}^{\frac{z_{u}-x_{u,0}}{h}}\int A(t)K(
t){\rm  d} t\cdot\mathbb{I}(z\in A_{1}).  \label{M2}
\end{equation}%
Then, by using \eqref{QE} and variable substitution, we obtain 
\begin{align*}
I_{u1} &\leq \frac{C}{nh^{2d+2}}\left\{
\int_{x_{u,0}-h}^{x_{u,0}+h}\int
\left\vert h^{d}\int_{-1}^{\frac{z_{u}-x_{u}}{h}}\int \left( -\frac{%
he_{u}^{\mbox{\tiny{T}}}Q^{-1}Q^{\ast }}{g_{1}^{2}(x-ht)}\frac{\partial g_{1}(x-ht)}{%
\partial t}Q^{-1}+O(h^{2}) \right) \right. \right. \\
&\left. \left. \frac{w_{1}( x_{1}-ht_{1}) A(t)
K(t) }{D_{1,u}(x_{1}-ht_{1},x_{u}-ht_{u}) }p_{\bar{u}%
}(x_{\bar{u}}-ht_{\bar{u}}) {\rm d}t\right\vert
^{r_{2}}p(x) {\rm d} x\right\} ^{\frac{2}{r_{2}}}\leq
\frac{C}{nh}h^{\frac{2}{r_{2}}+1}.
\end{align*}%
In view of \eqref{M1}, \eqref{M2} and variable substitution, it can
be inferred that
\begin{equation*}
\frac{1}{nh^{2+2d}}\left[ \int_{x_{u,0}-h}^{x_{u,0}+h}\!\!\!\int_{B_{u}}%
\left\vert M_{1}(z) -h^{d}M_{u}(z)
\right\vert
^{r_{2}}p(z) {\rm d} z\right] ^{\frac{2}{r_{2}}}\leq \frac{C}{nh}h^{%
\frac{2}{r_{2}}+1}
\end{equation*}%
and
\begin{equation*}
\frac{1}{nh^{2+2d}}\left[ \sum_{k=1,k\neq
u}^{d}\int_{x_{u,0}-h}^{x_{u,0}+h}\!\!\!\int_{B_{u,i}}\left\vert
M_{1}(z) -h^{d}M_{u,k}(x_{u,0},a_{k},z)\right\vert
^{r_{2}}p(z) {\rm d} z\right] ^{\frac{2}{r_{2}}}\leq
\frac{C}{nh}h^{\frac{4}{r_{2}}+1}.
\end{equation*}%
Hence, from the three inequalities above and the Cram\'er-Rao inequality, it
can be seen that $\frac{M_{u}(z) }{\sqrt{nh}}$ is
included on the RHS
of (3.3) with the remainder term $O_{\mathbb{P}}(h^{\frac{1}{r_{2}}+%
\frac{1}{2}}+h^{\frac{d}{r_{2}}-\frac{1}{2}}) $. 

Next, we consider $%
I_{u2}$. Let%
\begin{equation*}
I_{u2,1}=\left[ \int_{a_{1}-h}^{a_{1}+h}\int_{B_{1}}\left\vert
\int_{x_{u,0}}^{x_{u}}\int \frac{w_{1}(t_{1})
e_{u}^{\mbox{\tiny{T}}}Q_{jn}^{-1}A\big (\frac{z-t}{h}\big ) K\big (\frac{z-t}{h}%
\big ) p(t_{\bar{u}}) }{D_{1,u}(t_{1},t_{u}) }%
{\rm d} t\right\vert ^{r_{2}}p(z){\rm d} z\right] ^{\frac{2}{r_{2}}},
\end{equation*}%
\begin{equation*}
I_{u2,2}=\left[ \int_{a_{1}-h}^{a_{1}+h}\int_{B_{1}^{c}}\left\vert
\int_{x_{u,0}}^{x_{u}}\int \frac{w_{1}\left( t_{1}\right)
e_{u}^{\mbox{\tiny{T}}}Q_{jn}^{-1}A\big (\frac{z-t}{h}\big ) K\big (\frac{z-t}{h}%
\big ) p(t_{\bar{u}}) }{D_{1,u}(t_{1},t_{u}) }%
{\rm d} t\right\vert ^{r_{2}}p(z) {\rm d} z\right] ^{\frac{2}{r_{2}}}.
\end{equation*}%
Obviously, $I_{u2}\leq I_{u2,1}+I_{u2,2}$. According to \eqref{QE}, variable substitution and the known condition $%
w(t)=O\left( t\right) $ as $t\rightarrow a$, we obtain
\begin{align*}
\frac{I_{u2,1}}{nh^{2d+2}} &\leq \frac{C_{p}}{nh^{2}}\left\{
\int_{a-h}^{a+h}\int_{B_{1}}\left\vert
\int_{-1}^{\frac{z_{1}-a}{h}}\int w_{1}\left( z_{1}-ht_{1}\right)
\left( \frac{Q^{-1}}{g_{1}(z-ht)}+O\left(
h\right) \right) \right. \right. \\
&\left. \left. \cdot \frac{A(t) K(t) p\left(z_{%
\bar{u}}-ht_{\bar{u}}\right) }{D_{1,u}\left(
z_{1}-ht_{1},z_{u}-ht_{u}\right) }{\rm d}t\right\vert ^{r_{2}}f(
z) {\rm d}z\right\} ^{\frac{2}{r_{2}}}\leq
\frac{C_{r_{2}}}{nh}h^{1+\frac{2}{r_{2}}}.
\end{align*}%
Similarly, it can be shown that $\frac{I_{u2,2}}{nh^{2d+2}}=O\left( \frac{1}{%
nh}h^{1+\frac{2}{r_{2}}}\right) $. Therefore, $I_{u2}=O\left( \frac{1}{nh}%
h^{1+\frac{2}{r_{2}}}\right) $.

Finally, we consider $I_{c}$. Let
\begin{equation*}
M_{c}(Z_{i})=h^{d+1}\cdot e_{u}^{\mbox{\tiny{T}}}Q^{-1} \int \left. \frac{\partial }{%
\partial x^{\mbox{\tiny{T}}}}\left( \frac{w_{1}(x_{1}) p_{\bar{u}}( x_{\bar{%
u}}) }{g_{1}(x)D_{1,u}(x_{1},x_{u}) }\right)
\right\vert _{x=\widetilde{X}_{i}}t A(t)K(t){\rm d} t\cdot \mathbb{I}(
X_{i}\in A_{(u)}) .
\end{equation*}
From the fact that $e_{u}^{\mbox{\tiny{T}}}Q^{-1}\int A(t) K(t){\rm d} t=0$ and \eqref{QE}, by exploiting variable substitution
and  Taylor's expansion, we have
\begin{equation*}
\frac{I_{c}}{nh^{2d+2}}\leq \frac{Ch^{2}}{n}.
\end{equation*}%
Thus, $nh\frac{I_{c}}{n\left( h^{d+1}\right)^{2}}=O(
h^{3})$ so that $h^{\frac{3}{2}}$ appears in the remainder
term. 

Finally, we note that the contribution of term $I_{2}$ is similar to that of $I_{1}$ with $e_{u}^{\mbox{\tiny{T}}}$ replaced by $-D_{u}/D_{1}e_{1}^{\mbox{\tiny{T}}}$, which follows from comparing the expressions of $\Delta_{u}$ and $\Delta_{1,u}$. This completes the proof of   \textit{i)}.
\qed
\medskip

\textit{ii)} It can be seen from (2.8) and (2.9) that
\begin{equation}
\widehat{q}_{1}(x_{1}) -q_{1}(x_{1}) =c(
x_{1}) (\widehat{c}-c) +\left( \widehat{c}-c\right)
\int_{x_{1,0}}^{x_{1}}L_{n}(t_{1})
{\rm d}t_{1}+c\int_{x_{1,0}}^{x_{1}}L_{n}(t_{1}) {\rm d} t_{1},
\label{q1exp}
\end{equation}%
where
\begin{equation*}
L_{n}( t_{1}) =\int \left( \frac{\widehat{D}_{1,2}(
t_{1},t_{2}) }{\widehat{D}_{2}(t_{1},t_{2}) }-\frac{%
D_{1,2}(t_{1},t_{2}) }{D_{2}(t_{1},t_{2})
}\right) w_{2}(t_{2}){\rm  d} t_{2}.
\end{equation*}%
For $\widehat{c}-c$, it can be rewritten as $I_{11}+I_{12}$, where
\begin{equation*}
I_{11}=-\int \frac{L_{n}( t_{1}) w_{1}(t_{1}) }{%
\left( \int \frac{D_{1,2}(t_{1},t_{2}) }{D_{2}(
t_{1},t_{2}) }w_{2}(t_{2}){\rm d}t_{2}\right)
^{2}}{\rm d} t_{1}
\end{equation*}%
and
\begin{equation*}
I_{12}=\int \frac{w_{1}(t_{1}) L_{n}^{2}(t_{1}) }{%
\int \frac{\widehat{D}_{1,2}\left( t_{1},t_{2}\right) }{\widehat{D}_{2}\left(
t_{1},t_{2}\right) }w_{2}(t_{2}) {\rm d}t_{2}\left( \int \frac{%
D_{1,2}\left( t_{1},t_{2}\right) }{D_{2}\left( t_{1},t_{2}\right) }%
w_{2}(t_{2}) {\rm d}t_{2}\right)^{2}}{\rm d} t_{1}.
\end{equation*}%
From Lemma 3.2 it follows that
\begin{equation*}
\sqrt{nh}\left\vert I_{12}\right\vert =\sqrt{nh}O\left(
\sup_{t_{1}}L_{n}^{2}\left( t_{1}\right) \right)
=O\left( \left( nh^{7+\frac{%
2+2\varepsilon }{r}}\right) ^{-\frac{1}{2}}\right) .
\end{equation*}%
As for $I_{11}$, it is equal to
$I_{11}=-I_{111}-I_{112}+I_{113}+I_{114}$, where 
\begin{align*}
I_{111} &=\iint  \frac{w_{1}(t_{1}) w_{2}(t_{2})
\Delta_{1,2}(t_{1},t_{2}) }{D_{2}(t_{1},t_{2}) }%
{\rm d} t_{2}{\rm d} t_{1},\\
\quad I_{112}& =\iint \frac{\Delta_{1,2}(t_{1},t_{2})
\Delta _{2}(t_{1},t_{2}) }{\widehat{D}_{2}(
t_{1},t_{2}) D_{2}(t_{1},t_{2}) }w_{1}(
t_{1}) w_{2}(t_{2}) {\rm d} t_{2}{\rm d} t_{1},\\
I_{113}& =\iint \frac{\Delta _{2}(t_{1},t_{2})
D_{1,2}(t_{1},t_{2})}{D_{2}^{2}(t_{1},t_{2}) }w_{1}(t_{1}) w_{2}(
t_{2}) {\rm d} t_{2}{\rm d} t_{1},
\intertext {and}
I_{114} &=\iint \frac{\Delta _{2}^{2}(t_{1},t_{2})
D_{1,2}(t_{1},t_{2}) }{\widehat{D}_{2}(t_{1},t_{2})
D_{2}^{2}(t_{1},t_{2}) }w_{1}(t_{1}) w_{2}(t_{2}) {\rm d} t_{2}{\rm d} t_{1}.
\end{align*}
Using Lemma 3.2, we have
\begin{equation*}
I_{112}=O\left( \left( nh^{4+\frac{1+\varepsilon }{r}}\right)
^{-1}\right) ,\quad I_{114}=O\left( \left( nh^{4+\frac{1+\varepsilon
}{r}}\right) ^{-1}\right) .
\end{equation*}

We now consider $I_{111}$. Similar as before, note that $\Delta_{1,2}(t_{1},t_{2})$ can be rewritten as
\begin{eqnarray*}
&&\frac{1}{n}\sum_{j=1}^{n}\left( \partial_{1}\widehat{q}(
t_{1},t_{2},X_{j,\bar{2}}) -\partial _{1}q(t_{1},t_{2},X_{j,\bar{%
2}}) \right) \mathbb{I}(X_{j,\bar{2}}\in \mathcal{X}_{\bar{2}%
})\\
&&+\frac{1}{n}\sum_{j=1}^{n}\left( \partial _{1}q(t_{1},t_{2},X_{j,%
\bar{2}})\,\mathbb{I}(X_{j,\bar{2}}\in
\mathcal{X}_{\bar{2}}) -D_{1,2}(t_{1},t_{2})
\right) .
\end{eqnarray*}%
Substituting this into $I_{111}$, we derive two terms from this, say $%
I_{115} $ and $I_{116}$. Then, $I_{116}$ is included in
(3.5). This term is not essential for understanding the proof in this case, and hence has been omitted. As for $I_{115} $, by using Theorem \ref{theoremA6} and Lemma 3.1, it can be shown that
\begin{eqnarray*}
I_{115}\! \!&\!\!=\!\!&\!\!\frac{1}{n^{2}h^{d+1}}\sum_{j=1}^{n}\iint
e_{1}^{\mbox{\tiny{T}}}Q_{jn}^{-1}\sum_{i=1}^{n}K_{ij,t_{1,2}}A_{ij,t_{1,2}}\big (\left(
1-\alpha \right) -\mathbb{I}(\varepsilon _{i}\leq 0) \big ) \frac{%
w_{2}(t_{2}) w_{1}(t_{1}) }{D_{2}(t_{1},t_{2})}%
{\rm d} t_{1}{\rm d} t_{2} \\
&&+\,O\left( \frac{1}{h}\left( n^{1-\frac{1}{3r}-\frac{2\varepsilon }{3}%
}h^{d\left( 1+\frac{2}{3r}-\frac{1}{3r^{2}}\right) }\right) ^{-\frac{3}{4}%
}+h^{p-1}\right) .
\end{eqnarray*}%
According to the previous analysis and by the Hoeffding decomposition the leading term of $I_{115}$ is equal to
\begin{equation*}
\frac{1}{nh^{d+1}}\sum_{i=1}^{n}\big ( \left( 1-\alpha \right)
-\mathbb{I}(
\varepsilon _{i}\leq 0) \big )\mathbb{E}_{i}\iint\frac{%
w_{1}(t_{1}) w_{2}(t_{2})
e_{1}^{\mbox{\tiny{T}}}Q_{jn}^{-1}K_{ij,t_{1,2}}A_{ij,t_{1,2}}\mathbb{I}( X_{j,\bar{2}%
}\in \mathcal{X}_{\bar{2}}) }{D_{2}(t_{1},t_{2})}{\rm d} t_{1}{\rm d} t_{2}.
\end{equation*}%
And then, by a similar method as the proof of part \textit{i)}, it can be shown
that the
leading term of $I_{115}$ is equal to%
\begin{eqnarray*}
\sum_{i=1}^{n}\frac{\left( \left( 1-\alpha \right)
-\mathbb{I}(\varepsilon _{i}\leq 0) \right) w_{1}(
X_{i,1}) w_{2}(X_{i,2}) p(X_{i,\bar{2}}) e_{1}^{\mbox{\tiny{T}}}Q^{-1}}{%
nhD_{2}(X_{i,1},X_{i,2})g_{1}(X_{i})}\! 
\sum_{3\leq k\leq d}\Big(f_{k}\Big (\frac{X_{i,k}-b_{k}}{h}%
\Big ) -f_{k}\Big (\frac{X_{i,l}-a_{l}}{h}\Big ) \Big ) .
\end{eqnarray*}%

Analogously, we can deal with $I_{113}$. Its leading term is given
by
\begin{eqnarray*}
&&\!\!\sum_{i=1}^{n}\frac{\big ( \left( 1-\alpha \right)
-\mathbb{I}(\varepsilon _{i}\leq 0) \big ) w_{1}(X_{i,1}) w_{2}(X_{i,2}) p(
X_{i,\bar{2}}) D_{1,2}(X_{i,1},X_{i,2})
e_{2}^{\mbox{\tiny{T}}}Q^{-1}}{nhg_{1}(X_{i})D_{2}^{2}(X_{i,1},X_{i,2})} \\
&&\cdot \sum_{3\leq k\leq d}\Big (f_{k}\Big (\frac{X_{i,k}-b_{k}}{h}%
\Big ) -f_{k}\Big (\frac{X_{i,l}-a_{l}}{h}\Big ) \Big ) .
\end{eqnarray*}%
Let $I_{2}=c\int_{x_{1,0}}^{x_{1}}L_{n}(t_{1}) {\rm d} t_{1}$,%
\begin{align*}
I_{21} &=\int_{x_{1,0}}^{x_{1}}\! \int\! \frac{w_{2}(
t_{2}) \Delta
_{1,2}(t_{1},t_{2}) }{D_{2}(t_{1},t_{2}) }%
{\rm d} t_{2}{\rm d} t_{1}\mbox{, }\quad I_{22}=\int_{x_{1,0}}^{x_{1}}\int
\frac{w_{2}(t_{2}) \Delta _{2}(t_{1},t_{2})
D_{1,2}(t_{1},t_{2}) }{D_{2}^{2}(
t_{1},t_{2}) }{\rm d} t_{2}{\rm d} t_{1}\\
\intertext{and}
I_{23} &\!=\!\int_{x_{1,0}}^{x_{1}}\!\!\int \!\!\left( \Delta
_{12}(t_{1},t_{2}) \!-\!\Delta_{2}(
t_{1},t_{2}) \frac{D_{1,2}(t_{1},t_{2})
}{D_{2}(t_{1},t_{2}) }\right) \!\frac{\Delta _{2}(
t_{1},t_{2}) w_{2}(t_{2}) }{\left( \Delta_{2}(t_{1},t_{2}) +D_{2}(t_{1},t_{2})
\right) D_{2}(t_{1},t_{2}) }{\rm d} t_{2}{\rm d} t_{1}.
\end{align*}%
Then, $I_{2}=\left( I_{21}-I_{22}-I_{23}\right) c$. Similarly, the
leading term of $I_{21}$ is equal to
\begin{equation*}
\frac{1}{nh^{d+1}}\sum_{i=1}^{n}\left( \big ( 1-\alpha \right)
\mathbb{I}(\varepsilon _{i}\leq 0) \big )\mathbb{E}%
_{i}\!\!\!\int_{x_{1,0}}^{x_{1}}\!\!\int \!\frac{w_{2}(
t_{2})
e_{1}^{\mbox{\tiny{T}}}Q^{-1}K_{ij,t_{1,2}}A_{ij,t_{1,2}}\mathbb{I}(
X_{j,\bar{2}}\in \!\mathcal{X}_{\bar{2}}) }{D_{2}(
t_{1},t_{2}) }{\rm d} t_{2}{\rm d}t_{1}.
\end{equation*}%
By an analogous method, we can obtain the asymptotic representations of $%
I_{21}$ and $I_{22}$, which are included in (3.4).   This
completes the proof of  \textit{ii)}.
\endproof%

\noindent
\textbf{Remark 3.}
The argument of the proof of Theorem 3.1 uses $x_{u}\geq u_{u,o}$ at some steps, for instance to derive the integration boundaries in \eqref{M2}. A similar argument is possible for $x_{u}<x_{u,0}$.  

\subsection*{Proof of Theorem 4.1}  
\begin{proof}
We first consider the case $2\leq u\leq d$. Let $I_{1}$, $I_{2}$ and
$I_{3}$
be the same notations as those in part \textit{i)} of Theorem 3.1 and then $\widehat{q}%
_{u}(x_{u}) -q_{u}(x_{u}) $ is equal to $%
I_{1}-I_{2}-I_{3}$. From Lemma 3.2, we know that $\sqrt{nh}%
\sup_{x_{u}\in \left[ a_{u},b_{u}\right] }\left\vert
I_{3}\right\vert
=O\left( \left( nh^{7+\frac{2+2\varepsilon }{r}}\right) ^{-\frac{1}{2}%
}\right) $. For brevity, let $\gamma _{n}=\sqrt{\frac{\log n}{nh^{1+\frac{%
1+\varepsilon }{r}}}}$. Next, we only consider the uniform
convergence rate of $I_{1}$, since the methodology to deal with
$I_{1}$ and $I_{2}$ is almost
completely the same. In view of Theorem \ref{theoremA6} and Lemma 3.1, we know that $%
I_{1}=\phi \left( x_{u}\right) +o\left( \gamma _{n}\right)$, where
$\phi
(x_{u})=\int_{x_{u,0}}^{x_{u}}\psi _{n}(t_{u}){\rm d}t_{u}$, $\psi_{n}(t_{u})=%
\frac{1}{n^{2}h^{d+1}}\sum_{j=1}^{n}\sum_{i=1,i\neq j}^{n}\varphi
_{ij}(t_{u})$ and
\begin{equation*}
\varphi_{ij}(t_{u})=\big ( \left( 1-\alpha \right) -\mathbb{I}(
\varepsilon _{i}\leq 0) \big ) \int
\frac{e_{u}^{\mbox{\tiny{T}}}Q_{j,n}^{-1}\left(
t_{1},t_{u}\right) w_{1}(t_{1}) K_{ij,t_{1,u}}A_{ij,t_{1,u}}}{%
D_{1,u}\left( t_{1},t_{u}\right) }{\rm d} t_{1}.
\end{equation*}%
In the sequel, we only need to prove that
\begin{equation}
\sup_{x_{u}\in \left[ a_{u},b_{u}\right] }\phi \left( x_{u}\right)
=O(\gamma _{n}) .  \label{b7}
\end{equation}%
As usual, we divide the interval $\left[ a_{u},b_{u}\right] $ into a
sequence of disjoint subintervals, the length of which is equal to
$\ell _{n} $. Without loss of generality, we assume that $(b_u-a_u)\ell
_{n}^{-1}$ is an integer, and $\left\{ x_{u,k},k=1,\ldots,C\ell
_{n}^{-1}\right\} $ are the corresponding grid points. Then,
\begin{eqnarray}
&&\mathbb{P}\Big ( \sup_{x_{u}\in \left[ a_{u},b_{u}\right] }\left\vert
\phi \left( x_{u}\Big ) \right\vert \geq \gamma _{n}\right) \leq
\sum_{1\leq
k\leq \ell _{n}^{-1}}\mathbb{P}\Big ( \sup_{t_{u}\in \left[ x_{u,k},x_{u,k+1}%
\right] }\left\vert \psi_{n}(t_{u}) -\psi_{n}(
x_{u,k}) \right\vert \geq \frac{\gamma _{n}}{3\ell
_{n}}\Big)  \notag
\\
&&\qquad \qquad \qquad +\sum_{1\leq k\leq \ell
_{n}^{-1}}\mathbb{P}\Big (
\left\vert \psi _{n}(x_{u,k}) \right\vert \geq \frac{\gamma _{n}%
}{3\ell _{n}}\Big ) +\sum_{1\leq k\leq \ell _{n}^{-1}}\mathbb{P}\left(
\left\vert \phi (x_{u,k}) \right\vert \geq \frac{\gamma _{n}}{3}%
\right) .  \label{b4}
\end{eqnarray}%

We now consider the super bound of the first term on the RHS of 
\eqref{b4}. From condition (B2), there exists another kernel function $%
K_{-u}\left(\cdot\right)$ defined on $\mathbb{R}^{d-1}$ such that
\begin{equation}
\left|\frac{\partial K(t)}{\partial t_{u}}\right|\leq
CK_{-u}\left(t_{-u} \right),  \label{444}
\end{equation}
where $t_{-u}\in\mathbb{R}^{d-1}$ is obtained from $t$ by deleting
its $u$th
component. Similar as the proof of inequality \eqref{Qin} in Appendix \ref{app:D}, we know that $%
\left\Vert Q_{j,n}^{-1}(t_{1},t_{u}) -Q_{j,n}^{-1}(
t_{1},x_{u,k}) \right\Vert \leq C\ell _{n}$. Also, from condition (B8), it
can be seen that $\left\vert D_{1,u}^{-1}(t_{1},t_{u})
-D_{1,u}^{-1}(t_{1},x_{u,k}) \right\vert \leq C\ell
_{n}$. Therefore, if $\left\vert X_{i,u}-t_{u}\right\vert \leq h$
and $\left\vert
X_{i,u}-x_{u,k}\right\vert \leq h$ occur simultaneously, then from \eqref{444}, one of the leading terms for the bound of $\left\vert \varphi
_{ij}(t_{u})-\varphi_{ij}(x_{u,k})\right\vert $ is $\frac{C\ell_{n}}{h}%
W_{ij}$, where
\begin{equation*}
W_{ij}=\mathbb{I}\left( \left\vert X_{i,u}-x_{u,k}\right\vert \leq
h\right) \int
K_{-u}\left( \frac{X_{i,1}-t_{1}}{h},\frac{X_{i,\bar{u}}-X_{j,\bar{u}}}{h}%
\right) w_{1}(t_{1}) {\rm d} t_{1}.
\end{equation*}%
We have not specified the other terms since they can be dealt with
analogously. If $\left\vert X_{i,u}-t_{u}\right\vert \leq h$ and
$\left\vert X_{i,u}-x_{u,k}\right\vert >h$ occur simultaneously, or
$\left\vert X_{i,u}-t_{u}\right\vert >h$ and $\left\vert
X_{i,u}-x_{u,k}\right\vert \leq h$ occur simultaneously, one of the
leading terms of the bound of \[
\left\vert \varphi_{ij}(t_{u})-\varphi _{ij}(x_{u,k})\right\vert \] is $%
CU_{ij}$, where%
\begin{align*}
U_{ij}&=\mathbb{I}(h-\ell _{n}\leq \left\vert
X_{i,u}-x_{u,k}\vert
\leq h+\ell _{n}\right) \\
&\quad\cdot\int \frac{\left\Vert Q_{j,n}^{-1}\left( t_{1},x_{u,k}\right)
\right\Vert }{D_{1,u}\left( t_{1},x_{u,k}\right) }K_{-u}\left( \frac{%
X_{i,1}-t_{1}}{h},\frac{X_{i,\bar{u}}-X_{j,\bar{u}}}{h}\right)
w_{1}(t_{1}) {\rm d} t_{1}.
\end{align*}%
Therefore, according to the two cases above, the following two
summations, say $I_{31}$ and $I_{32}$, respectively,
\begin{equation*}
\frac{C}{n^{2}h^{d+1}}\sum_{j=1}^{n}\sum_{i=1,i\neq j}^{n}\left(
\frac{\ell _{n}}{h}W_{ij}+U_{ij}\right) ,
\end{equation*}%
are one of the two leading terms of the bound of $\sup_{t_{u}\in
\left[ x_{u,k},x_{u,k+1}\right] }\left\vert \psi_{n}(
t_{u}) -\psi_{n}(x_{u,k}) \right\vert $.

Similar as the proof of Theorem 3.1, $I_{31}$ can be rewritten
as
\begin{eqnarray*}
&&L_{1}+\frac{\left( n-1\right) \ell
_{n}}{n^{2}h^{d+2}}\sum_{j=1}^{n}\big (\mathbb{E}_{j}(
W_{ij}+W_{ji}) -\mathbb{E}\left(\mathbb{E}_{j}(
W_{ij}+W_{ji}) \right) \big )  -\frac{\left( n-1\right) \ell _{n}}{%
2n h^{d+2}}\mathbb{E}\big (\mathbb{E}_{j}(W_{ij}+W_{ji})
\big ) ,
\end{eqnarray*}%
where%
\begin{align*}
L_{1}=\frac{ \ell _{n}}{n^{2}h^{d+2}}%
\sum_{1\leq i<j\leq n}\left\{\left(
W_{ij}+W_{ji}\right)-\mathbb{E}_{j}(W_{ij}+W_{ji})
-\mathbb{E}_{i}(W_{ij}+W_{ji})+\mathbb{E}\big(\mathbb{E}_{i}(
W_{ij}+W_{ji}) \big )\right\}
\end{align*}%
is a degenerated U-statistic transformed from $I_{31}$. Next, we choose $%
\ell _{n}=h\gamma _{n}^{\frac{1}{2}}$\ and\ $q_{1}=\frac{nh}{\log
n}$. Noting that $\mathbb{E}W_{ij}^{sk}\leq Ch^{sk+d-1}$, by Lemma \ref{lemmaA2},
we have that
\begin{equation*}
\frac{1}{\ell _{n}}\cdot\frac{\ell _{n}^{2}}{\gamma_{n}^{2}}\cdot\mathbb{E}L_{1}^{2}\leq
\frac{C}{\ell_{n}}\cdot\frac{\ell_{n}^{2}}{\gamma_{n}^{2}}\cdot n^{2}\left( \frac{%
\ell _{n}}{n^{2}h^{d+2}}\right) ^{2}\cdot \left( h^{2s+d-1}\right)^{\frac{1}{s}}=%
\frac{C\ell_{n}^{3}}{\gamma_{n}^{2}n^{2}h^{d+3}}\cdot h^{-\left(
d-1\right) \frac{3+\varepsilon}{r} }\leq \frac{1}{n(\log n)^2},
\end{equation*}%
where $\frac{1}{s}=1-\frac{k+1+\varepsilon }{r}$. Since $\left\vert
\mathbb{E}_{j}W_{ji}\right\vert \leq Ch^{d-1}$ and $\mathbb{E}\left\vert
\mathbb{E}_{j}W_{ji}\right\vert^{r_{1}}\leq Ch^{r_{1}\left( d-1\right)
+1}$, by using Lemma \ref{lemmaA1}, we have
\begin{equation}
\frac{1}{\ell_{n}}\mathbb{P}\left( \frac{\ell_{n}}{nh^{d+2}}%
\sum_{j=1}^{n}\big (\mathbb{E}_{j}W_{ji}-\mathbb{E}(
\mathbb{E}_{j}W_{ji})
\big ) \geq B\frac{\gamma _{n}}{\ell_{n}}\right) \leq \frac{1}{\ell_{n}}\cdot%
n^{-CB}+\frac{1}{\ell_{n}}\cdot\frac{n}{q_{1}}\cdot\beta \left(
q_{1}\right)\leq \frac{1}{n(\log n)^2} ,  \label{b3}
\end{equation}%
where we use the fact $n\left(\mathbb{E}\left\vert
\mathbb{E}_{j}W_{ji}\right\vert
^{r_{1}}\right) ^{\frac{2}{r_{1}}}\leq $ $q_{1}\cdot Ch^{d-1}\cdot \frac{%
\gamma _{n}}{\ell_{n}}\cdot\frac{nh^{d+2}}{\ell _{n}}$ with $\frac{2}{r_{1}}=1-%
\frac{1+\varepsilon }{r}$. Also, it holds that $\frac{\ell _{n}}{nh^{d+2}}%
\sum_{j=1}^{n}\mathbb{E}\left(\mathbb{E}_{j}(W_{ij}+W_{ji})
\right) \leq \frac{C\gamma_{n}}{\ell_{n}}$. Similarly,
\begin{equation*}
\frac{1}{\ell_{n}}\mathbb{P}\left( \frac{\ell_{n}}{nh^{d+2}}%
\sum_{j=1}^{n}\left(\mathbb{E}_{j}W_{ij}-\mathbb{E}\left(
\mathbb{E}_{j}W_{ij}\right) \right) \geq B\frac{\gamma_{n}}{\ell_{n}}\right)
\leq \frac{C}{n(\log n)^2}.
\end{equation*}

Next, we focus on $I_{32}$. Adopting the same method as
that of $I_{31}$, let $L_{2}$ be the corresponding degenerated
U-statistic resulting
from $I_{32}$. For $k=4$, since $\mathbb{E}_{ij}^{ks}\leq C\frac{\ell_{n}}{h}%
h^{ks+d-2}$, it follows that
\begin{equation*}
\frac{1}{\ell_{n}}\cdot\frac{\ell_{n}^{k}}{\gamma
_{n}^{k}}\cdot\mathbb{E}_{2}^{k}\leq
\frac{C}{\ell_{n}}\cdot\frac{\ell_{n}^{k}}{\gamma _{n}^{k}}\cdot n^{k}\cdot \left( \frac{1}{%
n^{2}h^{d+1}}\right)^{k}\cdot \left(\frac{\ell_{n}}{h}h^{ks+d-2}\right)^{1-%
\frac{k+1+\varepsilon }{r}}\leq \frac{C}{n(\log n)^2}.
\end{equation*}%
Because $\left\vert \mathbb{E}_{i}U_{ij}\right\vert \leq Ch^{d-1}$ and $%
\mathbb{E}\left\vert \mathbb{E}_{i}U_{ij}\right\vert ^{r_{1}}\leq
Ch^{r_{1}\left( d-1\right) }\frac{\ell _{n}}{h}$, it holds that
\begin{equation*}
\frac{1}{\ell _{n}}\mathbb{P}\left\{ \left\vert \frac{1}{nh^{d+1}}%
\sum_{i=1}^{n}\left(\mathbb{E}_{i}U_{ij}-\mathbb{E}\left(
\mathbb{E}_{i}U_{ij}\right)
\right) \right\vert \geq \frac{B\gamma _{n}}{\ell_{n}}\right\} \leq \frac{C%
}{n(\log n)^2},
\end{equation*}
where the fact%
\begin{equation*}
n\left(\mathbb{E}\left\vert \mathbb{E}_{i}U_{ij}\right\vert^{r_{1}}\right) ^{\frac{%
2}{r_{1}}}\leq Cq_{1}\cdot h^{d-1}\cdot \frac{Bnh^{d+1}\gamma_{n}}{\ell_{n}%
}
\end{equation*}%
is used. Also, it holds that
$\frac{1}{nh^{d+1}}\sum_{i=1}^{n}\mathbb{E}\big (\mathbb{E}_{i}(
U_{ij}+U_{ji}) \big ) \leq \frac{C\gamma_{n}}{\ell_{n}}$.
Similarly,
\begin{equation*}
\frac{1}{\ell _{n}}\mathbb{P}\left( \left\vert \frac{1}{nh^{d+1}}%
\sum_{i=1}^{n}\big ( \mathbb{E}_{i}U_{ji}-\mathbb{E}(
\mathbb{E}_{i}U_{ji})\big ) \right\vert \geq B\frac{\gamma_{n}}{\ell_{n}}\right) \leq \frac{C}{%
n(\log n)^2}.
\end{equation*}

We now consider the last term on the RHS of \eqref{b4}.
Here, we assume that the two notations $U_n$ and $\psi_{i}$ are the
same as that appears in part \textit{i)} of Theorem 3.1, with $x_u$ replaced
by $x_{u,k}$. Analogous to \eqref{U2}, for $k=8$, it can be inferred
that
\begin{equation*}
\frac{1}{\gamma_{n}^{k}\ell_{n}}\cdot \mathbb{E}U_{n}^{k}\leq
\frac{1}{\gamma_{n}^{k}\ell_{n}}\cdot
\frac{n^kh^{d\left(1-\frac{k+1+\varepsilon}{r}\right) }}{%
(n^{2}h^{d+1})^{k}}\leq \frac{C}{n(\log n)^2}.
\end{equation*}%
Noting that $\left\vert \psi_{i}\right\vert \leq Ch^{d}$ and taking $q=\sqrt{%
\frac{nh^{1-\frac{1+\varepsilon _{0}}{r}}}{\log n}}$, it follows
that
\begin{equation*}
\frac{1}{\ell_{n}}\mathbb{P}\left( \left\vert \frac{1}{nh^{d+1}}%
\sum_{i=1}^{n}\psi_{i} \right\vert \geq \gamma _{n}\right) \leq \frac{C}{%
n(\log n)^2}.
\end{equation*}

Finally, we  consider the second term on the RHS of \eqref{b4}. As previously, let $L_{3}$ be the degenerated U-statistic
related to this term. Then, we have
\begin{equation*}
\frac{1}{\ell_{n}}\cdot \frac{\ell_{n}^{2}}{\gamma_{n}^{2}}\cdot
\mathbb{E}L_{3}^{2}\leq \frac{C}{\ell _{n}}\cdot \frac{\ell_{n}^{2}}
{\gamma_{n}^{2}}\cdot n^{2}\left(
\frac{1}{n^{2}h^{d+1}}\right) ^{2}\cdot \left( h^{d}\right)
^{1-\frac{2+\varepsilon }{r_{3}}}=\frac{C\ell_{n}}{\gamma
_{n}^{2}}\cdot \frac{1}{n^{2}h^{d+2+\frac{d\left( 2+\varepsilon \right) }{%
r_{3}}}}\leq \frac{C}{n\left( \log n\right) ^{2}}.
\end{equation*}%
Also, since $\left\vert \mathbb{E}_{i}\varphi_{ij}(x_{u,k})\right\vert
\leq Ch^{d-1}$ and $\mathbb{E}\left\vert \mathbb{E}_{i}\varphi
_{ij}(x_{u,k})\right\vert
^{r_{2}}\leq Ch^{r_{2}\left( d-1\right) +1}$, it follows that%
\begin{equation*}
\frac{1}{\ell_{n}}\mathbb{P}\left( \left\vert \frac{1}{nh^{d+1}}%
\sum_{i=1}^{n}\mathbb{E}_{i}\varphi_{ij}(x_{u,k})\right\vert \geq
\frac{B\gamma_{n}}{\ell_{n}}\right) \leq \frac{C}{n\left( \log
n\right) ^{2}}.
\end{equation*}%
Hence, \eqref{b7} holds.

Next, we  prove (3.6) for the  $u=1$. Here, we adopt
the same notations as in the proof of \textit{ii)} of Theorem 3.1. From
\eqref{q1exp}, it
suffices to show that $\widehat{c}-c=O(\gamma_{n}) $ and $%
\sup_{x_{1}\in \lbrack a_{1},b1]}\left\vert
\int_{x_{1,0}}^{x_{1}}L_{n}(t_{1}){\rm d}t_{1}\right\vert\newline
=O(\gamma_{n}) $. The first one can be inferred from
the proof of \textit{ii)} of
Theorem 3.1. Hence, we  just consider the second one. Recalling that $%
\frac{1}{c}I_{2}=I_{21}-I_{22}-I_{23}$, we only need to prove that $%
I_{21}=O(\gamma_{n})$ and $I_{22}=O(\gamma_{n}) $%
. Now, analogous to the previous proof for the  $2\leq u\leq d$,
these two results follow immediately.
\end{proof}

\subsection*{Proof of Theorem 5.1}  
\proof%
\textit{i)} For convenience of notation, let $\tau_{n}=\sqrt{\frac{\log n}{%
nh_{G}^{1+\frac{1+\varepsilon }{r}}}}$, $\delta =\frac{1}{\sqrt{nh^{1+\frac{%
1+\varepsilon }{r}}}}$ and $G_{n}(v) $ is the $\alpha
$th conditional quantile of $F_{n}(y|v)$. First, we show that with
probability one
\begin{equation}
\sup_{y\in \mathbb{R}}\left\vert
\widehat{F}_{n}(y|v)-F_{n}(y|v)\right\vert =O\left(
\frac{n^{-\varepsilon_0}}{\sqrt{nh_G}} \right) . \label{rate1}
\end{equation}%
In fact, this conclusion can be established from the following two
relationships
\begin{equation}
\sup_{y\in \mathbb{R}}\left\vert \sum_{j=1}^{n}\left( K_{G}\left( \frac{v-%
\widehat{q}_{0}(X_{j}) }{h_{G}}\right) -K_{G}\left( \frac{%
v-q_{0}(X_{j}) }{h_{G}}\right) \right) \mathbb{I}\left(
Y_{i}\leq y\right) \right\vert =O\left(
\sqrt{n^{1-2\varepsilon_0}h_G}  \right) \label{KRate2}
\end{equation}%
and
\begin{equation}
I=\sum_{j=1}^{n}\left( K_{G}\left( \frac{v-\widehat{q}_{0}\left( X_{j}\right) }{%
h_{G}}\right) -K_{G}\left( \frac{v-q_{0}\left(X_{j}\right)
}{h_{G}}\right) \right) =O\left( \sqrt{n^{1-2\varepsilon_0}h_G}
\right) . \label{KRate1}
\end{equation}%
Since the proofs of these two relationships are completely the same,
we only verify the\ second relationship \eqref{KRate1}. According to
Taylor's expansion, there exist $0\leq \lambda _{j}\leq 1,$
$j=1,\ldots,n,$ such that $I$ can be written as
\begin{equation*}
\sum_{j=1}^{n}K_{G}^{\prime }\left( \frac{v-q_{0}\left(X_{j}\right) }{h_{G}}%
\right) \frac{\widehat{q}_{0}(X_{j}) -q_{0}(X_{j}) }{%
h_{G}}+\sum_{j=1}^{n}K_{G}^{\prime \prime }\left(
\frac{v-q_{0}(X_{j}) +\bar{\theta}_{j}}{h_{G}}\right)
\frac{\big (\widehat{q}_{0}(X_{j}) -q_{0}(
X_{j}) \big ) ^{2}}{h_{G}^{2}},
\end{equation*}%
where $\bar{\theta}_{j}=\lambda _{j}\big ( \widehat{q}_{0}(
X_{j}) -q_{0}(X_{j}) \big )$. Let $I_{1}\
$and $I_{2}$ denote the two terms above, respectively. According to
variable substitution, integration by parts, conditions (C1) and (C2), it can be
shown subsequently that
\begin{align}
-\mathbb{E}\left\vert K_{G}^{\prime \prime }\left( \frac{v-q_{0}\left(
X_{j}\right) }{h_{G}}\right) \right\vert
&=\!h_{G}\!\int_{-1}^{1}K_{G}^{\prime \prime }\left( t\right)
f_{q_{0}}\left( v-th_{G}\right)
{\rm d} t=h_{G}^{2}\int_{-1}^{1}f_{q_{0}}^{\prime }\left( v-th_{G}\right)
K_{G}^{\prime }\left( t\right) {\rm d} t  \notag \\
&= h_{G}^{3}\int_{-1}^{1}f_{q_{0}}^{\prime \prime }\left(
v-th_{G}\right) K_{G}\left( t\right) {\rm d} t=O(h_{G}^{3}).  \label{b10}
\end{align}%
Likewise, it follows that
\begin{equation}
\mathbb{E}\left\vert K_{G}^{\prime }\left( \frac{v-q_{0}\left(
X_{j}\right) }{h_{G}}\right) \right\vert =O( h_{G}^{2}) .
\label{b19}
\end{equation}%
By the standard method of the proof of Strong Law of Large Numbers (SLLN) we have
\begin{equation}
\sum_{j=1}^{n}\left( \left\vert K_{G}^{\prime }\left(
\frac{v-q_{0}\left(X_{j}\right) }{h_{G}}\right) \right\vert
-\mathbb{E}\left\vert K_{G}^{\prime }\left( \frac{v-q_{0}\left(
X_{j}\right) }{h_{G}}\right)
\right\vert \right) =O\left( \sqrt{nh_{G}^{1-\frac{1+\varepsilon }{r}}}%
\right) .  \label{b8}
\end{equation}%
Thus, from the two points above, Lemma 4.1 and condition (C3), it can
be inferred that $\left\vert I_{1}\right\vert =O(nh_{G}\delta
)$. By using conditions (C1) and (C3), for some fixed constant $\varepsilon
_{0}>0$, we can see that
\begin{equation}
\left\vert I_{2}\right\vert \leq \sum_{j=1}^{n}\left( \left\vert
K_{G}^{\prime \prime }\left( \frac{v-q_{0}(X_{j}) }{h_{G}}%
\right) \right\vert +\mathbb{I}\left( \frac{\left\vert v-q_{0}\left(
X_{j}\right) \right\vert }{h_{G}}\leq 1+\varepsilon _{0}\right)
\frac{\delta }{h_{G}}\right) \left( \frac{\delta }{h_{G}}\right)
^{2}\equiv I_{21}+I_{22} \label{b16}
\end{equation}%
holds for $n$ sufficiently large. Analogous to the proof of relationship
\eqref{b8}, by using Lemma \ref{lemmaA1}, we have
\begin{equation}
\sum_{j=1}^{n}\left[\mathbb{I}\left( \frac{\left\vert v-q_{0}(
X_{j}) \right\vert }{h_{G}}\leq 1+\varepsilon_{0}\right)
-\mathbb{E}\left( \frac{\left\vert v-q_{0}(X_{j})
\right\vert
}{h_{G}}\leq 1+\varepsilon _{0}\right) \right] =O\left( \left( nh_{G}^{1-%
\frac{1+\varepsilon }{r}}\right) ^{\frac{1}{2}}\right) ,  \label{b9}
\end{equation}%
where we use the facts $q=\left( nh_{G}^{1-\frac{1+\varepsilon
}{r}}/\log n\right) ^{\frac{1}{2}}$ and $nq^{-(r+1)}\rightarrow 0$.
From this, condition (C3) and the fact
\begin{equation*}
n\mathbb{E}\left( \frac{\left\vert v-q_{0}(X_{j})
\right\vert }{h_{G}}\leq 1+\varepsilon_{0}\right) =O(nh_{G}) ,
\end{equation*}%
we know that $I_{22}=o\left( nh_{G}\delta \right) .$ As for
$I_{21}$,
similar to the proof of \eqref{b9}, we can establish%
\begin{equation*}
\sum_{j=1}^{n}\left( \left\vert K_{G}^{\prime \prime }\left( \frac{%
v-q_{0}(X_{j}) }{h_{G}}\right) \right\vert
-\mathbb{E}\left\vert
K_{G}^{\prime \prime }\left( \frac{v-q_{0}\left(X_{j}\right) }{h_{G}}%
\right) \right\vert \right) =O\left( \sqrt{nh_{G}^{1-\frac{1+\varepsilon }{r}%
}}\right) .
\end{equation*}%
From this, \eqref{b10} and condition (C3), we see that $I_{21}=o(
nh_{G}\delta
)$, where we use the fact $\frac{\delta \sqrt{\log n}}{h_{G}^{2}\sqrt{%
nh_{G}^{1+\frac{1+\varepsilon }{r}}}}\rightarrow 0$. Hence,
\eqref{KRate1} holds.

Next, in our  of mixing processes, we adopt the standard
procedure of the conditional quantiles \citep[see, e.g.,][Lemma 2.2]{mehra92},
to prove that
\begin{equation}
\sup_{\left\vert y-G\left( v\right) \right\vert \leq C\tau
_{n}}\left\vert F_{n}(y|v)-F_{n}(G\left( v\right) |v)-F(y|v)+\left(
1-\alpha \right)
\right\vert =O\left( \tau _{n}^{\frac{3}{2}-\frac{1+\varepsilon }{2r}%
}\right).  \label{FRate1}
\end{equation}%
Divide the interval $\left\vert y-G\left( v\right) \right\vert \leq
C\tau_{n}$ into a sequence of subintervals with length $\tau_{n}^{\frac{3}{2}-%
\frac{1+\varepsilon }{2r}}$. Let $\eta_{r}$, $r=1,2,\ldots,d_{n}$,
be the
corresponding grid points. Then, $d_{n}=C\tau_{n}^{-\frac{1}{2}+\frac{%
1+\varepsilon }{2r}}$. Then, it holds that
\begin{eqnarray}
\!\!\!\!\!\!\!\!\!\!&\!\!&\!\!\sup_{\left\vert y-G(v)
\right\vert \leq C\tau_{n}}\left\vert F_{n}(y|v)-F_{n}(G(
v)
|v)-F(y|v)+\left( 1-\alpha \right) \right\vert  \notag \\
\!\!\!\!\!\!\!\!\!\!&\!\!\leq \!\!&\!\!\max_{\left\vert r\right\vert
\leq Cd_{n}}\left\vert F_{n}(\eta_{r}|v)\! - \!F(\eta_{r}|v)\! -\!
F_{n}(G(v) |v)\! + \!\left( 1-\alpha \right) \right\vert
\!\!+\!\!\max_{\left\vert r\right\vert \leq Cd_{n}}\left\vert
F(\eta_{r+1}|v)-F(\eta_{r}|v)\right\vert \!.   \qquad\label{b13}
\end{eqnarray}%
From condition (C4), it is easy to see that
\begin{equation}
\max_{\left\vert r\right\vert \leq Cd_{n}}\left\vert
F(\eta_{r+1}|v)-F(\eta_{r}|v)\right\vert =
O\left( \tau {n}^{\frac{3}{2}-\frac{1+\varepsilon }{2r}%
}\right) .  \label{b14}
\end{equation}%
Let $\xi_{i,r,1}=K\left( \frac{v-q_{0}(X_{i})}{h_{G}}\right) \left(
\mathbb{I}(Y_{i}\leq \eta _{r})-\mathbb{I}(Y_{i}\leq G(v)\right) $ and $%
b_{n}=Bnh_{G}\tau_{n}^{\frac{3}{2}-\frac{1+\varepsilon }{2r}}$. By
using
Lemma \ref{lemmaA1}, we have%
\begin{equation}
d_{n}\mathbb{P}\left\{\left\vert \sum_{i=1}^{n}\left( \xi
_{i,r,1}-\mathbb{E}\xi
_{i,r,1}\right) \right\vert \geq b_{n}\right\} \leq d_{n}\left( n^{-CB}+%
\frac{n}{q}\beta \left( q\right) \right) \leq \frac{C}{n\left( \log
n\right)^{2}},  \label{b11}
\end{equation}%
where $q$ is equal to $\frac{b_{n}}{\log n}$. In view of condition (C4), it
follows that
\begin{equation*}
n\mathbb{E}\xi_{i,r,1}=nh_{G}\big( \eta _{r}-G(v)\big )
f_{q_{0}}(v) \big ( 1+O(h_{G}) \big )
\end{equation*}%
and
\begin{equation*}
nh_{G}\big( F(\eta_{r}|v)-\left( 1-\alpha \right) \big)
=nh_{G}\big (\eta_{r}-G(v)\big ) f_{q_{0}}(v) \big (
1+O(h_{G}) \big ) .
\end{equation*}%
We then have%
\begin{equation}
n\mathbb{E}\xi_{i,r,1}-nh_{G}\big( F(\eta_{r}|v)-\left( 1-\alpha
\right) \big ) =O(nh_{G}^{2}\tau_{n}).  \label{b12}
\end{equation}%
Also, it can be shown that $\sum_{i=1}^{n}K_{G}\left( \frac{v-q_{0}(X_{i})}{%
h_{G}}\right) =nh_{G}f_{q_{0}}\left( v\right) \big( 1+O(h_{G}^{2}) \big ) $. Hence, from this, \eqref{b11},
\eqref{b12} and the fact $b_{n}\geq $ $nh_{G}^{2}\tau_{n}$, it can
be inferred that
\begin{equation}
d_{n}\mathbb{P}\left\{ \left\vert F_{n}(\eta _{r}|v)-F_{n}(G(
v)
|v)-F(\eta _{r}|v)+\left( 1-\alpha \right) \right\vert \geq B\tau _{n}^{%
\frac{3}{2}-\frac{1+\varepsilon }{2r}}\right\} \leq \frac{C}{n\left(
\log n\right)^{2}}.  \label{b15}
\end{equation}%
Then \eqref{FRate1} is implied by \eqref{b13}, \eqref{b14} and
\eqref{b15}. From \eqref{rate1} and \eqref{FRate1}, we have
\begin{equation}
\sup_{\left\vert y-G\left( v\right) \right\vert \leq 2c\tau
_{n}}\left\vert \widehat{F}_{n}(y|v)-F_{n}(G(v)
|v)-F(y|v)+\left( 1-\alpha \right)
\right\vert =O\left( \tau _{n}^{\frac{3}{2}-\frac{1+\varepsilon }{2r}%
}+\delta \right) .  \label{FRate2}
\end{equation}%
Next, we prove (5.2). In order to do so, we first show that,
with probability one,
\begin{equation}
\left\vert G_{n}(v) -G(v) \right\vert
=O(\tau_{n}),  \label{rate4}
\end{equation}%
which is similar to Lemma 2.1 of \citet{mehra92}. For
any constant $M>0$, let $u_{n}^{+}=G(v)+M\tau _{n}$ and
$u_{n}^{-}=G(v)-M\tau _{n}$ and then consider the following
probability
\begin{equation*}
\mathbb{P}\left\{ \left\vert G_{n}(v)-G(v)\right\vert \geq M\tau
_{n}\right\} =\mathbb{P}\left\{ F_{n}\left( \left. u_{n}^{-}\right\vert
v\right) \leq 1-\alpha \right\} +\mathbb{P}\left\{ F_{n}\left( \left.
u_{n}^{+}\right\vert v\right) \geq 1-\alpha \right\} .
\end{equation*}%
Let $\zeta_{i}=K\left( \frac{v-q_{0}(X_{i})}{h_{G}}\right)
\mathbb{I}(Y_{i}\leq u_{n}^{+})$. From Lemma \ref{lemmaA1}, we have
\begin{equation*}
\mathbb{P}\left\{ \left\vert \sum_{i=1}^{n}\left( \zeta_{i}-\mathbb{E}\zeta
_{i}\right) \right\vert \geq M\tau _{n}\right\} \leq n^{-CM}+\frac{n}{q_{1}}%
\beta \left( q_{1}\right) \leq \frac{C}{n\left( \log n\right) ^{2}},
\end{equation*}%
where $q_{1}$ is taken as $\frac{1}{\tau _{n}h^{\frac{1+\varepsilon
}{r}}}$. From this and noting that $n\left(\mathbb{E}\zeta _{i}-\left(
1-\alpha \right) \mathbb{E}K\left( \frac{v-q_{0}(X_{i})}{h_{G}}\right)
\right) =O\left( nh_{G}\tau_{n}\right) $, it can be inferred that
$\mathbb{P}\left\{ F_{n}\left(
\left. u_{n}^{+}\right\vert v\right) \geq 1-\alpha \right\} \leq \frac{1}{%
n\left( \log n\right) ^{2}}$. Analogously, it can be shown that $%
\mathbb{P}\left\{ F_{n}\left( \left. u_{n}^{-}\right\vert v\right)\leq
1-\alpha
\right\} \leq \frac{1}{n\left( \log n\right)^{2}}$. And thus, we know that \eqref{rate4} holds. Then there exists some constant $C>0$ such that
\begin{equation}
1-\alpha \in \left[ F_{n}\big (G\left( v\right) -C\tau
_{n}|v\big ) ,F_{n}\big( G\left( v\right) +C\tau_{n}|v\big )
\right]  \label{interval}
\end{equation}%
holds with probability one. Let
\begin{equation*}
\xi_{i,r,2}=\xi _{i,r,2}\left( v\right) =K_{G}\left( \frac{v-q_{0}(X_{i})}{%
h_{G}}\right) \mathbb{I}\big( G(v)+C\tau _{n}<Y_{i}\leq G(v)+2C\tau
_{n}\big) .
\end{equation*}%
Then, from Lemma \ref{lemmaA1}, it follows that
\begin{equation}
\mathbb{P}\left\{ \left\vert \sum_{i=1}^{n}\left( \xi _{i,r,2}-\mathbb{E}\xi
_{i,r,2}\right) \right\vert \geq Bnh_{G}\tau _{n}^{\frac{3}{2}-\frac{%
1+\varepsilon }{2r}}\right\} \leq n^{-CB}+\frac{n}{q_{2}}\beta
(q_{2}) \leq \frac{C}{n\left( \log n\right) ^{2}},
\label{b25}
\end{equation}%
where $q_{2}=\frac{nh_{G}\tau_{n}^{\frac{3}{2}-\frac{1+\varepsilon }{2r}}}{%
\log n}$. From this and $\sum_{i=1}^{n}K_{G}\left( \frac{v-q_{0}(X_{i})}{%
h_{G}}\right) =nh_{G}f_{q_{0}}(v) +O(h_{G}^{2}) $, it can be inferred that
\begin{equation}
F_{n}\big(G(v)+2\tau_{n}|v\big)-F_{n}\big(G(v)+\tau _{n}|v\big)=c_{n1}\left(
v\right) \tau_{n}+O\left( \tau
_{n}^{\frac{3}{2}-\frac{1+\varepsilon }{2r}}\right) , \label{FRate3}
\end{equation}%
where $\liminf c_{n1}\left( v\right) >0$ as $n\rightarrow \infty $.
Similarly, it follows that
\begin{equation}
F_{n}\big(G(v)-2\tau_{n}|v\big)-F_{n}\big(G(v)-\tau_{n}|v\big)=-c_{n2}\left(
v\right) \tau_{n}+O\left( \tau
_{n}^{\frac{3}{2}-\frac{1+\varepsilon }{2r}}\right) , \label{FRate5}
\end{equation}%
where $\liminf c_{n2}\left( v\right) >0$. Therefore, from $\tau_{n}>\delta $
and monotonicity on the empirical conditional function, it can be
inferred subsequently that
\begin{eqnarray*}
&&\widehat{F}_{n}\big (G( v)-2C\tau_{n}|v\big ) \leq
F_{n}\big (G(v)-2C\tau_{n}|v\big ) +\delta \leq
F_{n}\big(G(v)
-C\tau_{n}|v\big) \\
&\leq &1-\alpha \leq F_{n}\big (G(v) +C\tau_{n}|v\big ) \leq
F_{n}\big (G(v)+2C\tau_{n}|v\big ) -\delta \leq \widehat{F}%
_{n}\big(G(v)+2C\tau_{n}|v\big) .
\end{eqnarray*}%
And this implies (5.2). In view of (5.2) and
\eqref{FRate2}, it can be shown that (5.3) holds.
\qed
\vspace{0.3cm}

\textit{ii)} To obtain the uniform convergence of (5.3), according
to the
previous proof, it suffices to show that \eqref{rate1}, \eqref{FRate1}, \eqref{rate4}, \eqref{FRate3} and \eqref{FRate5} hold uniformly for $v\in
V$. We deal with each relationship separately. To prove the results
on the uniform convergence for $v\in \mathcal{V}$, it is necessary
to divide $\mathcal{V}$ into smaller intervals with equal length
$d_{n}>0$. Assume that the corresponding grid
points are $v_{l}$, $l=1,2,\ldots ,O(d_{n}^{-1})$, assuming
$l$ is an integer. Denote the $%
l$th smaller interval $[v_{l},v_{l+1}]$ by $J_{l}$. Note that
$d_{n}$ may take different values at different s, for simplicity
of notation.

As for uniform convergence of \eqref{rate1}, it can be inferred from
the uniform convergence of \eqref{KRate2} and \eqref{KRate1}. Here,
since the
proofs of the two relationships are the same, it suffices to prove that 
\eqref{KRate1} holds uniformly for $v\in \mathcal{V}$. Then, going along
the same lines as the proof of \eqref{KRate1}, we need to prove that
$\sup_{v\in \mathcal{V}}\left\vert I_{1}\right\vert =O\left(
nh_{G}\delta _{n}\right) $, $\sup_{v\in\mathcal{V}}\left\vert
I_{21}\right\vert =o\left( nh_{G}\delta \right) $ and $\sup_{v\in
\mathcal{V}}\left\vert I_{22}\right\vert =o\left( nh_{G}\delta
\right) $. Since the proofs of these three relationships are
analogous, we here only prove the first one. Divide $\mathcal{V}$ as
being stated, and let $d_{n}=h_{G}^{2}$. Then, for $v\in J_{l}$, it
follows that
\begin{eqnarray*}
&&\left\vert K_{G}^{\prime }\left( \frac{v-q_{0}\left( X_{j}\right) }{h_{G}}%
\right) -K_{G}^{\prime }\left( \frac{v_{l}-q_{0}\left( X_{j}\right) }{h_{G}}%
\right) \right\vert\leq \frac{Cd_{n}}{h_{G}}\mathbb{I}\big (\left\vert
v_{l}-q_{0}\left(
X_{j}\right) \right\vert \leq h_{G}\big )  \\
&&\qquad\qquad\qquad+\, C\mathbb{I}\left( h_{G}-d_{n}\leq \left\vert
v_{l}-q_{0}\left( X_{j}\right) \right\vert \leq h_{G}+d_{n}\right)
\equiv\frac{Cd_{n}}{h_{G}}\xi_{i1}+C\xi_{i2}.
\end{eqnarray*}%
Thus, our objective is to show that%
\begin{equation}
\mathbb{P}\left( \cup _{l}\left( \frac{d_{n}}{h_{G}}\sum_{j=1}^{n}\xi
_{i1}\geq Cnh_{G}^{2}\right) \!\right) \leq \frac{1}{n\left( \log
n\right) ^{2}}\,\, \mbox{and}\,\,\mathbb{P}\left( \cup_{l}\left(
\sum_{j=1}^{n}\xi _{i2}\geq Cnh_{G}^{2}\right) \!\right) \leq
\frac{1}{n\left( \log n\right)^{2}}. \label{b18}
\end{equation}%
We  only prove the first one, because the proof of the other
one is
similar. Noting that $f_{q^{0}}\left( \cdot \right) $ is bounded on $%
\mathcal{V}$, we know that $n\mathbb{E}\xi _{i1}=O(nh_{G})
$. Then,
from Lemma A1, it can be derived that%
\begin{eqnarray*}
&&\mathbb{P}\left( \frac{d_{n}}{h_{G}}\sum_{j=1}^{n}\xi_{i1}\geq
CBnh_{G}^{2}\right) \leq \mathbb{P}\left( \sum_{j=1}^{n}\left( \xi
_{i1}-\mathbb{E}\xi_{i1}\right) \geq CBnh_{G}\right) \\
&\leq &2\exp \left\{ \frac{-\left( \frac{Cnh_{G}}{4}\right)
^{2}}{cn\left(
\mathbb{E}\xi _{i1}^{r_{2}}\right) ^{\frac{2}{r_{2}}}+\frac{2}{3}qM\frac{Cnh_{G}%
}{4}}\right\} +\!\frac{n}{q}\beta (q) \leq n^{-CB}+\!\frac{n}{q}%
\beta (q) \leq \frac{d_{n}}{n\left( \log n\right) ^{2}},
\end{eqnarray*}%
where $1-\frac{2}{r_{2}}=\frac{1+\varepsilon_{0}}{r}$ and $q=\frac{%
nh_{G}^{1+\frac{1+\varepsilon_{0}}{r}}}{\log n}$. This implies that
the first limit of \eqref{b18} holds. In view of \eqref{b8} and
Lemma \ref{lemmaA1},
\begin{eqnarray*}
&&\mathbb{P}\left( \cup_{l}\left( \left\vert
\sum_{j=1}^{n}K_{G}^{\prime }\left( \frac{v-q_{0}\left( X_{j}\right)
}{h_{G}}\right)
-\mathbb{E}K_{G}^{\prime }\left( \frac{v-q_{0}\left( X_{j}\right) }{h_{G}}%
\right) \right\vert \geq CBnh_{G}^{2}\right) \right) \\
&\leq &2\sum_{l}\left( \exp \left\{ \frac{-\left( \frac{Cnh_{G}^{2}}{4}%
\right) ^{2}}{cn\left(\mathbb{E}\xi _{i1}^{r_{2}}\right) ^{\frac{2}{r_{2}}}+%
\frac{2}{3}q_{1}M\frac{Cnh_{G}^{2}}{4}}\right\}
+\!\frac{n}{q_{1}}\beta
\left( q_{1}\right) \right) \\
&\leq & d_{n}^{-1}\left( n^{-CB}+\!\frac{n}{q_{1}}%
\beta \left( q_{1}\right) \right) \leq \frac{C}{n\left( \log
n\right) ^{2}},
\end{eqnarray*}%
where $q_{1}=\frac{nh_{G}^{2}}{\log n}$. Therefore, from this,
\eqref{b18}, Theorem 4.1 and the Borel-Cantelli lemma, we know
that $\sup_{v\in \mathcal{V}}\left\vert I_{1}\right\vert =O\left(
nh_{G}\delta \right)$.

We now begin to prove the uniform convergence of \eqref{FRate1}. For
the division of this, we take $d_{n}=\delta h_{G}$. First, note that
\begin{equation*}
I_{3}=\sup_{v\in J_{l}}\sup_{\left\vert y-G\left( v\right)
\right\vert \leq C\tau _{n}}\left\vert F_{n}(y|v)-F_{n}(G\left(
v\right) |v)-F(y|v)+\left( 1-\alpha \right) \right\vert \leq
I_{31}+I_{32},
\end{equation*}%
where
\begin{equation*}
I_{31}=\sup_{v\in J_{l}}\sup_{\left\vert y-G\left( v\right)
\right\vert \leq C\tau _{n}}\left\vert F_{n}(y|v)-F_{n}(G\left(
v\right) |v)-F(y|v)-\left( F_{n}(y|v_{l})-F_{n}(G\left( v_{l}\right)
|v_{l})-F(y|v_{l})\right) \right\vert
\end{equation*}%
and
\begin{equation*}
I_{32}=\sup_{v\in J_{l}}\sup_{\left\vert y-G\left( v\right)
\right\vert \leq C\tau _{n}}\left\vert F_{n}(y|v_{l})-F_{n}(G\left(
v_{l}\right) |v_{l})-F(y|v_{l})+\left( 1-\alpha \right) \right\vert
.
\end{equation*}%
Since $\left\vert G\left( v\right) -G\left( v_{l}\right) \right\vert
\leq
C\delta $ for any $v\in J_{l}$, it then can be seen that%
\begin{equation}
I_{32}\leq \sup_{\left\vert y-G\left( v_{l}\right) \right\vert \leq
2C\tau _{n}}\left\vert F_{n}(y|v_{l})-F_{n}(G\left( v_{l}\right)
|v_{l})-F(y|v_{l})+\left( 1-\alpha \right) \right\vert . \label{b24}
\end{equation}%
Let $I_{311}=\sup_{v\in J_{l}}\left\vert F_{n}(G\left( v\right)
|v)-F_{n}(G\left( v_{l}\right) |v_{l})\right\vert $,
$I_{312}=\sup_{v\in
J_{l}}\sup_{y}\left\vert F_{n}(y|v)-F_{n}(y|v_{l})\right\vert $ and $%
I_{313}=\sup_{v\in J_{l}}\sup_{y}\left\vert
F(y|v)-F(y|v_{l})\right\vert $. Then, $I_{31}\leq
I_{311}+I_{312}+I_{313}$. Set
\begin{equation*}
I_{3111}=\sup_{v\in J_{l}}\left\vert F_{n}(G\left( v_{l}\pm C\delta
\right) |v)-F_{n}(G\left( v_{l}\right) |v)\right\vert
\end{equation*}%
and $I_{3112}=\sup_{v\in J_{l}}\left\vert F_{n}(G\left( v_{l}\right)
|v)-F_{n}(G\left( v_{l}\right) |v_{l})\right\vert $. Then, we know that $%
I_{311}\leq I_{3111}+I_{3112}$. Note that\vspace{-0.2cm}
\begin{eqnarray*}
I_{3111} &\leq &2I_{312}+\left\vert F_{n}\big(G\left( v_{l}\pm C\delta
\right)
|v_{l}\big)-F_{n}(G\left( v_{l}\right) |v_{l})\right\vert \\
&\leq &2I_{312}+I_{32}+\left\vert F\big(G\left( v_{l}\pm C\delta \right)
|v_{l}\big)-F(G\left( v_{l}\right) |v_{l})\right\vert .
\end{eqnarray*}%
Then, from the facts that $I_{3112}\leq I_{312}$, $\left\vert
F\big(G\left( v_{l}\pm C\delta \right) |v_{l}\big)-F\big(G\left( v_{l}\right)
|v_{l}\big)\right\vert \leq C\delta $ and $I_{313}\leq C\delta $, which
is implied by condition (C6), it
suffices to consider $I_{312}$ and $I_{32}$. To get the result $%
\mathbb{P}\left\{ \left\vert \cup _{l}\left\{ I_{312}\geq C\delta
\right\} \right\vert \right\} \leq \frac{C}{n\left( \log n\right)
^{2}}$, it suffices to prove that
\begin{equation*}
\mathbb{P}\left\{\! \cup _{l}\left\{\! \sup_{v\in
J_{l}}\sup_{y}\left\vert
\sum_{j=1}^{n}\left( K_{G}\left( \frac{v-q_{0}\left( X_{j}\right) }{h_{G}}%
\right) -K_{G}\left( \frac{v_{l}-q_{0}\left( X_{j}\right)
}{h_{G}}\right) \right) \mathbb{I}\left( Y_{i}\leq y\right) \right\vert \geq
Cnh_{G}\delta \!\right\} \!\right\}
\end{equation*}%
and
\begin{equation}
\mathbb{P}\left\{ \cup _{l}\left\{ \sup_{v\in J_{l}}\left\vert
\sum_{j=1}^{n}\left( K_{G}\left( \frac{v-q_{0}\left( X_{j}\right) }{h_{G}}%
\right) -K_{G}\left( \frac{v_{l}-q_{0}\left( X_{j}\right)
}{h_{G}}\right) \right) \right\vert \geq Cnh_{G}\delta \right\}
\right\}  \label{b20}
\end{equation}%
are less than $\frac{C}{n\left( \log n\right) ^{2}}$. In considering
that the proofs of two relationships above are the same, we only
prove the later one. Since
\begin{eqnarray*}
&&\left\vert \sum_{j=1}^{n}\left( K_{G}\left( \frac{v-q_{0}\left(
X_{j}\right) }{h_{G}}\right) -K_{G}\left( \frac{v_{l}-q_{0}\left(
X_{j}\right) }{h_{G}}\right) \right) \right\vert \\
&\leq &C\frac{d_{n}}{h_{G}}\mathbb{I}\big (\left\vert v_{l}-q_{0}\left(
X_{j}\right) \right\vert \leq h_{G}\big ) +C\mathbb{I}\big (
h_{G}-d_{n}\leq \left\vert v_{l}-q_{0}(X_{j})
\right\vert \leq h_{G}+d_{n}\big ) ,
\end{eqnarray*}%
by using Lemma \ref{lemmaA1}, we obtain%
\begin{equation*}
\mathbb{P}\left\{ \cup _{l}\left\{ \left\vert \sum_{j=1}^{n}\frac{d_{n}}{h_{G}}%
\mathbb{I}\left( \left\vert v_{l}-q_{0}\left( X_{j}\right) \right\vert
\leq
h_{G}\right) \right\vert \geq Cnh_{G}\delta \right\} \right\} \leq \frac{C}{%
n\left( \log n\right) ^{2}}
\end{equation*}%
and%
\begin{equation*}
\mathbb{P}\left\{\cup_{l}\left\{ \left\vert \sum_{j=1}^{n}\mathbb{I}\big (
h_{G}-d_{n}\leq \left\vert v_{l}-q_{0}\left( X_{j}\right)
\right\vert \leq h_{G}+d_{n}\big ) \right\vert \geq Cnh_{G}\delta
\right\} \right\} \leq \frac{C}{n\left(\log n\right) ^{2}}.
\end{equation*}%
In view of \eqref{b24}, by adopting the same proof as that of 
\eqref{FRate1}, we arrive at \[\mathbb{P}\big \{\cup_{l}\{I_{32}\geq C\delta\}
\big\} \leq \frac{C}{n\left( \log n\right) ^{2}}.\] Thus, we know that \eqref{FRate1} holds uniformly for $v\in \mathcal{V}$ with the convergence rate $\delta$.

Next, we  verify that \eqref{rate4} holds uniformly for $v\in \mathcal{V%
}$. Let $d_{n}=\tau _{n}h_{G}a_{n}$, where the positive reals $%
a_{n}\rightarrow 0$ at a rather slow rate. For $v\in J_{l}$, since $%
|G(v)-G(v_{l})|\leq C\tau _{n}$, it suffices to prove that
\begin{equation}
\sum_{l}\mathbb{P}\left\{ \sup_{v\in J_{l}}|G_{n}(v)-G_{n}(v_{l})|\geq
C\tau _{n}\right\} \leq \frac{C}{n\left( \log n\right) ^{2}}
\label{b17}
\end{equation}%
and
\begin{equation*}
\mathbb{P}\left\{ \cup _{l}(|G_{n}(v_{l})-G(v_{l})|\geq C\tau
_{n})\right\} \leq \frac{C}{n\left( \log n\right) ^{2}}.
\end{equation*}%
The second relationship above can be obtained similarly as
\eqref{rate4}. Analogous to the proof of $I_{312}$, it follows that
\begin{equation*}
\sum_{l}\mathbb{P}\left\{ \sup_{v\in J_{l}}\sup_{y}\left\vert
F_{n}(y|v)-F_{n}(y|v_{l})\right\vert \geq ca_{n}\tau _{n}\right\} \leq \frac{%
C}{n\left( \log n\right) ^{2}}.
\end{equation*}%
Then, for any event $w\in $ $\left\{ \sup_{v\in
J_{l}}\sup_{y}\left\vert F_{n}(y|v)-F_{n}(y|v_{l})\right\vert
<ca_{n}\tau _{n}\right\} $ and $v\in J_{l}$, using relationship
\eqref{FRate3},
\begin{eqnarray*}
&&F_{n}\big(G(v_{l})-2C\tau_{n}|v\big)\leq F_{n}\big(G(v_{l})-2C\tau
_{n}|v_{l}\big)+Ca_{n}\tau_{n}\leq F_{n}\big(G(v_{l})-C\tau_{n}|v_{l}\big) \\
&\leq &1-\alpha \leq F_{n}\big(G(v_{l})+C\tau_{n}|v_{l}\big)\leq
F_{n}\big(G(v_{l})+2C\tau_{n}|v_{l}\big)-Ca_{n}\tau_{n}\leq
F_{n}\big(G(v_{l})+2C\tau_{n}|v\big).
\end{eqnarray*}%
And this together with \eqref{b17} implies \eqref{rate4}.

Finally, we show that \eqref{FRate3} and \eqref{FRate5} hold
uniformly for $v\in \mathcal{V}$. \ Let $d_{n}=a_{n}\tau_{n}$.\
First, note that
\begin{eqnarray*}
&&\left\vert \sum_{i=1}^{n}\big (\xi_{i,r,2}\left( v\right) -\xi
_{i,r,2}\left( v_{l}\right) \big ) \right\vert \\
&\leq &\sum_{i=1}^{n}\left( K_{G}\left( \frac{v_{l}-q_{0}(X_{i})}{h_{G}}%
\right) +\left\vert K_{G}\left( \frac{v-q_{0}(X_{i})}{h_{G}}\right)
-K_{G}\left( \frac{v_{l}-q_{0}(X_{i})}{h_{G}}\right) \right\vert \right) \\
&&\cdot \left\vert \mathbb{I}\big (G(v)+C\tau _{n}<Y_{i}\leq
G(v)+2C\tau _{n}\big ) -\mathbb{I}\left( G(v_{l})+C\tau_{n}<Y_{i}\leq
G(v_{l})+2C\tau_{n}\right) \right\vert \\
&&+\sum_{i=1}^{n}\left\vert K_{G}\left(
\frac{v-q_{0}(X_{i})}{h_{G}}\right) -K_{G}\left(
\frac{v_{l}-q_{0}(X_{i})}{h_{G}}\right) \right\vert \mathbb{I}\big(
G(v_{l})+C\tau _{n}<Y_{i}\leq G(v_{l})+2C\tau_{n}\big) .
\end{eqnarray*}%
Also, since $\left\vert G(v)-G(v_{l})\right\vert \leq ca_{n}\tau _{n}$ for $%
v\in J_{l}$, it holds that
\begin{eqnarray*}
&&\left\vert K_{G}\left( \frac{v-q_{0}(X_{i})}{h_{G}}\right)
-K_{G}\left(
\frac{v_{l}-q_{0}(X_{i})}{h_{G}}\right) \right\vert \\
&\leq &C\tau_{n}\mathbb{I}\big (\left\vert
v_{l}-q_{0}(X_{i})\right\vert \leq h_{G}\big ) +\mathbb{I}\big (
h_{G}\left( 1-C\tau_{n}\right) <v_{l}-q_{0}(X_{i})\leq h_{G}\left(
1+C\tau_{n}\right) \big ) 
\end{eqnarray*} \vspace{-0.4cm}
and \vspace{-0.2cm}
\begin{eqnarray*}
&&\left\vert \mathbb{I}\big( G(v)+C\tau_{n}<Y_{i}\leq G(v)+2C\tau
_{n}\big) -\mathbb{I}\big(G(v_{l})+C\tau_{n}<Y_{i}\leq
G(v_{l})+2C\tau _{n}\big)
\right\vert \\
&\leq &\mathbb{I}\big(G(v_{l})+C\tau_{n}-Ch_{G}\tau_{n}<Y_{i}\leq
G(v_{l})+C\tau_{n}+Ch_{G}\tau_{n}\big ) \\
&&+\mathbb{I}\big (G(v_{l})+2C\tau_{n}-Ch_{G}\tau_{n}<Y_{i}\leq
G(v_{l})+2C\tau_{n}+Ch_{G}\tau_{n}\big ).
\end{eqnarray*}%
From the three relationships above, we see that $\left\vert
\sum_{i=1}^{n}\left( \xi_{i,r,2}\left( v\right) -\xi_{i,r,2}\left(
v_{l}\right) \right) \right\vert $ has a bound which is independent
of $v$. Let
\begin{equation*}
\zeta _{i1}=K_{G}\left( \frac{v_{l}-q_{0}(X_{i})}{h_{G}}\right)
\mathbb{I}\big (G(v_{l})+c\tau_{n}-ca_{n}\tau_{n}<Y_{i}\leq
G(v_{l})+c\tau_{n}+ca_{n}\tau_{n}\big)
\end{equation*}%
and%
\begin{equation*}
\zeta_{i2}=c\frac{a_{n}\tau_{n}}{h_{G}}\mathbb{I}\big ( \left\vert
v_{l}-q_{0}(X_{i})\right\vert \leq h_{G}\big ) \mathbb{I}\big (
G(v_{l})+c\tau _{n}<Y_{i}\leq G(v_{l})+2c\tau_{n}\big) .
\end{equation*}%
Next, we  focus on the convergence rate of this bound. Here, we
only consider the two terms
\begin{equation*}
\sum_{i=1}^{n}\zeta_{i1}\quad \mbox{ and }\quad \sum_{i=1}^{n}\zeta
_{i2}.
\end{equation*}%
The other terms can be dealt with analogously. Noting that
\begin{equation*}
n\mathbb{E}\zeta _{i1}=Cf_{q_{0}}(v_{l}) nh_{G}a_{n}\tau
_{n}\big( 1+o(1)\big ) \mbox{  and  }n\mathbb{E}\zeta
_{i1}=Cna_{n}\tau _{n}^{2}f_{q_{0}}(v_{l}) \big (
1+o(1)\big) ,
\end{equation*}%
and then by using Lemma \ref{lemmaA1}, we have
\begin{eqnarray*}
&&\mathbb{P}\left( \cup _{l}\left( \left\vert \sum_{i=1}^{n}\zeta
_{i1}\right\vert \geq cnh_{G}a_{n}\tau_{n}\right) \right) \\
&\leq &2d_{n}^{-1}\exp \left\{ \frac{-\left( \frac{Cnh_{G}a_{n}\tau _{n}}{4}%
\right) ^{2}}{Cn\left(\mathbb{E}\zeta_{i1}^{r_{2}}\right) ^{\frac{2}{r_{2}}}+%
\frac{2}{3}q\frac{Cnh_{G}a_{n}\tau_{n}}{4}}\right\} +d_{n}^{-1}\!\frac{n}{q}%
\beta(q) \leq
d_{n}^{-1}n^{-CB}+d_{n}^{-1}\!\frac{n}{q}\beta(q)
\end{eqnarray*}%
and, similarly,
\begin{equation*}
\mathbb{P}\left( \cup _{l}\left( \left\vert \sum_{i=1}^{n}\zeta
_{i2}\right\vert
\geq Cnh_{G}a_{n}\tau_{n}\right) \right) \leq d_{n}^{-1}n^{-CB}+d_{n}^{-1}\!%
\frac{n}{q}\beta(q) ,
\end{equation*}%
where $q=\frac{a_{n}nh_{G}\tau_{n}}{\log n}$ and $\frac{2}{r_{2}}=1-\frac{%
1+\varepsilon }{r}$. By a similar method as that of \eqref{b25}, it
follows that
\begin{equation*}
\mathbb{P}\left\{ \cup _{l}\left( \left\vert \sum_{i=1}^{n}\left( \xi
_{i,r,2}(v_{l})-\mathbb{E}\xi_{i,r,2}(v_{l})
\right) \right\vert \geq Bnh_{G}a_{n}\tau _{n}\right) \right\} \leq
\frac{C}{n\left( \log n\right) ^{2}}.
\end{equation*}%
Therefore, from the analysis above, we have
\begin{equation*}
\sup_{v\in \mathcal{V}}\left\vert \sum_{i=1}^{n}\big(\xi_{i,r,2}(v) -\mathbb{E}\xi _{i,r,2}(v) \big )
\right\vert =O(nh_{G}a_{n}\tau _{n}) .
\end{equation*}%
Similar as the previous proof of \eqref{b20}, it follows that
\begin{equation*}
\sup_{v\in \mathcal{V}}\left\vert \sum_{i=1}^{n}\left( K_{G}\left( \frac{%
v-q_{0}(X_{i})}{h_{G}}\right) -\mathbb{E}K_{G}\left( \frac{v-q_{0}(X_{i})}{h_{G}}%
\right) \right) \right\vert =O(nh_{G}a_{n}) .
\end{equation*}%
Then, from the two relationships above and the condition $%
\inf_{v}f_{q_{0}(X_{i})}(v) >0$, it follows that
\begin{equation*}
\sup_{v\in \mathcal{V}}\left\vert \frac{\sum_{i=1}^{n}\xi
_{i,r,2}(v) }{\sum_{i=1}^{n}K_{G}\left( \frac{v-q_{0}(X_{i})}{h_{G}}\right) }-%
\frac{\mathbb{E}\xi _{i,r,2}(v) }{\mathbb{E}K_{G}\left( \frac{%
v-q_{0}(X_{i})}{h_{G}}\right) }\right\vert =O(a_{n}\tau_{n}).
\end{equation*}%
Since $\mathbb{E}\xi _{i,r,2}(v) =cf_{q_{0}}(v)
h_{G}\tau_{n}$, we obtain
\begin{equation*}
\sup_{v\in \mathcal{V}}\left\vert F_{n}\big(G(v)
+2C\tau_{n}|v\big ) -F_{n}\big (G(v) +C\tau
_{n}|v\big ) -C\big (1+o(1) \big ) \tau_{n}\right\vert =O(a_{n}\tau_{n}) .
\end{equation*}%
Similarly, we have
\begin{equation*}
\sup_{v\in \mathcal{V}}\left\vert F_{n}\big (G(v)
-2C\tau_{n}|v\big ) -F_{n}\big (G(v) -C\tau
_{n}|v\big ) +C\big (1+o(1) \big ) \tau_{n}\right\vert =O(a_{n}\tau_{n}) .
\end{equation*}%
This completes the proof of part \textit{ii)}.
\endproof%

\renewcommand{\theequation}{D.\arabic{equation}}
\setcounter{equation}{0}
\setcounter{lemma}{0}
\setcounter{section}{1}
\setcounter {subsection}{3}
\renewcommand{\thetheorem}{D,\arabic{section}}
\renewcommand{\thetheorem}{\thesubsection\arabic{section}}

\subsection{Proofs of some Lemmas and a Corollary}\label{app:D}

\noindent \textbf{Proof of Lemma 3.1}  
\proof%
Denote  $W_{n}(t)$ by $W_{j,n}(t_{1,u})$
with $\widetilde{X}_j$ replaced by $t$. By the 
mean value of integration
and by Taylor expansion, there exist $0< \theta_1=\theta_1(X_i,t),
\theta_2=\theta_2(X_i,t)<1$ such that
\begin{align*}
\mathbb{E}\left[\left. \mathbb{I}_{\left(\varepsilon_i\leq
0\right)}-\mathbb{I}_{\left(\varepsilon_i\leq
r_{i,t}\right)}\right|X_i\right]=\int_{r_{i,t}}^{0} g(X_i,s){\rm d}s
=r_{i,t}g(X_i,\theta_2r_{i,t})=r_{i,t}g(X_i,0)+O(r_{i,t}^2)
\end{align*}
and
\[
r_{i,t}=\frac{\left(X_i-t\right)^p}{p!}q^{(p)}(t)
+\frac{\left(X_i-t\right)^{p+1}}{(p+1)!}q^{(p+1)}\big (t+\theta_1(X_{i}-t
) \big ).
\]
From the two relationships above, variable substitution and the
continuity of both the density function $p(\cdot)$ and the $(p+1)$th
derivative of $q(\cdot)$, it can be inferred that
\begin{align*}
\frac{1}{nh^{d+1}}\mathbb{E} W_{n}\left( t\right)  & h^{p-1}\int K(s)
A(s) s^{p} {\rm d}s \cdot\frac{q^{(p)}(t
)g_1\left(t\right)}{p!}+h^{p}\left[\int K(s) A(s) s^{p+1}{\rm  d}s \cdot
\frac{q^{(p+1)}(t)g_1(t)}{(p+1)!} \right.\notag\\
&\left.+ \int K(s) A(s) s^{p}{\rm d}s \cdot \frac{q^{(p)}(t
g_{1}'(t)}{p!} \right]+o(h^{p}).
\end{align*}
According to this and the facts that \eqref{QE}, 
$e_u^{\mbox{\tiny{T}}}Q^{-1}\int
A(s)K(s)s^l{\rm d}s=0$ with  $l=p$ or $p+1$, it can be inferred that
\begin{equation*}
\frac{1}{nh^{d+1}}Q_{jn,t_{1,u}}^{-1}\mathbb{E}_jW_{j,n}(t_{1,u}) =
h^pB_2(\widetilde{X}_j)  +o(h^p).
\end{equation*}
By the uniformly SLLN for kernel regression
function \citep[see, e.g.,][]{masry96}, with probability one, it
holds uniformly for $t\in \mathcal{X}$ that
\begin{eqnarray*}
\frac{1}{nh^{d+1}}\big(W_{n}(t)-\mathbb{E} W_{n}(t)
\big)=O\left(\frac{1}{h} \left(\frac{\log n}{nh^d}\right)^{1/2}\right)
.
\end{eqnarray*}
According to the SLLN, we have that
\[
 \frac{1}{n}\sum_{j=1}^n \left[\iint  B_2(\widetilde{X}_j){\rm d}t_{1}{\rm d}t_{2}
 -\mathbb{E}\iint  B_2(\widetilde{X}_j){\rm d}t_1{\rm d}t_2\right]
 =O\left(\sqrt{\frac{\log n}{n}}\right).
\]
Thus, from the three relationships above, it can be inferred that
Lemma 3.1 holds.
\endproof%

\subsection*{Proof of Lemma 3.2} 
\proof%
Let $\gamma _{n}=\frac{\sqrt{\log n}}{h^{2+\frac{1+\varepsilon
}{r}}\sqrt{n}}$ and $\ell_{n}=\varepsilon_{n}h^{2}$. We only
verify $\Delta_{1,u}$ since the other cases can be dealt with
analogously. As the usual way of getting the weak law of large
number, it follows that
\begin{equation*}
\frac{1}{n}\sum_{j=1}^{n}\left(\partial_{1}q( \widetilde{X}_{j})\right)
 \mathbb{E}\partial_{1}q(\widetilde{X}_{j}) ) =O
\left(\left(\frac{\log  n}{n}\right)^{\frac{1}{2}}\right)
\end{equation*}
holds uniformly for $(t_{1},t_{u})\in \mathcal{X}_{1,u}$. Then, 
from this, (3.1) and Theorem \ref{theoremA6}, we see that, with probability one,
\begin{equation*}
\Delta_{1,u}=\frac{1}{n^{2}h^{d+1}}\sum_{1\leq i\neq j\leq n}\psi
_{1,ij}+o(\gamma_{n})
\end{equation*}%
holds uniformly for $t_{1,u}\in \mathcal{X}_{1,u}$, where $\psi
_{1,ij}=e_{1}^{\mbox{\tiny{T}}}Q_{jn,t_{1,u}}^{-1}K_{ij,t_{1,u}}A_{ij,t_{1,u}}\big (
\alpha \mathbb{I}(\varepsilon_{i}\leq 0) \big )$. Let $\xi_{i}=
\mathbb{E}_{i}\psi_{1,ij}$. Then,
\begin{equation*}
\frac{1}{n^{2}h^{d+1}}\sum_{1\leq i\neq j\leq n}\psi_{1,ij}=\frac{1}{%
n^{2}h^{d+1}}\sum_{1\leq i<j\leq n}\left( \psi_{1,ij}+\psi_{1,ji}-\xi
_{i}-\xi _{j}\right) +\frac{\left( n-1\right) }{n^{2}h^{d+1}}%
\sum_{i=1}^{n}\xi_{i}.
\end{equation*}%
Let $U_{n}$ be the first term on the RHS of the
relationship
above. We now use Lemma \ref{lemmaA2} to get the convergence rate of $U_{n}$. And let $%
M_{p_{1}k}$ be the corresponding quantity stated in Lemma \ref{lemmaA2} with
$k<r\left(
1-\frac{1}{p_{1}}\right) -1$. In view of condition (B7), it can be inferred that $%
M_{p_{1}k}\leq Ch^{d/p_{1}}$. Applying Lemma A2, in view of condition (B6), it
follows that\ \ \
\begin{align*}
\frac{\mathbb{E}(U_{n}^{k}) }{\gamma_{n}^{k}\ell
_{n}^{2}}&\leq
\frac{Ch^{d/p_{1}}}{\left( \frac{\log n}{nh^{4+\frac{2+2\varepsilon }{r}}}%
\right) ^{\frac{k+2}{2}}\left( nh^{d+1}\right) ^{k}}   \\
&\leq \frac{C}{\left( \log n\right)
^{\frac{k+2}{2}}n^{\frac{k-2}{2}}h^{k\left( d+1\right)
-2\left( k+2\right) -\frac{d}{p_{1}}-\left( k+2\right) \frac{1+\varepsilon }{%
r}}}\leq \frac{1}{n(\log n)^{2}}
\end{align*}%
if $k=10$ and $r>11$. Then, it can be inferred that $\left\vert \xi
_{i}\right\vert \leq Ch^{d-2}\equiv M$ and $\left(\mathbb{E}\left\vert
\xi
_{i}^{r_{3}}\right\vert \right) ^{\frac{2}{r_{3}}}\leq Ch^{2(d-2)+\frac{4}{%
r_{3}}}$, where $1-\frac{2}{r_{3}}=\frac{1+\varepsilon }{r}$. By
Lemma \ref{lemmaA1} and taking $q=\frac{h^{1+\frac{1}{r}+\varepsilon
}\sqrt{n}}{\sqrt{\log n}}$, we have
\begin{eqnarray*}
\mathbb{P}\left\{ \left\vert \sum_{1\leq i\leq n}\xi_{i}\right\vert
\geq nh^{d+1}\gamma _{n}\right\}
\leq \exp \left\{\frac{-\left(
\varepsilon _{n}nh^{d+1}\right)^{2}}{\frac{n}{2q}q\left(
\mathbb{E}\left\vert \xi
_{i}\right\vert^{r_{3}}\right)^{\frac{2}{r_{3}}}+qM\gamma_{n}nh^{d+1}}%
\right\} +\frac{n}{q}\beta(q) \leq
n^{-CB}+\frac{n}{q}\beta(q) .
\end{eqnarray*}%
For any $s\in J_{l}$, from condition (B5) and
\begin{equation*}
Q_{jn,t_{1,u}}=\left. \int K(x)A(x)
A^{\mbox{\tiny{T}}}(x) g_{1}(z-
hx,0){\rm d} x\right\vert _{z=\widetilde{X}_{j}},
\end{equation*}%
we have $\left\Vert Q_{jn,s}-Q_{jn,t_{1,u}}\right\Vert \leq C\ell
_{n}$. Since $Q_{jn,t_{1,u}}\rightarrow Q\left( t_{1,u}\right) $ and
$Q(t_{1,u})$ is a positive definite matrix and is
continuous on the compact set, it can be inferred that $\left\Vert
Q_{jn,t_{1,u}}^{-1}\right\Vert $ is bounded by a constant.
Accordingly, we have
\begin{equation}
\left\Vert Q_{jn,s}^{-1}-Q_{jn,t_{1,u}}^{-1}\right\Vert \leq
\left\Vert Q_{jn,s}^{-1}\right\Vert \left\Vert
Q_{jn,s}-Q_{jn,t_{1,u}}\right\Vert \left\Vert
Q_{jn,t_{1,u}}^{-1}\right\Vert \leq C\ell _{n}.  \label{Qin}
\end{equation}%
Similar as the previous lemma, it follows that
\begin{eqnarray*}
\left\vert L_{1}(s) -L_{1}(t_{1,u})
\right\vert 
\leq \frac{C}{n^{2}h^{d+1}}\sum_{1\leq i\neq j\leq n}\left\{
\mathbb{I}\left(
\left\vert X_{i}-\widetilde{X}_{j}\right\vert \leq h\right) \frac{\ell_{n}}{h}%
+\mathbb{I}\left( h-\ell_{n}\leq X_{i}-\widetilde{X}_{j}\leq h+
\ell_{n}\right) \right\}.
\end{eqnarray*}%
Let $L_{2}$ and $L_{3}$ be the two terms on the RHS of
the
relationship above. We now consider the convergence rate of $L_{2}$. Let $%
\psi _{ij}=\mathbb{I}\big (|X_{i}-\widetilde{X}_{j}|\leq h\big ) +\mathbb{I}\big( |X_{j}-\widetilde{X}_{i}| \leq h\big )$, $\psi
_{i}=\mathbb{E}_{i}\psi_{ij}$ and $\varphi _{ij}=\psi_{ij}-\psi
_{i}-\psi_{j}+\mathbb{E}\psi_{1}$. Then,
\begin{equation*}
L_{2}=\frac{2\ell _{n}}{n^{2}h^{d+2}}\sum_{1\leq i<j\leq n}\varphi _{ij}+%
\frac{2\left( n-1\right) \ell_{n}}{n^{2}h^{d+2}}\sum_{i=1}^{n}\left( \psi
_{i}-\mathbb{E}\psi_{i}\right) -\frac{n\left( n-1\right) \ell_{n}}{%
n^{2}h^{d+2}}\mathbb{E}\psi_{1}\equiv L_{21}+L_{22}+L_{23}.
\end{equation*}%
By Lemma \ref{lemmaA2}, condition (B7), for $k<r\left( 1-\frac{1}{p_{1}}\right) -1$, we
have
\begin{equation*}
\mathbb{E}L_{21}^{k}\leq Cn^{k}\left( \frac{\ell
_{n}}{n^{2}h^{d+2}}\right)^{k}h^{d/p_{1}}.
\end{equation*}%
Thus, if $k$ is chosen to be $4$, in view of condition (B6), then
\begin{equation*}
\frac{\mathbb{E}_{21}^{k}}{\ell_{n}^{2}\gamma_{n}^{k}}\leq
\frac{C}{n(\log n)^{2}}.
\end{equation*}%
Since $\psi_{i}$ includes two terms and the proofs of these two
terms are the same, we just consider one of its terms $\xi
_{i}=\mathbb{E}_{i}\mathbb{I}\big(|X_{j}-\widetilde{X}_{i}|
\leq h\big )$. It is easy to see that $\left\vert \xi
_{i}\right\vert \leq Ch^{d-2}\mathbb{I}_{\left( t_{1,2}-h\leq X_{i,12}\leq
t_{1,2}+h\right) }$ and $\left(\mathbb{E}\left\vert \xi
_{i}\right\vert^{r_{3}}\right)^{\frac{2}{r_{3}}}\leq Ch^{2d-4}\cdot h^{%
\frac{4}{r_{3}}}=Ch^{2d-4}\cdot h^{2\left( 1-\frac{1+\varepsilon
}{r}\right) }=Ch^{2d-2\left( 1+\frac{1+\varepsilon }{r}\right) }$.
Then, from Lemma \ref{lemmaA1}, we have
\begin{eqnarray*}
\mathbb{P}\left\{ \left\vert \frac{\ell
_{n}}{nh^{d+2}}\sum_{i=1}^{n}\left( \xi_{i}-\mathbb{E}\xi
_{i}\right) \right\vert \geq \gamma_{n}\right\}
\!\leq\!
\exp \left\{ \frac{-\left( \frac{\gamma_{n}nh^{d+2}}{\ell_{n}}\right)^{2}%
}{\frac{n}{2q}q\left(\mathbb{E}\left\vert \xi_{i}\right\vert
^{r_{3}}\right)
^{\frac{2}{r_{3}}}+qM\frac{\gamma_{n}nh^{d+2}}{\ell_{n}}}\right\} +\frac{n%
}{q}\beta(q) \leq n^{-CB}+\frac{n}{q}\beta(q)
\end{eqnarray*}%
with $1-\frac{2}{r_{3}}=\frac{1+\varepsilon }{r}$ and $q=\frac{nh^{2}}{\log n%
}$. Then, $\frac{1}{\ell _{n}^{2}}\frac{n}{q}\beta(q)
\rightarrow 0.$ Since $\mathbb{E}\psi_{i}=Ch^{d}$, $|L_{23}|\leq \frac{1}{2}%
\gamma _{n}$. Next, we will consider the convergence rate of
$L_{3}$. Let
\begin{equation*}
\psi_{ij}=\mathbb{I}(h-\ell _{n}\leq X_{i}-\widetilde{X}_{j}\leq
h+\ell _{n}) +\mathbb{I}(h-\ell _{n}\leq
X_{j}-\widetilde{X}_{i}\leq h+\ell _{n}) ,
\end{equation*}%
$\psi_{i}=\mathbb{E}\psi_{ij}$ and $\varphi_{ij}=\psi _{ij}-\psi
_{i}-\psi _{j}+\mathbb{E}\psi_{1}$. Then,
\begin{equation*}
L_{3}=\frac{2}{n^{2}h^{d+1}}\sum_{1\leq i<j\leq n}\varphi _{ij}+\frac{%
2\left( n-1\right) }{n^{2}h^{d+1}}\sum_{i=1}^{n}\left( \psi_{i}-\mathbb{E}\psi_{i}\right) -\frac{n\left( n-1\right) }{n^{2}h^{d+1}}\mathbb{E}\psi_{1}%
\,\widehat{=}\, L_{31}+L_{32}+L_{33}.
\end{equation*}%
Again by exploiting Lemma \ref{lemmaA2} and condition (B6), if $k=8$ and $r\geq
5.4$, we have the following inequality
\begin{equation*}
\frac{1}{\ell _{n}^{2}\gamma_{n}^{k}}\mathbb{E}L_{31}^{k}\leq
\frac{1}{\ell
_{n}^{2}\gamma_{n}^{k}}\cdot\frac{C}{n^{k}h^{dk+2k}}\left( h^{d}\frac{\ell_{n}}{%
h}\right) ^{1-\frac{k+1}{r}-\varepsilon_{1}}\leq \frac{C}{n(\log
n)^{2}}.
\end{equation*}%
Similar to the case of $L_{22}$, we only consider one of the terms in $%
L_{32} $, i.e.,
\begin{equation*}
\mathbb{E}_{i}\mathbb{I}\big(h-\ell_{n}\leq |X_{i}-\widetilde{X}%
_{j}| \leq h+\ell_{n}\big) .
\end{equation*}%
After some calculation, it is asymptotically equivalent to
\begin{equation*}
h^{d-2}f\left( \bar{x}_{1,u}\right) \left( C_{1}\mathbb{I}\big (h-\ell
_{n}\leq \left\vert x_{1,u}-t_{1,u}\right\vert \leq h+\ell
_{n}\big ) +C_{2}\mathbb{I}\big (|x_{1,u}-t_{1,u}|
\leq h-\ell_{n}\big) \frac{\ell_{n}}{h}\right) .
\end{equation*}%
Since the proofs of the two terms above are the same, we only
consider the
first one, say $\xi_{i}$. Thus, $|\xi _{i}|\leq Ch^{d-2}$ and $\left(
\mathbb{E}\left\vert \xi _{i}\right\vert^{r_{3}}\right)
^{\frac{2}{r_{3}}}\leq
\left( h^{d-2}\right) ^{2}\left( h^{2}\frac{\ell_{n}}{h}\right) ^{1-\frac{%
1+\varepsilon }{r}}$ with $\frac{2}{r_{3}}=1-\frac{1+\varepsilon }{r}$. $%
\frac{1}{h^{d+1}}\mathbb{E}\xi_{i}\leq \gamma_{n}$. By using Lemma \ref{lemmaA1}, we have
\begin{equation*}
\mathbb{P}\left\{ \left\vert \frac{1}{nh^{d+1}}\sum_{i=1}^{n}\left( \xi _{i}-
\mathbb{E}\xi _{i}\right) \right\vert \geq \gamma _{n}\right\} \leq \exp
\left\{
\frac{-\left( \gamma _{n}nh^{d+1}\right) ^{2}}{\frac{n}{2q}q\left(\mathbb{E} \left\vert \psi _{i}\right\vert ^{r_{3}}\right)
^{\frac{2}{r_{3}}}+qM\gamma _{n}nh^{d+1}}\right\} +\frac{n}{q}\beta
(q) ,
\end{equation*}%
where $q=\frac{\gamma _{n}nh^{3}}{\log n}$. \ Since $\varepsilon $
is an arbitrary small constant, $\varepsilon _{n}$ can be replaced
by $\left( nh^{4+\frac{1+\varepsilon }{r}}\right) ^{-\frac{1}{2}}$.
This completes the proof of the lemma.
\endproof

\subsection*{Proof of Lemma 5.1} 
\begin{proof}%
In view of the definitions of $\widehat{q}_0(X_m)$ and $q_0(X_m)$, it
suffices to prove that, for $1\leq u\leq d$,
\[
\frac{1}{nh_G^2}\sum_{m=1}^{n}K_{G,m}^{'}\big(\widehat{q}_u(X_{m,u})
-q_u(X_{m,u})\big)=O\left(\frac{n^{-\varepsilon_0}}{\sqrt{nh_G}}\right).
\]
We here only give the proof for the case $2\leq u\leq d$ while
the proof of the  case $u=1$ is similar. Noting from the proof
of Theorem 3.1, we know that
$\widehat{q}_u(x_u)-q_u(x_u)=I_{1}-I_{2}-I_{3}$. At a random point $X_{m,u}$,
we  write $\widehat{q}_u(X_{m,u})-q_u(X_{m,u})=I_{1,m}-I_{2,m}-I_{3,m}$,
in which $I_{i,m}$ is equal to $I_{m}$ with $x_u$ replaced by
$X_{m,u}$ for $i=1,2,3$. Denote by
$J_i=\frac{1}{nh_G^2}\sum_{m=1}^{n}K_{G,m}^{'} I_{i,m}$, $i=1,2,3$.
Thus, it suffices to prove that $
J_i=O\left(\frac{n^{-\varepsilon_0}}{\sqrt{nh_G}}\right)$ for each $
i=1,2,3.$ Denote by
\[
l(X_{m,u},\widetilde{X}_j)= K_{G,m}'\int_{x_{u,0}}^{X_{m,u}}\int
\frac{w_1(t_1)\left(\partial_u q
(\widetilde{X}_j)\mathbb{I}(X_{j,\bar{u}}\in
\mathcal{X}_{j,\bar{u}})-D_u(t_1,t_u)\right)}{D_{1,u}(t_1,t_u)}
{\rm d}t_1{\rm d}t_u,
\]
 $J_{11}=\frac{1}{n^2h_G^2}\sum_{i=1}^n \sum_{j=1,j\neq m}^{n}l(X_{m,u},
 \widetilde{X}_j)$ and
\begin{align*}
J_{12}=\frac{1}{n^2h_G^2}\sum_{m=1}^{n} \sum_{j=1,j\neq
m}^{n}K_{G,m}^{'} \int_{x_{u,0}}^{X_{m,u}}\! \int\! \frac{w_1(t_1)
}{D_{1,u}(t_1,t_u)}\left(\partial_u \widehat{q}_u(\widetilde{X}_j)-
\partial_u
q_u(\widetilde{X}_j)\right)\mathbb{I}(X_{j,\bar{u}}\in
\mathcal{X}_{j,\bar{u}}){\rm d} t_1{\rm d} t_u.
\end{align*}
Then, $J_{1}=J_{11}+J_{12}$. Rewrite
\[
J_{11}=\frac{1}{n^2h_G^2}\sum_{m=1}^{n}  \sum_{j=1,j\neq
m}^{n}[l(X_{m,u},\widetilde{X}_j)-\mathbb{E}_{j}
l(X_{m,u},\widetilde{X}_j)]+\frac{n-1}{n^2h_G^2} \sum_{j=1}^{n}\mathbb{E}_{j}
l(X_{m,u},\widetilde{X}_{j})=J_{13}+J_{14}.
\]
\[
\left(\frac{\sqrt{nh_G}}{n^{-\varepsilon_0}}\right)^{k}
\mathbb{E}\left|J_{13}\right|^k\leq Cn^k
\left(\frac{\sqrt{nh_G}}{n^{2-\varepsilon_{0}}h_G^2}\right)^kM_{sk}^k=
 C
\left(n^{1+3\varepsilon_0}h_G^3\right)^{-\frac{k}{2}}h_{G}^{\frac{1}{s}}\leq
\frac{1}{n(\log n)^2}
\]
\[
\frac{\sqrt{nh_G}}{n^{-\varepsilon_{0}}}J_{14}=\frac{\sqrt{nh_G}}{n^{-\varepsilon_0}}\cdot\frac{n-1}{n^2h_G^2}
\cdot\frac{Ch_G^{2}\sqrt{n}}{\log n}\rightarrow 0.
\]
We now consider $J_{12}$. Noting from the proof of Theorem \ref{theoremA6}, to
deal with $J_{12}$, it suffices to prove the following four
expressions
\begin{align*}
\frac{1}{n^3h^{d+1} h_G^2}\sum_{m=1}^{n} \sum_{j=1,j\neq m}^{n}
\sum_{i=1,i\neq
        j}^{n} K_{G,m}^{'} \int _{x_{u,0}}^{X_{m,u}} \int \frac{w_1(t_1)
        \mathbb{I}(X_{j,\bar{u}}\in \mathcal{X}_{j,\bar{u}})
}{D_{1,u}(t_1,t_u)}Q_{ij,t_{1,u}}^{-1}K_{ij,t_{1,u}} \\
\cdot\, A_{ij,t_{1,u}}(\alpha-\mathbb{I}(\varepsilon_i\leq -r_{ij})){\rm d} t_{1}{\rm d} t_{u}~~~~~~~~~~~~~
\end{align*}
\begin{align*}
\frac{1}{n^{3}h^{d+1} h_G^2}\sum_{m=1}^{n} \sum_{j=1,j\neq m}^{n}
\sum_{i=1,i\neq
        j}^{n}  K_{G,m}^{'} \int _{x_{u,0}}^{X_{m,u}} \int \frac{w_1(t_1)
        \mathbb{I}(X_{j,\bar{u}}\in \mathcal{X}_{j,\bar{u}})
}{D_{1,u}(t_1,t_u)}\Delta_{ij}(t_{1,u},\beta){\rm d} t_1{\rm d} t_u,
\end{align*}
$\frac{1}{n^2h  h_G^2}n\cdot nh_G\cdot(\delta^{2}+\delta h^p)$ and
$\frac{1}{n^2h  h_G^2}n\cdot nh_G\cdot\frac{C}{nh^{d}}$ are all of the
order $\frac{n^{-\varepsilon_0}}{\sqrt{nh_G}}$, and moreover, the
second expression is of order uniform for $|\beta-\beta_{\tau}|\leq
\frac{B}{\sqrt{nh^{(1+\frac{1}{r})d+\varepsilon}}}$. In more 
details, the first three expressions come from the RHS
of \eqref{a4} and the last two terms on the LHS of
\eqref{a4}, and the last expression comes from the last term of
\eqref{a5}. For convenience, denote the first two terms by $J_{13}$
and $J_{14}$, respectively. As for $J_{13}$, by using the SLLN three
times, it can be obtained that
\begin{align*}
J_{13}&=\frac{n^3}{n^3h^{d+1} h_G^2}\mathbb{E}\left\{ K_{G,m}^{'}
\int_{x_{u,0}}^{X_{m,u}} 
\int \mathbb{E}\left[\frac{w_1(t_1)
\mathbb{I}(X_{j,\bar{u}}\in \mathcal{X}_{j,\bar{u}})
}{D_{1,u}(t_1,t_u)}\right.\right.\\
&\quad\left.\left. \mathbb{E}_{j} \left(Q_{ij,t_{1,u}}^{-1}K_{ij,t_{1,u}}
A_{ij,t_{1,u}}\big(\alpha-\mathbb{I}(\varepsilon_i\leq
-r_{ij})\big)\right)\right]{\rm d} t_1{\rm d} t_u\right\}\big(1+o(1)\big)=O(h^{p-1}),
\end{align*}
where the last step follows from the property of the conditional
expectation.

\begin{equation*}
\zeta_{m,i,j}=  K_{G,m}^{'} \int _{x_{u,0}}^{X_{m,u}}
\int \frac{w_1(t_1)
}{D_{1,u}(t_1,t_u)}Q_{jn,t_{1,u}}^{-1}\left(\Delta_{ij}(t_{1,u},\beta)-
\mathbb{E}\Delta_{ij}(t_{1,u},\beta)
\right){\rm d} t_1{\rm d} t_u
\end{equation*}
\begin{equation*}
\frac{1}{n^3h^{d+1} h_G^2}\sum_{m=1}^{n} \sum_{j=1,j\neq m}^{n}
\sum_{i=1,i\neq
        j}^{n} \zeta_{m,i,j}=\frac{1}{n^3h^{d+1} h_G^2}\sum_{j=1}^{n}
\sum_{i=1,i\neq
        j}^{n} \sum_{m=1,m\neq j}^{n} \zeta_{m,i,j}
\end{equation*}

\begin{equation*}
\eta_{i,j}(\beta)=\mathbb{E}_{i,j}\left[K_{G,m}^{'}\int
_{x_{u,0}}^{X_{m,u}} \int \frac{w_1(t_1)
}{D_{1,u}(t_1,t_u)}Q_{jn,t_{1,u}}^{-1}\left(\Delta_{ij}(t_{1,u},\beta)-
\mathbb{E}\Delta_{ij}(t_{1,u},\beta)
\right){\rm d} t_1{\rm d}t_u\right]
\end{equation*}
\begin{equation*}
\frac{1}{n^3h^{d+1} h_G^2}\sum_{j=1}^{n}
\sum_{i=1,i\neq
        j}^{n} \sum_{m=1,m\neq j}^{n} (\zeta_{m,i,j}-\eta_{i,j}(\beta))
\end{equation*}

\begin{equation*}
\xi_{m,i}=\zeta_{m,i,j}-\eta_{i,j}(\beta)
\end{equation*}

\begin{equation*}
\frac{(\sqrt{nh_G})^{k}\mathbb{E}|\xi_{m,i}|^k}{(n^2h^{d+1}h_G^2)^k} \leq
\frac{C(\sqrt{nh_G})^{k} h_G h^{d+p}}{(n^2h^{d+1}h_G^2)^k}\leq \frac{C
}{n(\log n)^2}.
\end{equation*}
If $k=6,$
\begin{equation*}
\frac{Cn h_G h^{d+p}}{(\sqrt{nh_G}n h^{d+1}h_G )^k}=\frac{Cn h_G
        h^{d+p}}{(\sqrt{nh_G}n h^{d+1}h_G )^6}\leq \frac{C }{(\log n)^2}.
\end{equation*}
This completes the proof.
\end{proof}

\subsection*{Proof of Lemma A5}
\begin{proof}%
Let $\delta_{n}$ denote the RHS of \eqref{Bahadur4}. In
order
to prove uniformity, we adopt a similar division on the domain $\mathcal{X}%
_{1,u}$ and the sphere $|\beta -\beta_{j,t_{1,u}}|
 =\delta $ as that of Lemma \ref{lemmaA5}. And the same notations are used here. Let $%
 \ell_{n}=d_{1}=\delta h$ and $q=\delta_{n}n^{-\frac{\varepsilon
 }{2}}$.
 For any real $r_{2}>0$, on the event $\{|X_{i}-\widetilde{X}%
 _{j}| \leq h\}$, it can be inferred that
  \begin{eqnarray*}
  \mathbb{E}_{j}\left[ \left. \left\vert \mathbb{I}\left( \varepsilon
  _{i}\leq -r_{ij,t_{1,u}}-P_{ij,t_{1,u}}\right) -\mathbb{I}\left(
  \varepsilon _{i}\leq -r_{ij,t_{1,u}}\right) \right\vert
  ^{r_{2}}\right\vert X_{i}\right] 
  =\left\vert
 \int_{-r_{ij,t_{1,u}}}^{-r_{ij,t_{1,u}}-P_{ij,t_{1,u}}}g(
  X_{i},u) {\rm d} u\right\vert \leq C\delta .
  \end{eqnarray*}%
  Furthermore, we have $\mathbb{E}_{j}\left\vert \Delta _{ij}\left(
  t_{1,u},\beta \right) \right\vert ^{r_{2}}\leq Ch^{d}\delta $. Then,
  as previously proved, in view of the inequality (\ref{Bernstein3}),
  we know that
  \begin{equation*}
  \frac{n}{q}\cdot q\left( h^{d}\delta \right)
  ^{\frac{2}{r_{2}}}<Cq\delta_{n},
  \end{equation*}%
  where $1-\frac{2}{r_{2}}=\frac{1+\varepsilon }{r}$ for some
  sufficiently small constant $\varepsilon >0$. Then, from Lemma \ref{lemmaA1},
  it follows that
  \begin{equation}
   \mathbb{P}\left( \left\vert \sum_{i=1}^{n}\big (\Delta_{ij}(t_{1,u},\beta) -\mathbb{E}_{j}\Delta_{ij}(t_{1,u},\beta )
  \big) \right\vert >\frac{\delta _{n}}{2}\right) \leq n^{-CB}+\frac{n}{q}%
  \beta \left( q\right) +C\left( \frac{q}{\delta_{n}}\right)
  ^{r_{1}}. \label{in2}
   \end{equation}%
   In addition, note that%
   \begin{eqnarray}
   &&\left\vert \beta^{\mbox{\tiny{T}}}A_{ij,t_{1,u}}-\gamma
  ^{\mbox{\tiny{T}}}A_{ij,s}\right\vert \notag\\
   \!\!&\!\!\leq\!\!&\!\! \left\vert (\beta -\beta
    _{j,t_{1,u}})^{\mbox{\tiny{T}}}A_{ij,t_{1,u}}-(\gamma -\beta
    _{j,s})^{\mbox{\tiny{T}}}A_{ij,s}\right\vert +\left\vert \beta
    _{j,t_{1,u}}^{\mbox{\tiny{T}}}A_{ij,t_{1,u}}-
\beta_{j,s}^{\mbox{\tiny{T}}}A_{ij,s}\right\vert
\leq C\ell_{n}. \quad \label{res2}
\end{eqnarray}%
Analogous to the proof of Lemma \ref{lemmaA2}, from \eqref{res1} and
\eqref{res2}, it follows that
\begin{eqnarray*}
&&\left\vert \left[ \Delta_{ij}(s,\gamma)
 -\mathbb{E}_{j}\Delta_{ij}(s,\gamma) \right] -\left[
 \Delta_{ij}(t_{1,u},\beta) -\mathbb{E}_{j}\Delta_{ij}(t_{1,u},\beta) %
 \right] \right\vert \\
  &\leq &C\,\mathbb{I}_{\left( \left\vert X_{i}-\widetilde{X}_{j}\right\vert
 \leq h\right) }\left\{ \frac{\ell _{n}}{h}\mathbb{I}_{\left( \left\vert
 \varepsilon _{i}+r_{ij,t_{1,u}}\right\vert \leq \delta \right)
 }+\mathbb{I}_{\left( \left\vert \varepsilon
  _{i}+r_{ij,t_{1,u}}+P_{ij,t_{1,u}}\right\vert \leq C\ell_{n}\right)
  }+\mathbb{I}_{\left( \left\vert \varepsilon
  _{i}+r_{ij,t_{1,u}}\right\vert \leq C\ell_{n}\right) }\right\} \\
   &&\,+\,C\,\mathbb{I}_{\left( h-\ell_{n}\leq \left\vert
   X_{i}-\widetilde{X}_{j}\right\vert \leq h+\ell _{n}\right)
  }\mathbb{I}_{\left( \left\vert \varepsilon
 _{i}+r_{ij,t_{1,u}}\right\vert \leq C\left( \ell_{n}+\delta \right)
   \right) }.
  \end{eqnarray*}%
  Then, through using Lemma \ref{lemmaA1} and (\ref{Bernstein3}), we know that
  the following four probabilities
  \begin{equation*}
  \mathbb{P}\left( \left\vert \sum_{i}\frac{\ell_{n}}{h}\mathbb{I}_{\left(
  \left\vert X_{i}-\widetilde{X}_{j}\right\vert \leq h\right)
   }\mathbb{I}_{\left( \left\vert \varepsilon
  _{i}+r_{ij,t_{1,u}}\right\vert \leq \delta \right) }\right\vert
  >\frac{\delta_{n}}{8}\right) ,
  \end{equation*}%
  \begin{equation*}
  \mathbb{P}\left( \left\vert \sum_{i}\mathbb{I}_{\left( \left\vert X_{i}-\widetilde{X}%
                _{j}\right\vert \leq h\right) }\mathbb{I}_{\left( \left\vert \varepsilon _{i}+r_{ij,t_{1,u}}+P_{ij,t_{1,u}}\right\vert \leq C\ell_{n}\right)
        }\right\vert >\frac{\delta_{n}}{8}\right) ,
        \end{equation*}%
        \begin{equation*}
        \mathbb{P}\left( \left\vert \sum_{i}\mathbb{I}_{\left( \left\vert X_{i}-\widetilde{X}%
                _{j_{n}}\right\vert \leq h\right) }\mathbb{I}_{\left( \left\vert
                \varepsilon
                _{i}+r_{ij,t_{1,u}}\right\vert \leq C\ell _{n}\right) }\right\vert >\frac{%
                \delta_{n}}{8}\right)
        \end{equation*}%
        and%
        \begin{equation*}
        \mathbb{P}\left( \left\vert \sum_{i}\mathbb{I}_{\left( h-\ell_{n}\leq
                \left\vert X_{i}-\widetilde{X}_{j_{n}}^{u}\right\vert \leq h+\ell
                _{n}\right) }\mathbb{I}_{\left( \left\vert \varepsilon
                _{i}+r_{ij,t_{1,u}}\right\vert \leq
                C\left( \ell_{n}+\delta \right) \right) }\right\vert >\frac{\delta_{n}}{8}%
        \right)
        \end{equation*}%
        are bounded by the RHS of \eqref{in2}, where the
        relationships
        \begin{equation*}
        \mathbb{E}_{j}\mathbb{I}_{\left( \left\vert X_{i}-\widetilde{X}_{j}\right\vert
                \leq h\right) }\mathbb{I}_{\left( \left\vert \varepsilon
                _{i}+r_{ij,t_{1,u}}+P_{ij,t_{1,u}}\right\vert \leq C\ell_{n}\right)
        }\leq Ch^{d}\ell_{n}
        \end{equation*}%
        and%
        \begin{equation*}
        \mathbb{E}_{j}\left(\mathbb{I}_{\left( h-\ell_{n}\leq \left\vert X_{i}-\widetilde{X}%
                _{j}\right\vert \leq h+\ell _{n}\right) }\mathbb{I}_{\left( \left\vert
                \varepsilon _{i}+r_{ij,t}\right\vert \leq C\left( \ell_{n}+\delta
                \right) \right) }\right) \leq C\left( \ell_{n}+\delta \right) \ell
        _{n}h^{d-1}
        \end{equation*}%
        are used. Also, the inequality (\ref{in3}) follows from the
        condition on $r$ and condition (B3). Similar as the proof of Lemma \ref{lemmaB3},
        (\ref{Bahadur4}) can be inferred from the Borel-Cantelli lemma.
\end{proof}
\subsection*{Proof of Corollary 5.1 (Sketch)}  %
\begin{proof}
From (5.3), we have \vspace{-0.3cm}
\begin{align}
&(1-\alpha )-F_{n}\big(G(v)|v\big) \notag \\
&=\frac{\sum_{i=1}^{n}\left\{ \left[ K_{G,i,v}\big ( \left(
1-\alpha
\right) -\mathbb{I}\left( Y_{i}\leq G\left( v\right) \right) \big) \right] -
\mathbb{E}\left[ K_{G,i,v}\big ( \left( 1-\alpha \right)
-\mathbb{I}(
Y_{i}\leq G\left( v\right) ) \big ) \right] \right\} }{%
\sum_{i=1}^{n}K_{G,i}}  \notag \\
&-\frac{n\left(\mathbb{E}\left[ K_{G,i,v}\mathbb{I}\big(Y_{i}\leq
G(v)\big) \right] -\mathbb{E}\left[
K_{G,i,v}\mathbb{I}(\varepsilon_{i}\leq 0) \right] \right)
}{\sum_{i=1}^{n}K_{G,i}}
\notag \\
&=I_{1}+I_{2}.  \label{a6}
\end{align}
It can be proved that
\begin{equation}
\frac{\sum_{i=1}^{n}\!\left\{\left[ K_{G,i,v}\big (\left( 1-\alpha
\right)
-\mathbb{I}\left( Y_{i}\leq G(v) \right) \big ) \right]\! -
\mathbb{E}\!\left[ K_{G,i,v}\big (\left( 1-\alpha \right)\! -
\mathbb{I}\left(
Y_{i}\leq G\left( v\right) \right) \big ) \right] \right\}
}{\sqrt{\alpha \left(
1-\alpha \right) f_{q_{0}}(v) nh_{G}}}\overset{d}{\rightarrow }%
\mathcal{N}(0,1).  \label{a7}
\end{equation}%
Moreover, it can be proved that%
\begin{align}
\sum_{i=1}^{n}K_{G,i,v}=n\mathbb{E}K_{G,i,v}+ O\left( 
\frac{\sqrt{\log n}}{\sqrt{nh_{G}} }\right)
=nh_{G}f_{q_{0}}(v) \big(
1+O(h_{G}^{2}) \big ) +O\left( \frac{\sqrt{\log
n}}{\sqrt{nh_{G}}}\right).  \label{a8} 
\end{align}
By variable substitution, Taylor's expansion and $\int sK_{G}\left(
s\right){\rm d}s=0$, it follows that
\begin{equation}
\frac{\mathbb{E}\left[ K_{G,i,v}\mathbb{I}\big (Y_{i}\leq G(
v)
\big) \right] -\left( 1-\alpha \right) \mathbb{E}K_{G,i,v}}{%
h_{G}f_{q_{0}}(v) }=a(v) h_{G}^{2}+o\left(
h_{G}^{2}\right).  \label{a10}
\end{equation}%
Then, Corollary 5.1 can be inferred from part \textit{i)} of Theorem 5.1, \eqref{a6}, \eqref{a7}, \eqref{a8}, \eqref{a10} and the fact that
$\Big (nh^{1+\frac{1+\epsilon}{r}}\Big)^{-\frac{1}{2}}=o(h^{2}_{G})$.
\end{proof}

\end{appendices}

\end{document}